\documentclass[onecolumn]{article}
\usepackage[utf8]{inputenc}
\usepackage[english]{babel}
\usepackage{float}
\usepackage{amsmath}
\usepackage{geometry}
\usepackage{algorithm}
\usepackage{algpseudocode}
\usepackage{subcaption}
\usepackage{wrapfig}
\usepackage{amssymb}
\usepackage{amsthm}
\usepackage{aliascnt}
\usepackage[breaklinks=true]{hyperref}
\usepackage{breakcites}
\usepackage{hyperref}
\usepackage{graphicx}
\usepackage{subcaption}
\usepackage{tikz}
\usepackage{mathtools}
\usepackage{lmodern}
\usepackage{url}
\usepackage{hyperref}
\usepackage[nameinlink]{cleveref}
\usepackage{algorithm}
\usepackage{appendix} 
\usepackage{algpseudocode}
\usepackage{multicol}
\usepackage[firstinits=true,maxbibnames=99,doi=false,url=false,eprint=true,
backend=biber,style=numeric]{biblatex}
\usepackage[export]{adjustbox}
\usepackage{xr-hyper}
\usetikzlibrary{arrows}
\usetikzlibrary{decorations.markings}
\graphicspath{ {images/} }
\usepackage{booktabs}

\usepackage{enumerate}
\hypersetup{
  colorlinks = true,
  allcolors = green!50!black,
  urlcolor = red!50!black,
}

\theoremstyle{theorem}
    \newtheorem{theorem}{Theorem}
    \numberwithin{theorem}{section}
	\crefname{theorem}{Theorem}{Theorems}

\theoremstyle{definition}
    \newaliascnt{definition}{theorem}
    
    \aliascntresetthe{definition}
	\crefname{definition}{Definition}{Definitions}

\theoremstyle{corollary}
    \newaliascnt{corollary}{theorem}
    
    \aliascntresetthe{corollary}
	\crefname{corollary}{Corollary}{Corollarys}
    
\theoremstyle{proposition}
    \newaliascnt{proposition}{theorem}
    \newtheorem{proposition}[proposition]{Proposition}
    \aliascntresetthe{proposition}
	\crefname{proposition}{Proposition}{Propositions}
    
\theoremstyle{ex}
    \newaliascnt{ex}{theorem}
    \newtheorem{ex}[ex]{Example}
    \aliascntresetthe{ex}
	\crefname{ex}{Example}{Examples}
    
\theoremstyle{lemma}
    \newaliascnt{lemma}{theorem}
    \newtheorem{lemma}[lemma]{Lemma}
    \aliascntresetthe{lemma}
	\crefname{lemma}{Lemma}{Lemmas}
    
\theoremstyle{defprop}
    \newaliascnt{defprop}{theorem}
    \newtheorem{defprop}[defprop]{Definition and Proposition}
    \aliascntresetthe{defprop}
	\crefname{defprop}{Definition and Proposition}{Definition and Propositions}

\theoremstyle{remark}
    \newaliascnt{remark}{theorem}
    \newtheorem{remark}[remark]{Remark}
    \aliascntresetthe{remark}
	\crefname{remark}{Remark}{Remarks}
    
\theoremstyle{assumption}
    \newaliascnt{assumption}{theorem}
    \newtheorem{assumption}[assumption]{Assumption}
    \aliascntresetthe{assumption}
	\crefname{assumption}{Assumption}{Assumptions}
  
\crefname{figure}{Figure}{Figure}
\crefname{table}{Table}{Table}

\numberwithin{figure}{section}
\numberwithin{table}{section}
\numberwithin{equation}{section}
\numberwithin{theorem}{section}
\numberwithin{algorithm}{section}

\makeatletter
 \let\c@algorithm\c@theorem
\makeatother
 


\newcommand\keywordsname{Key words}
\newcommand\keywordname{Key word}
\newcommand\AMSname{AMS subject classifications}
\newcommand\AMname{AMS subject classification}

\newenvironment{@abssec}[1]{%
     \if@twocolumn
       \section*{#1}%
     \else
       \vspace{.05in}\footnotesize
       \parindent .2in
         {\upshape\bfseries #1. }\ignorespaces 
     \fi}
     {\if@twocolumn\else\par\vspace{.1in}\fi}
     
\newenvironment{keywords}{\begin{@abssec}{\keywordsname}}{\end{@abssec}}

\newenvironment{AMS}{\begin{@abssec}{\AMSname}}{\end{@abssec}}

\author{Martin Holler \and Andreas Habring\thanks{Institute of Mathematics and Scientific Computing, University of Graz. MH: \href{mailto:martin.holler@uni-graz.at}{martin.holler@uni-graz.at}, \url{https://imsc.uni-graz.at/hollerm}, AH: \href{mailto:andreas.habring@uni-graz.at}{andreas.habring@uni-graz.at}.}}

\usepackage{amsopn}

\usepackage{filecontents}
\title{A Generative Variational Model for Inverse Problems in Imaging\thanks{The authors acknowledge funding by the Austrian Research Promotion Agency (FFG) (Project number 881561). MH further is a member NAWI Graz (\url{https://www.nawigraz.at}) and BioTechMed Graz (\url{https://biotechmedgraz.at}).}}

\pgfdeclarelayer{bg}    %
\pgfsetlayers{bg,main}

\DeclareMathAlphabet{\mymathbb}{U}{BOONDOX-ds}{m}{n}

\ifpdf
  \DeclareGraphicsExtensions{.eps,.pdf,.png,.jpg}
\else
  \DeclareGraphicsExtensions{.eps}
\fi

\newcommand\norm[1]{\left\lVert#1\right\rVert}

\newcommand{\R}{\mathbb{R}}
\newcommand{\N}{\mathbb{N}}

\newcommand{\Dc}{\mathcal{D}}
\newcommand{\Rc}{\mathcal{R}}
\newcommand{\Gc}{\mathcal{G}}
\newcommand{\Mc}{\mathcal{M}}
\newcommand{\Jc}{\mathcal{J}}
\newcommand{\Ic}{\mathcal{I}}
\newcommand{\Lc}{\mathcal{L}}
\newcommand{\Ec}{\mathcal{E}}

\newcommand{\conv}{\ast}
\newcommand{\st}{\,|\,}

\DeclareMathOperator*{\TV}{TV}
\DeclareMathOperator*{\BV}{BV}

\DeclareMathOperator*{\KL}{KL}

\usepackage{amsopn}

\ifpdf
\hypersetup{
  pdftitle={A Generative Variational Model for Inverse Problems in Imaging},
  pdfauthor={A. Habring and M. Holler}
}\fi

\begin{document}
 
\maketitle

\begin{abstract}
This paper is concerned with the development, analysis and numerical realization of a novel variational model for the regularization of inverse problems in imaging. The proposed model is inspired by the architecture of generative convolutional neural networks; it aims to generate the unknown from variables in a latent space via multi-layer convolutions and non-linear penalties, and penalizes an associated cost. In contrast to conventional neural-network-based approaches, however, the convolution kernels are learned directly from the measured data such that no training is required.
 
The present work provides a mathematical analysis of the proposed model in a function space setting, including proofs for regularity and existence/stability of solutions, and convergence for vanishing noise. Moreover, in a discretized setting, a numerical algorithm for solving various types of inverse problems with the proposed model is derived. Numerical results are provided for applications in inpainting, denoising, deblurring under noise, super-resolution and JPEG decompression with multiple test images.
\end{abstract}

\begin{keywords}
  Inverse problems, data science, mathematical imaging, generative neural networks, machine learning, deep image prior, convolutional sparse coding
\end{keywords}

\begin{AMS}
49J27 $\cdot$ 65J20 $\cdot$ 68T07 $\cdot$ 94A08
\end{AMS}

\section{Introduction}
 
Inverse problems require the inversion of a forward model $A:X \rightarrow Y$, that describes the relation of an unknown $u \in X$ to some given (possibly noisy) data $y \in Y$, i.e., they require the solution of $A(u) \approx y$ for $u \in X$.
Characteristic about inverse problems is that this inversion is ill-posed, meaning that a classical inverse of $A$ is not well-defined and/or, even more severely, that a direct inversion of $A$ is unstable in the sense that small deviations (such as noise) in the given data can lead to large deviations in a corresponding solution.
 
Variational regularization is a well-established method for solving inverse problems. It defines an approximate solution of an inverse problem as solution of a minimization problem of the form
\begin{equation} \label{eq:var_model_example_intro}
 \min_{u} \Dc_{y}(A(u)) + \lambda \Rc (u),
\end{equation}
where $\Dc_{y}:Y \rightarrow [0,\infty]$ is discrepancy term measuring the distance of $A(u)$ to the given data, $\Rc:X \rightarrow [0,\infty]$ is a regularization functional and $\lambda>0$ is a (regularization) parameter.
Advantages of such an approach are as follows:
\begin{enumerate}[i)]
\item With a suitable choice of $\Rc$, the mapping $y \mapsto [u \text{ solving \eqref{eq:var_model_example_intro}} ]$ can be shown to be continuous (in an appropriate sense), meaning that variational methods allow to obtain a continuous mapping from given data to a corresponding solution. This is most evident when analyzing inverse problems in infinite dimensional spaces, as only in this case, there is a clear distinction between continuous and non-continuous solution methods even in a linear setting.
\item The functional $\Rc$ allows to incorporate any prior knowledge that one might have on the unknown $u$. This is in particular important in imaging, where the unknown corresponds to image data that is typically highly structured. Exploiting such structure in the unknown via suitable choices of $\Rc$ has proven to be highly beneficial for the quality of the obtained reconstruction and can be seen as one of the main reasons for the success of variational methods for inverse problems.
\end{enumerate}
In the past decades, a main focus of research on variational methods for inverse problems in imaging and beyond was on sparsity-based modeling of the unknown, either via sparsity in some transform domain \cite{mallat2009wavelettour_mh} or sparsity of (higher-order) derivatives \cite{rudin1992tv_mh,holler20ip_review_mh}. In this context, significant progress has been made in terms of both practical performance and analytic understanding of variational methods, with examples for the latter being, for instance, recovery guarantees, a description of the structure of solutions, or convergence results for vanishing noise.
 
More recently, (machine-)learning-based approaches have become increasingly popular also in inverse problems, mainly for their practical performance for instance in image processing \cite{arridge19_inverse_learning_review_mh}. The latter is particularly striking for more complex image data containing repeating patterns or texture, a class of image data, where classical variational models still show limitations and are outperformed also by approaches such as dictionary- or patch-based methods \cite{Lebrun12denoising_review_mh}. 
So far, however, the level of analytic understanding of machine learning methods, in particular in view of an analysis for inverse problems in function space, is far from the one of variational methods. Furthermore, the requirement of large sets of training data is a limiting factor for the application of machine learning methods to inverse problems, where the availability of training data is often very limited. 
 
The present work is a step towards overcoming these shortcomings; it is concerned with the development, analysis, and numerical realization of a novel variational model for inverse problems in imaging. The proposed approach uses techniques from generative models in machine learning to overcome a limited performance of existing variational methods in particular for highly structured image data such as texture, see \cref{fig_inpainting_intro} for an inpainting-example. In contrast to existing works on conventional machine learning approaches, however, our approach does not rely on training data, but enforces regularity purely by the architecture of the underlying network itself. Further, our model is amenable to analysis in function space, such that analytic guarantees on well-posedness, stability, and regularity of solutions for general inverse problems can be ensured independently of discretizations.
 
Analytically, our work can be seen as a multi-layer extension of \cite{Chambolle2020}, where our numerical experiments confirm a significantly improved practical performance resulting from the introduction of a multi-layer structure. Conceptually, our approach is also closely related to the recent works on \emph{deep image priors} \cite{deep_image_prior}, and we refer to \Cref{sec:related_works} for a detailed discussion of relevant works in that context.

\subsection{The Proposed Approach}
 
\begin{figure}
\centering
\includegraphics[width = 3cm]{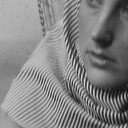}%
\includegraphics[width = 3cm]{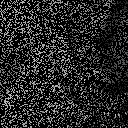}%
\includegraphics[width = 3cm]{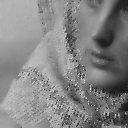}%
\includegraphics[width = 3cm]{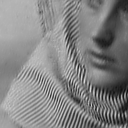}
\caption{Inpainting from 30\% known pixels. Left to right: ground truth, corrupted, reconstruction with total generalized variation regularization \cite{bredies2010tgv}, reconstruction with the proposed method.}
\label{fig_inpainting_intro}
\end{figure}
 
We consider a class of methods where a low dimensional variable in a so-called latent space $Z$ is mapped to an image $u \in X$ via a neural network. Trained appropriately, these networks can be regarded as parametrizations of certain classes of images. Due to their ability to generate images of a given class/distribution we refer to these networks as \emph{generative}. In view of inverse problems, a direct application of a pre-trained generative neural network, denoted by $g_\theta:Z \rightarrow X$, where $\theta$ summarizes all pre-trained parameters, would be to solve
\[ \min_{z} \Dc_{y} (A(g_\theta (z)) \]
for reconstruction, where again $A$ defines the forward model and $\Dc_{y}$ is a data-discrepancy term.
 
In order to avoid training, works such as Ulyanov et.al \cite{deep_image_prior} instead learn also the parameters $\theta$ directly from the measured data by solving
\[ \min_{z,\theta} \Dc_{y} (A(g_\theta (z)) .\]
(Note that the original work \cite{deep_image_prior} leaves $z$ fixed in practice). Surprisingly, this yields rather impressive results for different inverse problems in imaging, suggesting that already the architecture of the generator comprises a good prior for image data. 
 
A problem with this approach, however, is that a generator of sufficient size may be able to produce almost any given image, even random noise \cite{veen2018compressed}, meaning that the above model has no regularizing effect at all. To account for that, \cite{deep_image_prior} relies on early stopping of the numerical solution algorithm (as well as different heuristics), arguing that the generator usually fits natural images quite fast, but needs more time to fit noise.
 
Noting the relation of such approaches to convolutional-sparse coding methods \cite{Chambolle2020,papyan2017convolutional}, our approach realizes a similar idea, but within the framework of variational methods in function space, which are amenable to mathematical analysis. As regularization, we define a generative prior $\Gc:X \rightarrow [0,\infty]$ that generates the unknown $u \in X$ from multi-layer convolutions of a variable in latent space in a way, that is optimal with respect to an associated cost. In detail, we define $\Gc$ as
\begin{equation}\label{eq_texture_prior_intro}
  \mathcal{G}(u) = \inf\limits_{\substack{\mu = (\mu^l)_{l=1}^L, \\ \theta=(\theta^l)_{l=1}^L}}  \norm{\mu}_\mathcal{M} + \sum\limits_{k,l} \Jc_{l-1} (\mu^{l-1}_{k} - \sum\limits_{n=1}^{N_l} \mu^l_{n}*\theta^l_{n,k})
  \quad \text{subject to }
  \begin{cases}
 u = \sum\limits_{n=1}^{N_1} \mu^1_{n}*\theta^1_{n,1},\\
  \norm{\theta^l_{n,k}}_2 \leq 1 ,  \quad \text{for all $n,k,l$}, \\ %
  Z\theta = 0
  \end{cases}
\end{equation}
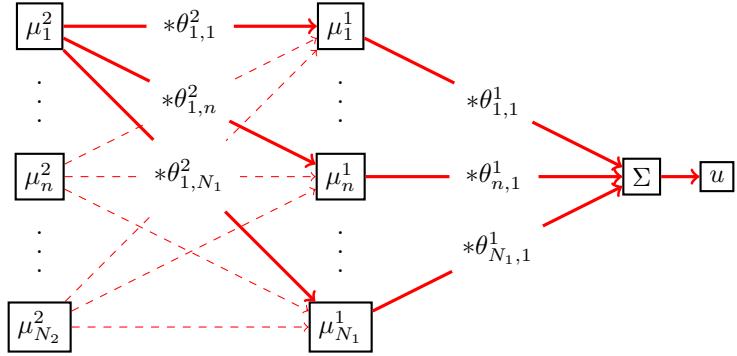
\begin{figure}
\centering
\begin{tikzpicture}
\begin{scope}[every node/.style={rectangle,thick,draw}]
    \node (v) at (15,2) {$u$};
    \node (sum) at (14,2) {$\Sigma$};
    \node (11) at (10,4) {$\mu^1_1$};
    \node (12) at (10,2) {$\mu^1_n$};
    \node (13) at (10,0) {$\mu^1_{N_1}$};
    \node (21) at (6,4) {$\mu^2_1$};
    \node (22) at (6,2) {$\mu^2_{n}$};
    \node (23) at (6,0) {$\mu^2_{N_2}$};
\end{scope}
    \node  at (10,3.25) {.};
    \node  at (10,3) {.};
    \node  at (10,2.75) {.};
    
    \node  at (10,1.25) {.};
    \node  at (10,1) {.};
    \node  at (10,0.75) {.};
    
    \node  at (6,3.25) {.};
    \node  at (6,3) {.};
    \node  at (6,2.75) {.};
    
    \node  at (6,1.25) {.};
    \node  at (6,1) {.};
    \node  at (6,0.75) {.};
 
\begin{scope}[->,
              every node/.style={fill=white,circle},
              every edge/.style={draw=red,very thick}]
              
    \path [->] (21) edge node {$*\theta^2_{1,1}$} (11);
    \path [->] (21) edge node {$*\theta^2_{1,n}$} (12);
    \path [->] (21) edge node {$*\theta^2_{1,N_1}$} (13);
    \begin{pgfonlayer}{bg}    %
        \path [->] (22) edge[dashed,thin] (11); %
        \path [->] (22) edge[dashed,thin] (12); %
        \path [->] (22) edge[dashed,thin] (13); %
        \path [->] (23) edge[dashed,thin] (11); %
        \path [->] (23) edge[dashed,thin] (12); %
        \path [->] (23) edge[dashed,thin] (13); %
    \end{pgfonlayer}
 
    \path [->](11) edge node {$*\theta^1_{1,1}$} (sum);
    \path [->](12) edge node {$*\theta^1_{n,1}$} (sum);
    \path [->](13) edge node {$*\theta^1_{N_1,1}$} (sum);
    \path [->](sum) edge (v);
\end{scope}
\end{tikzpicture}
\caption{Sketch of the generative network with $L=2$ layers.}
\label{fig_sketch_network}
\end{figure}
see \cref{fig_sketch_network} for a visualization. Here, the latent variables $\mu^l$ are modeled as measures, the norm $\|\cdot \|_\Mc$ denotes the Radon norm (the measure-space counterpart of the $L^1$ norm) and the $\theta^l$ denote filter kernels (modeled as $L^2$ functions), which are normalized by the constraint $\|\theta_{n,k}^l \|_2 \leq 1$. The functionals $\Jc_l$ can for instance be indicator functions of $\{0\}$, i.e., $\Jc_l(x) = 0$ if $x=0$ and $\Jc_l(x) = \infty$ else,  to ensure that $u$ is indeed generated from a sequence of multi-layer convolutions of the latent variable $\mu^L$ in the deepest layer. However, we allow for more general $\Jc_l$, including also smooth relaxations, since our numerical implementation relies on differentiability of the $\Jc_l$. Non-linearities in the multi-layer convolution are replaced by non-linear penalties on the latent variables of the intermediate layers. The last constraint $Z\theta = 0$, with $Z$ any linear, bounded operator, allows as to add some additional constraints on the filter kernels, e.g., to have zero mean. For a detailed introduction of the generative prior $\Gc$, we refer to \Cref{sec_texture_prior}.

\begin{figure}
\centering
\includegraphics[width = 3cm]{images_gen_reg_inpainting_barbara_crop_original.png}%
\includegraphics[width = 3cm]{images_gen_reg_inpainting_barbara_crop_corrupted.png}%
\includegraphics[width = 3cm]{images_gen_reg_inpainting_barbara_crop_recon.png}%
\includegraphics[width = 3cm]{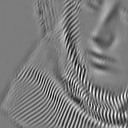}%
\includegraphics[width = 3cm]{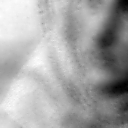}
\caption{Inpainting from 30\% known pixels with the proposed method. Left to right: ground truth image, corrupted image, reconstruction $u$, generative part $v$, remaining part $u-v$.}
\label{fig_decomp_intro}
\end{figure}
 
In this work $\Gc$ as in \eqref{eq_texture_prior_intro} is analyzed and applied as regularization for general inverse problems. Our motivation for introducing $\Gc$ is to capture image structure, such as repeating patterns, that is not captured well by classical variational methods such as Tikhonov or total-variation regularization. Moreover, the multi-layer approach inspired by the architecture of convolutional neural networks shows to be advantageous in practice compared to a single layer convolutional model (see \Cref{experiments}).

On the other hand, many existing variational methods are already very well suited for simpler image structures such as piecewise smooth regions.
Furthermore, our analysis in \Cref{chapter_theory} shows that generative priors such as $\Gc$ will always produce continuous image structures as output. Accounting for these facts,
 we allow for the infimal convolution of $\Gc$ with a second regularization functional (such as the total variation functional \cite{rudin1992tv_mh}) in our overall approach to the regularization of inverse problems. This results in a variational energy of the form
\begin{equation}\label{eq_problem_thesis}
\begin{gathered}
\min_{u,v} \lambda \mathcal{D}_{y}(A(u)) + s_\Rc (\nu) \mathcal{R}(u-v) + s_\Gc (\nu) \mathcal{G}(v),
\end{gathered}
\end{equation}
with $\lambda\in (0,\infty)$ a regularization parameter and $s_\Rc (\nu) = \frac{\nu}{\min (\nu, 1-\nu)}$ and $s_\Gc (\nu) = \frac{1-\nu}{\min (\nu, 1-\nu)}$ for $\nu \in (0,1)$ being parametrizations of a $\nu$-dependent balancing between $\Rc$ and $\Gc$.
 In \eqref{eq_problem_thesis} $\Rc$ can either be the indicator function of $\{0\}$, i.e., $\Rc(u) = 0$ if $u=0$ and $\infty $ else, resulting in a model that only uses $\Gc$ for regularization. Or, alternatively, $\Rc$ can be chosen to be any other suitable regularization functional, such as the total variation (TV) functional \cite{rudin1992tv_mh}, and the resulting infimal-convolution of the two functionals yields a decomposition $u = v + (u-v)$ of the solution $u$ of \eqref{eq_problem_thesis} as shown for example in \cref{fig_decomp_intro} for an inpainting experiment.
 
\paragraph{\textbf{Contributions}} By developing an energy-based regularization approach for general inverse problems in function space, our work provides the following contributions.
\begin{itemize}
\item \emph{Well-posedness:} Analytically, we prove properties of $\Gc$ such as coercivity and lower semicontinuity in function space, that yield well-posedness (including stability, but not uniqueness) for the application of our model to any linear inverse problem, without any assumptions, and to non-linear inverse problems under weak standard assumptions \cite{Hofmann07_nonlinear_tikhonov_banach_mh}. Specifically, we provide the following results:
\begin{itemize}
\item The infimum over $(\mu,\theta)$ in the definition of $\Gc$ as in \eqref{eq_texture_prior_intro} is attained, and $\Gc$ is coercive and weakly lower semi-continuous, see \cref{lem_Gv_attained,lem_G_coercive,lem_G_lsc}.
\item Problem \eqref{eq_problem_thesis} admits a solution, see \cref{thm_existence}.
\item Solutions of \eqref{eq_problem_thesis} are weakly, subsequentially stable up to shifts of $u$ in a subspace where both $A$ and $\Rc$ are invariant (\cref{thm_stability}), and for vanishing noise and appropriately chosen regularization parameters, weak, subsequential convergence of the solutions to solutions of an appropriate limit problem can be guaranteed (\cref{SM_thm_convergence}).
\end{itemize}
\item \emph{Regularity:} 
We provide a regularity result for the generative prior in function space. That is, we show that our convolutional network inspired architecture generates only continuous functions as outputs whenever at least two layers are used, i.e., $\Gc(u)<\infty$ only if $u$ is continuous, see \cref{prop_regularity_network}.
This is an important observation as it has several crucial implications: i) Our prior is well-suited for removing noise-like artifacts, as it is not capable of fitting highly discontinuous noise. ii) It is very reasonable to combine our generative prior, as well as other related generative neural-network-based priors, with a second term that allows for discontinuous structures. In our work, this is achieved via the infimal convolution with a TV-like functional $\Rc$.
iii) It is the penalization of latent variables via norms, as done in this work, that yields well-posedness and regularity of solutions. For other works that do not penalize latent variables, such as the deep image prior \cite{deep_image_prior}, similar results cannot be expected. This confirms the need for early stopping, as done in those works, in order to avoid the approximation of noise.

\item \emph{Generality:} Conceptually, the proposed approach is applicable to general inverse problems and we provide numerical results showing state-of-the-art performance for inpainting, denoising, deblurring under noise, super-resolution and JPEG decompression. 

\item \emph{Model-complexity:} While our approach still requires the solution of a non-convex optimization problem, for which global-optimality cannot be guaranteed, we still believe that our energy-based formulation and the fact that we do not require any heuristics is a big advantage in terms of simplicity and interpretability of our model compared to existing deep learning methods. In \cite{deep_image_prior} for instance, the authors propose to use skip connections in the network, a noise based regularization that is perturbing the network input with a random noise at each iteration of the fitting process, and also a method of reverting to previous iterates during the minimization. All of these features improve practical performance but are difficult to interpret or understand. We do not use any such heuristics in the implementation and all the theoretical results proven in the paper directly apply to the experiments.

In particular, we highlight that we are able to achieve results that are at least comparable to the ones of the deep image prior, while our model reduces the number of effective degrees of freedom, e.g., by a factor of $100$ (from $\approx 3  \times 10^6$  to $\approx 3 \times 10^4$) for inpainting with the \emph{Barbara} image as shown in \cref{fig_inpainting_intro} (or by a factor of 10, e.g., to $2 \times 10^5$, if one counts also auxiliary variables introduced in order to smooth hard-constraints in the numerical realization). In this context we note that also the work \cite{heckel2019deep} reduces model complexity for the deep image prior, an refer to  \Cref{sec:related_works} for a comparison to \cite{heckel2019deep}.

\item \emph{Foundations for further analysis:} By developing a generative prior in function space that delivers state-of-the-art numerical results, we enable a further analysis and mathematical development of such models with respect to, for instance, convex relaxations as done in \cite{Chambolle2020} for a single-layer version, or a further understanding of structural properties of solutions using results from \cite{Boyer19representer_mh,Carioni20_representer}.
\end{itemize}

\subsection{Related Works} \label{sec:related_works}

The deep image prior (DIP) \cite{deep_image_prior} has been a great incentive for research in the direction the present work is headed. While most works in that context are focused on experimental results, some also take a more mathematical point of view.
 
In \cite{Chambolle2020}, a single layer version of the proposed model, using only one convolutional layer, is introduced and analyzed as well as tested on images, showing promising results for images comprised of texture and piecewise smooth parts.
In \cite{heckel2019denoising}, the authors investigate how convolutional neural networks act as successful priors for image reconstruction solely via their architecture, as empirically shown by \cite{deep_image_prior}. The authors prove that fitting a CNN with fixed convolution kernels via early stopped gradient descent to a signal denoises it.
In \cite{heckel2019deep}, the authors propose the deep decoder, an underparametrized, untrained generative network for inverse problems in imaging. To be precise, the authors use a very simple, non-convolutional network (only consisting of pixel-wise linear combination of different channels, upsampling, non-linearities and normalization layers) as an image generator. The network architecture is chosen such that the number of parameters is less than the image dimensionality. It is shown theoretically and empirically, that the underparametrization acts as regularization by removing noise, without compromising the reconstruction error too much. These results are in line with our observations in the sense that we could also reduce the number of parameters and model complexity significantly compared to the original DIP.
In \cite{dittmer2019regularization}, a different point of view is introduced, showing that, under some conditions, a deep image prior can be interpreted as learning an optimal Tikhonov functional instead of training a neural network.
In \cite{jagatap2019algorithmic}, algorithmic guarantees for solving inverse problems using a deep image prior are provided.
 
In \cite{baguer2020computed,cascarano2020admm,obmann2020deep,
veen2018compressed,asim2020invertible,whang2020compressed}, neural networks are combined with additional regularization terms for solving inverse problems: In \cite{baguer2020computed}, the authors consider a problem of the form
\begin{equation}\label{eq_problem_baguer2020computed}
\begin{gathered}
\min_{\theta} \mathcal{D}_{y}(Ag_\theta(z)) + \lambda \mathcal{R}(g_\theta(z)),
\end{gathered}
\end{equation}
with a regularizing functional $\mathcal{R}$. Existence, stability, and convergence of solutions are obtained by imposing suitable assumptions. Moreover, also different learned approaches to obtain a better initialization for the proposed method are used. A main difference to our method is that we do not penalize $g_\theta(z)$, but $\theta$ and $z$ instead.
Also in \cite{cascarano2020admm} a problem of the form \eqref{eq_problem_baguer2020computed} with $\Dc_{y}(z)=\frac{1}{2}\|z-y\|^2$ and $\mathcal{R}$ a weighted total-variation is considered and solved using the Alternating Direction Method of Multipliers (ADMM) framework.
 
The problem formulation in \cite{obmann2020deep} is of the form
\begin{equation}\label{eq_problem_obmann2020deep}
\begin{gathered}
\min_{z} \|A\phi_\alpha (z)-y\|^2 + \alpha \norm{z}_{1,w},
\end{gathered}
\end{equation}
where $(\phi_\alpha)_{\alpha\geq 0}$ is a family of functions modeling (pre-trained) neural networks and $\norm{\; . \;}_{1,w}$ denotes a weighted 1-norm. The authors provide proofs of existence, stability, and convergence of solutions of the method in an abstract setting employing a list of assumptions on the involved functions and spaces.
 
In \cite{veen2018compressed}, a deep image prior is combined with total variation regularization and a regularization of the network parameters leading to
\begin{equation}\label{eq_problem_veen2018compressed}
\begin{gathered}
\min\limits_{\theta} \norm{Ag_\theta(z) - y}^2 + \lambda \TV (g_\theta(z)) + \mu \mathcal{R}(\theta),
\end{gathered}
\end{equation}
where $\mathcal{R}$ is a learned regularization functional. Contrary to the present work, here the $\TV$ is applied to the output of the neural network, which in our model shall in particular contain oscillatory features of the image. %
In \cite{asim2020invertible,whang2020compressed}, the authors use a pre-trained invertible generative network, $g_\theta:\mathbb{R}^n\rightarrow \mathbb{R}^n$ with $\theta$ learned, as a prior. In \cite{asim2020invertible}, the authors consider
\[\min_{z} \norm{Ag_\theta(z) - y}^2_2 + \lambda \| z \|^2_2\]
and in \cite{whang2020compressed} this ansatz is further generalized. Both works use a likelihood-based approach and discuss recovery guarantees.
 
A field that is also related to our work due to its generative modeling is sparse coding.
Sparse coding aims to find a representation of a vector $y\in\mathbb{R}^m$ in terms of a (typically pre-trained) dictionary $D\in\mathbb{R}^{m\times n}$ and a sparse coefficient vector $x\in\mathbb{R}^n$, such that $y=Dx$. In convolutional sparse coding, the dictionary matrix $D$ is further restricted to represent a convolution operator.
 
In \cite{sulam2018multilayer}, a multi-layer convolutional sparse coding method is discussed, i.e., instead of the single dictionary $D$, a composition of convolutional dictionaries $D_ND_{N-1}...D_1$ is used. 
The works \cite{zhang2016convolutional,papyan2017convolutional} present methods for cartoon texture decomposition via single-layer convolutional sparse coding combined with a TV penalty, resembling our combination of the generative prior $\Gc$ with a second regularization. In \cite{zhang2016convolutional}, the dictionary is pre-trained, but in \cite{papyan2017convolutional}, this results in a minimization problem similar to a single layer version of the proposed method.

All the before mentioned approaches using neural networks were generative in the sense that the network is used to generate the unknown. To the contrary, in \cite{lunz2019adversarial,li2019nett} the authors use pre-trained neural networks directly as regularization functionals. More precisely, a neural network $\psi_\theta: X \rightarrow \mathbb{R}$ with parameters $\theta$ is trained to map an image to a scalar penalty. Then, once the parameters are fixed, $\psi_\theta$ is applied for solving inverse problems via

\[ \min\limits_u \Dc_y(Au) + \lambda\psi_\theta(u). \]

In both works the authors also present theoretical analyses of the models relying on assumptions on the trained network. Main differences to our work are that we use the neural network as an image generator and that our network is untrained. Moreover, contrary to \cite{lunz2019adversarial,li2019nett}, we do not make any additional assumptions on our network for the analysis.
 
\subsection{Outline of the Paper}
 
In \Cref{chapter_theory}, we introduce and analyze the proposed method in an infinite dimensional setting proving, in particular, results on existence, stability, and regularity of solutions. In \Cref{experiments}, we provide numerical results obtained with our method and a comparison to the state of the art for different imaging applications.

\section{The Generative Model in Function Spaces}\label{chapter_theory}
In order to properly introduce the proposed model and prove theoretical results, we first need some preliminaries.
 
\subsection{Notation and Preliminaries}
In the following, $\Omega,\Sigma\subset \mathbb{R}^d$ are bounded domains, i.e., bounded, open, and simply connected sets in $\mathbb{R}^d$, with $0 \in \Sigma.$ Moreover, we denote $\Omega_\Sigma = \Omega-\Sigma = \left\{x- y\; \middle| \; x\in\Omega, \; y\in\Sigma \right\}$. For $q\in [1,\infty]$, we denote the Hölder conjugate exponent as $q'$, i.e., $q'\in [1,\infty]$ is, such that $\frac{1}{q}+\frac{1}{q'}=1$, using the convention that $\frac{1}{\infty}=0$. For $\omega\subset \mathbb{R}^d$, $f \in L^q(\omega)$ and $g \in L^{q'}(\omega)$, we use the notation $\langle f,g\rangle\coloneqq \int_\omega f(x)g(x)\;dx$ for the standard dual pairing and $\|f\|_q^q =\int_\omega |f(x)|^q\; dx$ for the $L^q$-norm. We define the zero extension of any function $f:\omega \rightarrow \R$ as
\begin{equation*}
\tilde{f}:\mathbb{R}^d \rightarrow \mathbb{R}, \quad\tilde{f}(x) =\begin{cases}
f(x)\quad &\text{if } x\in\omega,\\
0\quad &\text{else.}
\end{cases}
\end{equation*}
Further, we define $C_0(\omega)$ as the closure of $C_c(\omega)$ with respect to the uniform norm and the space $(\mathcal{M}(\omega),\|\; .\; \|_\mathcal{M})$ as the dual of $C_0(\omega)$. Note that $\Mc(\omega)$ coincides with the space of finite Radon measures and $\|\; .\; \|_\mathcal{M}$ with the total variation norm, whenever $\omega $ is a locally compact, separable metric space \cite[Theorem 1.54, Remark 1.57]{bvfunctions}, a property that holds for all domains considered in this work. For $\omega$ bounded, we will sometimes identify a function $f\in L^q(\omega)$ with an element of $\mathcal{M}(\omega)$ via $C_0(\omega)\ni \phi \mapsto \langle f,\phi\rangle$. For normed spaces $X$, $Y$, we denote the space of all linear, bounded operators from $X $ to $ Y$ as $\Lc(X,Y)$. Lastly, we denote the indicator function on a set $M$ as $\mathcal{I}_M$, i.e., $\mathcal{I}_M(x)=0$ if $x\in M$ and $\infty$ otherwise. We also need a notion of convolution.
 
\begin{defprop}\label{lem_conv_measures}\
Let $q\in (1,2]$, $g\in L^2(\Sigma)$ and $\mu\in\mathcal{M}(\Omega_\Sigma)$. Then there exists a unique function denoted by $\mu*g\in L^{q}(\Omega)$, such that for all $\phi\in C_c(\Omega)$,
\begin{equation}\label{eq_lem_conv_meas}
\begin{gathered}
\langle\mu*g, \phi\rangle = \int\limits_{\Omega_\Sigma} \int\limits_\Sigma \tilde{\phi}(x+y) g(x)  \; dx \; d\mu(y).
\end{gathered}
\end{equation}
In particular, the right-hand side is well-defined for all such $\phi$. Moreover, there exists $C>0$, such that $\norm{\mu*g}_q \leq C\norm{\mu}_\mathcal{M}\norm{g}_2$.
\end{defprop}
\begin{proof}
This is shown, for instance, in \cite[Section 3.1]{Chambolle2020}. 
\end{proof}
Note that, in case $g$ and $\mu$ are sufficiently regular, by Fubini's theorem, $\mu \conv g$ can be identified with the classical convolution given as
\[ 
(\mu*g)(x) = \int_{\Omega_\Sigma} g(x-y) d\mu(y).
\]
 
\subsection{Modeling and Analysis}\label{section_th_problem_formulation}
With the proposed method, we aim to solve inverse problems in imaging by (partially) generating the unknown image from an architecture inspired by generative convolutional neural networks. We also allow for an optional combination of this approach via infimal convolution with a second regularization term, that models parts of the unknown image that are not well-described by the generative approach. To be precise, let $q \in (1,2]$, $Y$ a Banach space, $y \in Y$ the given data and $A:L^q(\Omega) \rightarrow Y$ a forward operator, then we aim to find $u \in L^q(\Omega)$ with  $A(u)\approx y$ by solving
\begin{equation}\label{eq_cont_problem}
\tag{$P(y)$}
\min_{u,v\in L^q(\Omega)} \Ec_{y}(u,v),\text{ with } \Ec_{y}(u,v) \coloneqq \lambda\mathcal{D}_{y}(A(u)) + s_\Rc(\nu)\Rc(u-v)+s_\mathcal{G}(\nu)\mathcal{G}(v),
\end{equation}
where $\lambda\in (0,\infty)$, $\nu \in (0,1)$, $s_\Rc(\nu) = \frac{
\nu}{\min(\nu, 1-\nu)}$, $s_\mathcal{G}(\nu) = \frac{1-\nu}{\min(\nu, 1-\nu)}$ and
\begin{itemize}
\item $\mathcal{D}_{y}$ is a data fidelity term penalizing the discrepancy between $A(u)$ and the data $y$,
\item $\mathcal{G}$ is a \emph{generative prior} inspired by the architecture of generative convolutional neural networks, introduced in detail in \Cref{sec_texture_prior} below, and
\item $\Rc$ is any other regularization functional satisfying standard conditions as stated in \cref{ass_well_posedness} below.
We think of $\Rc$ as a functional that enforces piecewise smoothness such as $\Rc= \TV$, but it's choice is rather flexible and also $\Rc = \Ic_{\{0\}}$ is feasible, reducing our model to using only $\Gc$ for regularization.
\end{itemize}
In \eqref{eq_cont_problem}, $u$ is the image, $v$ as the \emph{generative part} of $u$, which shall in particular contain oscillatory features of the image, and $u-v$ contains all information not captured by the generative part $u$. The parameters $\nu$ and $\lambda$  balance the different regularization terms and the data fidelity, respectively, where the parametrizations $s_\Rc(\nu)$ and $s_\mathcal{G}(\nu)$ are introduced to achieve a $\nu$-dependent balancing between $\Rc$ and $\Gc$, that does not interfere with the balancing between data fidelity and regularization.

\subsubsection{The Generative Prior $\mathcal{G}$}\label{sec_texture_prior}

We model the generative part $v\in L^q(\Omega)$ in \eqref{eq_cont_problem} to be the output of a neural-network-inspired generative architecture as follows: Let $L$ be the number of hidden layers, $N_l$ be the number of latent variables in layer $l=1,\ldots,L$, and set $N_0 = 1$.
Then we define, for $l=1,\ldots,L$, via 
\[ E^l \subset \big \{ (n,k)  \st n \in \{1,\ldots,N_l\}, k \in \{1,\ldots,N_{l-1}\} \big \}
\]
a set of all connections between latent variables of layer $l$ and of layer $l-1$, see \cref{fig_sketch_network} for a visualization of a corresponding generative network. In order to avoid latent variables that do not influence the output, we assume that each latent variable of layer $l$ is connected to at least one latent variable of layer $l-1$, i.e., for each $n \in \{1,\ldots,N_l\}$ there exists $k \subset \{1,\ldots,N_{l-1}\}$ such that $(n,k) \subset E^l$. 
 
Given the domains $\Sigma, \Omega \subset \R^d$, we recursively define domains $\Omega_l$ via $\Omega^1=\Omega-\Sigma$ and $\Omega^{l+1}=\Omega^{l}-\Sigma$ for $l=1,\ldots,L-1$. Further, we denote
\begin{itemize}
\item $M^l \coloneqq \mathcal{M}(\Omega^l)^{N_l}$,
\item $M \coloneqq M^1\times M^2\times ... \times M^L$
\end{itemize}
to be the spaces of latent variables $\mu = (\mu^1,\ldots,\mu^L)\in M$ with $\mu^l = (\mu^l_1,\ldots,\mu^l_{N_l})\in M^l$, and
\begin{itemize}
\item $\Theta^l \coloneqq L^2(\Sigma)^{|E^l|}$,
\item $\Theta \coloneqq \Theta^1 \times \Theta^2 \times... \times \Theta^L$
\end{itemize}
to be the spaces of filter kernels $\theta = (\theta^1,\ldots,\theta^L)\in\Theta$ with $\theta^l = (\theta_{n,k}^l)_{(n,k)\in E^l} \in \Theta^l$.
 
Our generative prior will constrain the input $v$ to be written as a sum of multi-layer convolutions of the latent variable $\mu^L \in M^L$. To this aim, we define vectorial convolution operators $\conv_l$ via
\noindent\begin{minipage}{.5\linewidth}
\begin{align*}
\conv_1 : M^1 \times \Theta^1 & \rightarrow L^q({\Omega}) \\
(\mu,\theta) & \mapsto \sum\limits_{n=1}^{N_1} \mu_{n}*\theta_{n,1},
\end{align*}
\end{minipage}%
\begin{minipage}{.5\linewidth}
\begin{align*}
\conv_l : M^l \times \Theta^l & \rightarrow M^{l-1} \\
(\mu,\theta) & \mapsto \left(  \sum_{(n,k) \in E^{l}} \mu_n \conv \theta_{n,k} \right) _{k=1}^{N_{l-1}},
\end{align*}
\end{minipage}\\\\
for $l=2,\ldots,L$, where we note that, by \cref{lem_conv_measures}, the convolutions $\mu^l \conv_l \theta^l$ are actually contained in more regular $L^q$-spaces.
The functional $\mathcal{G}: L^q(\Omega)\rightarrow[0,\infty]$ is then defined as
\begin{equation}\label{eq_texture_prior}
\tag{GEN}
  \begin{gathered}
  \mathcal{G}(v) = \inf\limits_{\substack{\mu \in M, \\ \theta \in \Theta}}  \norm{\mu}_\mathcal{M} +\sum\limits_{l=2}^L  \Jc_{l-1}( \mu^{l-1} -\mu^l \conv_l \theta^l ) ,
  \end{gathered}\quad \text{s.t.} \quad 
  \begin{cases}
  v = \mu^1*_1\theta^1,\\
  \norm{\theta^l_{n,k}}_2 \leq 1 \quad \text{for all $n,k,l$,}\\
  Z \theta = 0,\\
  \end{cases}  
\end{equation}
where $\|\mu\|_\Mc = \sum_{l=1}^L \sum_{n=1}^{N_l} \|\mu^l_n\|_\Mc$ denotes a component-wise Radon norm, $Z \in \Lc(\Theta ,\R^{N_c})$ with $N_c \in \N$, and $\Gc(v) = \infty$ in case the constraints in \eqref{eq_texture_prior} cannot be met by any $\mu,\theta$.
Note that, while our interest is to enforce the constraint $\mu^{l-1} = \mu^l \conv_l \theta^l$ exactly, we allow also a relaxed version of this constraint via incorporating functionals
$\Jc_l: M^{l} \rightarrow [0,\infty]$, that are weak* lower semicontinuous and such that $\Jc_l(0) = 0$. This is a generalization, as the choice $\Jc_l = \Ic_{\{0\}}$ is feasible and allows to enforce hard constraints.
 
We can make the following observations regarding the generative prior $\Gc$: The latent variables $\mu^l$ are penalized with the $L^1$-type norm $\|\cdot \|_\Mc$. In our approach, this is the replacement for non-linearities in the network and aims to enforce sparsity of the $\mu^l$. A relation to non-linearities in neural networks is given indirectly via the soft-thresholding operator resulting from $\|\cdot \|_\Mc$, which acts on the $\mu^l$ within our numerical solution algorithm.
In order to ensure well-posedness, and in particular that \eqref{eq_texture_prior} admits a minimum, we also need to penalize the filter kernels. Indeed, note that due to the bilinearity of the convolution, it is possible to shift a non-zero scalar $\alpha$ between $\mu$ and $\theta$ without changing the outcome, that is $\mu^l*_l \theta^l = (\frac{1}{\alpha}\mu^l)*_l(\alpha \theta^l)$. This would, however, decrease the objective functional value for $\alpha>1$. Therefore, without a penalty on $\theta$, one could always let $\alpha \rightarrow \infty$ above and hence \eqref{eq_texture_prior} would not admit a minimum. We decided to use a norm constraint on the filter kernels instead of an additive term since again by the bilinearity of the convolution this is not limiting the model, but avoids ambiguities.
We allow for additional, linear constraints $Z\theta = 0$ on the filter kernels. In our applications later on, where $\Gc$ is combined with a total-variation-type functional, this will be used to enforce zero mean of the filter kernels of the last layer, i.e., $\int_\Sigma \theta^1 _{n,1} = 0$ for all $n=1,\ldots,N_1$. In the context of total-variation-type functionals this is natural. The total variation does not penalize constant translations, therefore we want them to be contained in $u-v$ and thus $v$ shall have zero mean. In the general model, however,
any type of (weakly continuous) constraints, in particular also no constraints ($Z=0$), are possible.
 
In order to gain more understanding of the regularizing properties of the proposed network architecture, we will in the following discuss regularity of the output $v$. To simplify notation, we denote for $\theta \in \Theta$ and $\mu \in M$,
\begin{equation}
  G(\mu,\theta) = \norm{\mu}_\mathcal{M} +\sum\limits_{l=2}^L  \Jc_{l-1}( \mu^{l-1} -\mu^l \conv_l \theta^l ).
\end{equation}
and, for $v\in L^q(\Omega)$, the feasible set as
\[F(v) \coloneqq \left\{ (\mu,\theta)\in M\times \Theta\; \middle| \; (\mu,\theta) \text{ satisfies the constraints in \eqref{eq_texture_prior}}\right\},\]
such that $\mathcal{G}(v)=\inf\limits_{(\mu,\theta)\in F(v)}G(\mu,\theta)$. 
 
\begin{proposition}\label{prop_regularity_network}
Let $\Gc:L^q(\Omega) \rightarrow [0,\infty]$ be given as in \eqref{eq_texture_prior} with $L\geq 2$ and assume additionally that $\Jc_l$ is such that $\Jc_l(\mu)<\infty$ implies $\mu\in L^2(\Omega^l)^{N_l}$. Then, if $\Gc(v)<\infty$, it holds true that $v\in C(\overline{\Omega})$.
\end{proposition}
\begin{proof}
Let $\Gc(v)<\infty$. This implies the existence of $(\mu,\theta)\in F(v)$ such that $G(\mu,\theta)<\infty$. From \cref{lem_conv_measures}, it follows that $\mu^{L}*_L\theta^L\in L^2(\Omega^{L-1})^{N_{L-1}}$. Together with the assumption on $\Jc_l$, this implies $\mu^{L-1}\in L^2(\Omega^{L-1})^{N_{L-1}}$. Repeating this argument, in case $L>2$, yields $\mu^1\in L^2(\Omega^1)^{N_1}$. Finally, $v=\mu^1*_1\theta^1\in C(\overline{\Omega})$ as it is a sum of convolutions of two $L^2$ functions, see, for instance, \cite[Theorem 3.14]{bredies2018mathematical}.
\end{proof}
Note that this result requires a network depth of at least $L=2$, the single-layer variant with $L=1$ can also produce discontinuous outputs in general. Also note that the assumption on $\Jc_l$ in \cref{prop_regularity_network} is in particular true for $\Jc_l = \Ic_{\{0\}}$ as well as for the relaxed version of $\Ic_{\{ 0\}}$, which is introduced in \cref{ex_data_reg} and used in the applications in \Cref{experiments}.
 
\begin{remark}[Insights from function space] In our opinion, an important reason for the success of using deeper generative networks in imaging (as opposed to single-layer models) %
lies in the fact that the convolution increases regularity, as proved in \cref{prop_regularity_network}, and as can also be practically observed in \cref{fig_smoothing_convoution}, where randomly initialized latent variables and filter kernels are convolved.
\cref{prop_regularity_network} in particular shows that the output of our convolutional network is continuous whenever the network has at least two layers.
This has several important implications: First, it explains the efficacy of the proposed generative prior $\Gc$ for removing noise-like artifacts, which are typically highly discontinuous. Second, \cref{prop_regularity_network} also suggests to combine $\Gc$ with a second regularizing functional $\Rc$, that allows for jump discontinuities (e.g. $\Rc = TV$), since otherwise our model would not be able to reconstruct these.
 
If, on the contrary, also the filter kernels $\theta$ are not more regular than belonging to $\Mc(\Sigma)$, it is easy to see that the convolution does, in general, not increase regularity as we can only expect $\mu^{L-1}\in \Mc(\Omega^{L-1})^{N_{L-1}}$. This has an interesting consequence for generative-neural-network based models: In case no regularization of the latent variables $\mu$ and the kernels $\theta$ is used, as done in the original deep image prior \cite{deep_image_prior} and most of the subsequent works, the learned latent variables and kernels can be arbitrarily irregular and the resulting network might generate noise. This explains why, indeed, it makes sense to use early stopping as proposed in \cite{deep_image_prior}. 
 
\begin{figure}[h]
\centering
\includegraphics[scale=0.5]{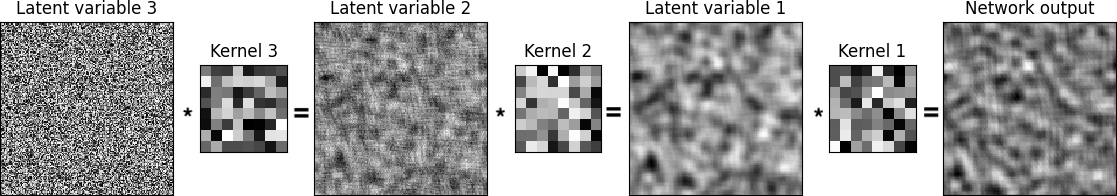}
\caption{Smoothing effect of the convolution with increasing network depth: After randomly initializing all kernels and the latent variables in the deepest layer (latent variable 3), this figure shows from, left to right, the resulting variables of intermediate layers and the network output. As can be observed visually, the smoothness of the variables increases with an increased number of subsequent convolutions.}
\label{fig_smoothing_convoution}
\end{figure}
\end{remark}
 
We now begin our analysis of the functional $\mathcal{G}$. Our main goal is to prove that $\mathcal{G}$ is coercive and lower semicontinuous, which makes it an appropriate regularizing functional. We start with a general continuity result for the convolution.

\begin{lemma}[Sequential weak*-continuity of the convolution]
\label{lem_conv_weak*_cont}
Take $q \in (1,2]$ arbitrary and let $(g_m)_m$, $ g$ in $ L^2(\Sigma)$ and $(\mu_m)_m $, $\mu $ in $\mathcal{M}(\Omega_\Sigma)$ be such that $g_m\rightharpoonup g$  and $\mu_m \xrightharpoonup{*} \mu$ as $m \rightarrow\infty $. Then, it follows that
\[
\mu_m*g_m\rightharpoonup\mu*g \text{ in }L^q(\Omega) \text{ as }m \rightarrow \infty .\]
\end{lemma}
\begin{proof}  The proof can be achieved in three steps: First, use the Arzelà–Ascoli theorem to show that, for $\phi \in C_c(\Omega)$ fixed, a subsequence of $(\langle\mu_m*g_m,\phi\rangle )_m$ converges to $\langle\mu*g,\phi\rangle$. Second, extend this to the entire sequence and third, via density, conclude weak* convergence as claimed. For the readers convenience, this proof is elaborated in detail in \Cref{supplement_proofs} of the supplementary material for this paper.
\end{proof}

Based on this, we can now establish lower semicontinuity of $\Gc$ in an appropriate topology. In this context, all notions of convergence in the product spaces $M$ and $\Theta$, such as weak*  convergence in $M$ or weak convergence in $\Theta$, refer to component-wise convergence in the underlying $\Mc$ and $L^2$, respectively.
 
\begin{lemma}\label{lem_Gv_lsc}
Let $\Gc:L^q(\Omega) \rightarrow [0,\infty]$ be given as in \eqref{eq_texture_prior}, where $q \in (1,2]$, $Z\in \Lc(\Theta ,\R^{N_c})$ and the $\Jc_l:M^l \rightarrow [0,\infty]$ are weak* lower semicontinuous with $\Jc_l(0) = 0$.
 
Then, for sequences $(v_m)_m $ in $L^q(\Omega)$, $(\mu_m)_m$ in $ M$ and $(\Theta_m)_m$ in $\Theta$ such that
\[ v_m\rightharpoonup v \text{ in } L^q(\Omega), \quad 
\mu_m \xrightharpoonup{*} \mu \text{ in } M, \text{ and}\quad 
\theta_m \rightharpoonup \theta \text{ in }\Theta,
\]
it holds that
\begin{itemize}
\item[i)] $G(\mu,\theta) \leq \liminf\limits_{m\rightarrow\infty} G(\mu_m,\theta_m)$ and
\item[ii)] if $(\mu_m,\theta_m)\in F(v_m)$ for all $m$, then also $(\mu,\theta)\in F(v)$.
\end{itemize}
\begin{proof}\
We start by proving i). By \cref{lem_conv_weak*_cont}, for all $l=2,\ldots,L$,
\[ (\mu_m)^l\conv_l(\theta_m)^l \rightharpoonup \mu^l\conv_l\theta^l \text{ in } L^2(\Omega^{l-1})^{N_{l-1}}\text{ as }m \rightarrow \infty.\]
This implies also convergence in the weak* sense in $M^{l-1}$. Further, $\Jc_l$ is weak* lower semicontinuous in $M^l$ by assumption and $\|\;.\; \|_\mathcal{M}$ is weak* lower semicontinuous as being a dual norm. Hence, $G$ is lower semicontinuous as claimed.
 
To prove ii), we have to show that $(\mu,\theta)$ satisfies the three constraints in \eqref{eq_texture_prior}. The first constraint $v = \mu^1 \conv_1 \theta^1$ holds true by uniqueness of weak limits and by weak convergence of $(\mu_m)^1 \conv_1 (\theta_m)^1$ to $\mu^1 \conv_1 \theta^1$ as shown in \cref{lem_conv_weak*_cont}.
The constraints $\|\Theta_{n,k}^l \|_2 \leq 1$ and $Z \theta = 0$ follow from weak convergence of $(\theta_m)_m$ using weak lower semicontinuity of $\|\cdot \|_2$ and weak-to-weak continuity of $Z$, respectively.
\end{proof}
\end{lemma}
 
\begin{lemma}\label{lem_Gv_attained} With the assumptions of \cref{lem_Gv_lsc}, for any $v\in L^q(\Omega)$ with $\mathcal{G}(v)<\infty$, the infimum in $\mathcal{G}(v)$ is attained, i.e., there exists $(\mu,\theta)\in F(v)$, such that $\mathcal{G}(v)=G(\mu,\theta)$.
\begin{proof}
Since $\mathcal{G}(v)<\infty$, we pick a minimizing sequence $(\mu_m$, $\theta_m)_m$ in $F(v)$ such that the sequence $(G(\mu_m,\theta_m))_m$ is bounded. From this, the definitions of $G$ and $F(v)$ imply that the sequences $(\mu_m)_m\subset M$ and $(\theta_m)_m\subset\Theta$ are bounded.
By the theorem of Banach-Alaoglu applied to the space of Radon measures and reflexivity of $L^2$, we can hence find (non-relabeled) subsequences such that $\mu_{m}\xrightharpoonup{*} \mu$ in $M$ and $\theta_{m}\rightharpoonup \theta$ in $\Theta$ as $m\rightarrow\infty$. \cref{lem_Gv_lsc} yields, that also the limit $(\mu,\theta)$ is feasible, i.e., $(\mu,\theta)\in F(v)$, and that
\[ G(\mu,\theta) \leq \liminf\limits_{m\rightarrow\infty} G(\mu_m,\theta_m) = \mathcal{G}(v),\]
which concludes the proof.
\end{proof}
\end{lemma}

\begin{lemma}\label{lem_G_lsc}
With the assumptions of \cref{lem_Gv_lsc}, the generative prior $\mathcal{G}: L^q(\Omega)\rightarrow [0,\infty]$ is weakly lower semicontinuous.
\begin{proof}
Let $v_m\rightharpoonup v$ in $L^q(\Omega)$. We want to show that $\mathcal{G}(v)\leq \liminf_{m\rightarrow\infty} \mathcal{G}(v_m)$ for which we can assume that $\liminf_{m\rightarrow\infty} \mathcal{G}(v_m)<\infty$, as otherwise there is nothing to prove. By moving to a (non-relabeled) subsequence, we can further assume that $\liminf_{m\rightarrow\infty} \mathcal{G}(v_m)=\lim_{m\rightarrow\infty} \mathcal{G}(v_m)$ and that $\mathcal{G}(v_m)<\infty$ for all $m$. By \cref{lem_Gv_attained}, for every $m$, there exists $(\mu_m,\theta_m)\in F(v_m)$, such that $\mathcal{G}(v_m)=G(\mu_m,\theta_m)$. The fact that $(\mathcal{G}(v_m))_m$ is bounded and the definitions of $G$ and $F(v_m)$ imply that the sequences $(\mu_m)_m$ and $(\theta_m)_m$ are bounded in $M$ and $\Theta$, respectively. As in the proof of \cref{lem_Gv_attained}, we can extract a further subsequence, again not relabeled, and find $\mu$, $\theta$ such that $\mu_{m}\xrightharpoonup{*} \mu$ in $M$ and $\theta_{m}\rightharpoonup \theta$ in $\Theta$ as $m\rightarrow\infty$. \cref{lem_Gv_lsc} then implies that $(\mu,\theta)\in F(v)$ and
\begin{equation*}
\begin{gathered}
\mathcal{G}(v) \leq G(\mu,\theta) \underbrace{\leq}_\text{\cref{lem_Gv_lsc}} \liminf\limits_{m\rightarrow\infty} G(\mu_m,\theta_m)=\liminf\limits_{m\rightarrow\infty} \mathcal{G}(v_m) = \lim\limits_{m\rightarrow\infty} \mathcal{G}(v_m).
\end{gathered}
\end{equation*}
\end{proof}
\end{lemma}
 
\begin{lemma}\label{lem_G_coercive}With the assumptions of \cref{lem_Gv_lsc}, $\mathcal{G}: L^q(\Omega)\rightarrow [0,\infty]$ is proper and coercive.
\begin{proof}
In order to show that $\mathcal{G}$ is proper, we simply note that $\mathcal{G}(0)=0$. To prove coercivity, let $(v_m)_m\subset L^q(\Omega)$ such that $\norm{v_m}_q\rightarrow\infty$. We have to show that $\mathcal{G}(v_m)\rightarrow\infty$ as well. Assuming the contrary, there is a subsequence $(v_{m_k})_k$ of $(v_m)_m$ such that $(\mathcal{G}(v_{m_k}))_k$ is bounded. By \cref{lem_Gv_attained}, we can pick $(\mu_{m_k}, \theta_{m_k})\in F(v_{m_k})$ such that $\mathcal{G}(v_{m_k})=G(\mu_{m_k},\theta_{m_k})$. This implies again that $(\mu_{m_k})_k$ and $(\theta_{m_k})_k$ are bounded, which, from \cref{lem_conv_measures}, in turn implies that also $v_{m_k}$ is bounded in $L^q(\Omega)$, contradicting our assumption.
\end{proof}
\end{lemma}

\subsubsection{Analysis for inverse problems}

The goal of this section is to prove that \eqref{eq_cont_problem} admits a stable minimum, which converges for vanishing noise. This will be done for $\Gc$ as defined in \Cref{sec_texture_prior}. %
and for rather general functionals $\Rc$ and $\Dc_y$. 
For the sake of simplicity, we will provide the results in this section for linear forward operators $A\in \Lc(L^q(\Omega),Y)$, as this is possible without any additional assumptions on the forward model. Building on the properties of $\Gc$ as we have established in \Cref{sec_texture_prior}, a generalization of these results to non-linear inverse problems is possible under standard assumptions, see for instance \cite{Hofmann07_nonlinear_tikhonov_banach_mh,scherzer2008variational}. 
Also for the case of linear $A$, our proofs on existence, stability and convergence for vanishing noise can essentially be derived from the obtained properties of $\Gc$ with classical techniques (see for instance \cite{Hofmann07_nonlinear_tikhonov_banach_mh,scherzer2008variational}). However, as the infimal-convolution-based approach in our model as well as our technique of factoring out the kernel of the forward operator $A$ in order to avoid additional assumptions requires special care, we provide short self-contained proofs in the following.

The assumptions used in this setting are summarized as follows, with concrete examples of $Z$, $\Jc_l$, $\Rc$ and $\Dc_y$ fulfilling these assumptions provided in \cref{ex_data_reg} below.
\begin{assumption}[Assumptions for existence/stability]\label{ass_well_posedness} Let $q \in (1,2]$ and assume further:
\begin{enumerate}
\item The data fidelity functional $\mathcal{D}_y: Y\rightarrow [0,\infty]$ is proper, coercive, lower semicontinuous, and convex, and $A \in \Lc(L^q(\Omega),Y)$ is the forward operator.
\item The generative prior $\Gc:L^q(\Omega) \rightarrow [0,\infty]$ is given as in \eqref{eq_texture_prior}, where $Z\in \Lc(\Theta ,\R^{N_c})$ and the $\Jc_l:M^l \rightarrow [0,\infty]$ are weak* lower semicontinuous with $\Jc_l(0) = 0$.
\item\label{item_ass_R} The prior $\Rc:L^q(\Omega)\rightarrow [0,\infty]$ is proper, weakly lower semicontinuous, and there exists a closed subspace $U\subset L^q(\Omega)$ and a continuous linear projection $P_U:L^q(\Omega)\rightarrow U$ such that for every $u\in L^q(\Omega)$, $v\in U$, $\Rc(u) = \Rc(u+v)$ and $\| u - P_Uu \|_q\leq C \Rc(u)$ with $C>0$ and 
\begin{enumerate}
\item $U$ is finite dimensional or
\item $U\cap \ker(A)$ admits a complement $Z$ in $U$ and $\| u \|_q\leq D\|Au\|_Y$ for all $u\in Z$ and some $D>0$.
\end{enumerate}
\end{enumerate}
\end{assumption}
\begin{remark} Note that in \cref{ass_well_posedness}, \eqref{item_ass_R}, i) implies ii), so we will assume ii) in the following. Indeed, if $U$ is finite dimensional, then $U\cap \ker(A)$ admits a complement $Z$ in $U$. Moreover, the restriction $A|_Z:Z\rightarrow A(Z)$ is a bijective, linear operator between finite dimensional spaces, which implies $\| u \|_q = \| (A|_Z)^{-1}A|_Z u\|_q \leq \| (A|_Z)^{-1}\| \|A|_Zu\|_Y = \| (A|_Z)^{-1}\| \|Au\|_Y$ for all $u\in Z$.
\end{remark}
 
\begin{ex}\label{ex_data_reg}
We list some examples of functionals $\mathcal{D}_y$, $\Gc$ and $\Rc$ satisfying \cref{ass_well_posedness}.
For the \emph{data discrepancy} term $\mathcal{D}_y$, power-of-norm discrepancies such as $\mathcal{D}_y(z)= \frac{1}{r}\| z-y \|_Y^r$ with $r \in [1,\infty)$ are feasible. 
This includes $Y = L^2(\Xi)$ with $\Xi \subset \R^{N_{\Dc}}$ and $\mathcal{D}_y(z)= \frac{1}{2}\| z-y \|_{L^2(\Xi)}^2$, the standard choice in case the data is perturbed by Gaussian noise. For inpainting, $\mathcal{D}_y(z) = \mathcal{I}_{\{y\}}(z)$ is feasible and, in case the data is perturbed by Poisson distributed noise, it is feasible to choose 
$\mathcal{D}_y = \KL(\;.\;,y)$, with $\KL$ being the Kullback-Leibler divergence. For properties of $\KL$ that indeed ensure \cref{ass_well_posedness} to hold, we refer to \cite{holler20ip_review_mh}.

Regarding $\Gc$, possible specifications of $Z$ and $\Jc_l$ are $Z = 0$, $Z\theta = (\int_{\Sigma} \theta_{n,1}^1)_{n=1}^{N_1} $ or $Z\theta = (\int_{\Sigma} \theta_{n,k}^l)_{n,k,l}$, the last two yielding zero-mean constraints for the filter kernels of the last or of all layers, respectively. The $\Jc_l$ can for example be chosen as $\Jc_l = \Ic_{\{0\}}$, yielding the constraint $\mu^{l-1} = \mu^l \conv_l \theta^l$, or as $\Jc_l = \frac{\gamma}{2}\mathcal{V}_l^2$ with $\gamma>0$, where, for $\mu=(\mu_1, \ldots, \mu_{N_l})\in M^l$, 
\[\mathcal{V}_l(\mu) = \sup \left\{\; \sum\limits_{n=1}^{N_l}\int\limits_{\Omega^l} \phi_n(x)\; d\mu_n(x)\; | \; \phi=(\phi_1,\ldots, \phi_{N_l})\in C_0(\Omega^l)^{N_l}, \; \|\phi\|_2\leq 1 \right\}.\]
Indeed, the $\mathcal{V}_l$ (and, consequently, the $\Jc_l$) are weak* lower semicontinuous as being the pointwise supremum of weak* continuous functions. Also note that
$\mathcal{V}_l(\mu) = \|\mu\|_2 $ in case $ \mu \in L^2(\Omega)^{N_l}$, and $\mathcal{V}_l(\mu) = \infty$ else, 
yielding $\Jc_l(\mu^{l-1}-\mu^l*_l\theta^l) = \frac{\gamma}{2}\| \mu^{l-1}-\mu^l*_l\theta^l\|^2_2$, in the case $\mu^{l-1}\in L^2(\Omega^{l-1})^{N_{l-1}}$, a relaxed version of the constraint $\mu^{l-1} = \mu^l \conv_l \theta^l$.

Possible choices for the regularizing functional $\Rc$ are for instance $\Rc = \mathcal{I}_{\{0\}}$, which means that only $\Gc$ is used for regularization, or $\Rc = \TV: L^1(\Omega)\rightarrow [0,\infty]$ in case $1 < q\leq \frac{d}{d-1}$, where $d$ is the dimension of $\Omega$ (necessary for Poincar\'e's inequality). More generally, in the latter case a feasible choice is also $\Rc = J^{**}$, the bipolar/biconjugate of $J$, where, for $j:\mathbb{R}^d\rightarrow [0,\infty)$ coercive, convex, Lipschitz continuous and with linear growth,
\begin{equation*}
\begin{aligned}
J : L^1(\Omega) &\rightarrow [0,\infty] \\
u &\mapsto \begin{cases}
\int\limits_{\Omega} j(\nabla u)\; dx \quad &\text{if } u\in W^{1,1}(\Omega)\\
\infty \quad &\text{else.}
\end{cases}
\end{aligned}
\end{equation*}
In this case, $J^{**}$ coincides with the lower semicontinuous regularization of $J$. For background information and application-specific details in this context we refer to \cite{ekeland_teman_conv_analysis} and \cite{habring_master}, respectively. Lower semicontinuity and convexity of $\Rc$ hold by definition and the subspace $U$ from \cref{ass_well_posedness} is the space of all constant functions, as in the case $\Rc=\TV$. For further generalizations, allowing also for a spatial dependence of $j$ in $\Omega$, we refer to \cite{amar2008lower,hintermuller2018function}.
\end{ex}
 
Under \cref{ass_well_posedness}, existence can now be guaranteed as follows.
\begin{theorem}[Existence of solutions]\label{thm_existence}
Suppose, that \cref{ass_well_posedness} holds true and assume there exists $\hat{u}\in L^q(\Omega)$ such that $\Rc(\hat{u})<\infty$ and $\mathcal{D}_y(A\hat{u})<\infty$. Then \eqref{eq_cont_problem} admits a minimum.
\begin{proof}
Since, by assumption, $0\leq \Ec_y(\hat{u},0)<\infty$, we can find a minimizing sequence $(u_m,v_m)_m$ such that $\lim_{m \rightarrow \infty} \Ec_y(u_m,v_m) = \inf_{(u,v)} \Ec_y(u,v)\in [0,\infty)$. Since $\mathcal{D}_y$ and $\Rc$ are non-negative, the coercivity of $\mathcal{G}$ implies that $(v_m)_m$ is bounded in $L^q(\Omega)$. 
The sequence $(u_m)_m$, however, is not bounded in general, but we can construct a modified minimizing sequence that is bounded as follows: With $Z$ the complement of $U \cap \ker(A)$ as in \cref{ass_well_posedness}, let $P_Z:U\rightarrow Z$ be a continuous, linear projection onto $Z$ such that $I-P_Z:U\rightarrow U\cap \ker(A)$ is the corresponding projection onto $U\cap \ker(A)$ (here $I$ denotes the identity). With $P_U$ the projection onto $U$ from \cref{ass_well_posedness}, we define
\[ w_m \coloneqq u_m -(I-P_Z)P_Uu_m = u_m -P_Uu_m+P_ZP_Uu_m.\]
Since $(I-P_Z)P_Uu_m\in U\cap \ker(A)$, we have that $Au_m=Aw_m$ as well as $\Rc(u_m-v_m)=\Rc(w_m-v_m)$ and thus $\Ec_y(u_m,v_m)=\Ec_y(w_m,v_m)$. Hence, $(w_m,v_m)_m$ is also a minimizing sequence for $\Ec_y$. Additionally, we can now show that $(w_m)_m$ is bounded: By \cref{ass_well_posedness}, we have
\begin{equation}\label{eq_existence1}
\begin{gathered}
\| (u_m-v_m) -P_U(u_m-v_m)\|_q\leq C \Rc(u_m-v_m) \leq \frac{C}{s_\Rc(\nu)} \Ec_y(u_m,v_m),
\end{gathered}
\end{equation}
for $C>0$, therefore, $((u_m-v_m) -P(u_m-v_m))_m$ is bounded. Since $(v_m)_m$ is bounded, also $(v_m-P_Uv_m)_m$ is bounded and, consequently, so is $(u_m-P_Uu_m)_m$. Further, we find that, for constants $C,D>0$, 
\begin{equation*}
\begin{gathered}
\| P_ZP_Uu_m \|_q \leq D\|AP_ZP_Uu_m\|_Y = D\|AP_Uu_m\|_Y \leq D(\|A(u_m-Pu_m)\|_Y + \|Au_m\|_Y)\\
\leq D(\|A\| \underbrace{\|u_m-Pu_m\|_q}_\text{a)} + \underbrace{\|Au_m\|_Y}_\text{b)}),
\end{gathered}
\end{equation*}
where a) is bounded as just shown and b) is bounded by the coercivity of $\mathcal{D}_y$ and non-negativity of $\Rc$ and $\mathcal{G}$. Hence, $(w_m,v_m)_m$ is a bounded minimizing sequence and since $L^q(\Omega)$ is reflexive, we can pick (non-relabeled) weakly convergent subsequences of $(w_m)_m$ and $(v_m)_m$ such that $w_m\rightharpoonup w$ and $v_m\rightharpoonup v$ with $w,v\in L^q(\Omega)$. In particular, this also implies $w_m-v_m\rightharpoonup w-v$ and $Aw_m\rightharpoonup Aw$ by weak-to-weak continuity of $A$. Since all involved functionals $\Dc_y$, $\Rc$ and $\Gc$ are weakly lower semicontinuous and $s_\Rc(\nu), s_\mathcal{G}(\nu), \lambda>0$, we find that 
\[ \Ec_y(u,v) \leq \liminf\limits_{m\rightarrow\infty} \Ec_y(w_m,v_m) = \inf\limits_{\tilde{u},\tilde{v}\in L^q(\Omega)} \Ec_y(\tilde{u},\tilde{v}),\]
which concludes the proof.
\end{proof}
\end{theorem}
 
In our applications, we will use $\Rc = J^{**}$ with an appropriate functional $J:L^1(\Omega)\rightarrow [0,\infty ]$ and $\Jc_l = \frac{\gamma}{2}\mathcal{V}_l^2$ as mentioned in \cref{ex_data_reg}. For this setting, the regularity of solutions of \eqref{eq_cont_problem} can be described as follows.
 
\begin{proposition} In case the network has $L\geq 2$ layers, $\Rc = \TV$ or, more generally, $\Rc = J^{**}$ as in \cref{ex_data_reg} and $\Jc_l$ is such that $\Jc_l(\mu)<\infty$ implies $\mu\in L^2(\Omega^l)^{N_l}$, any solution $u$ of \eqref{eq_cont_problem} can be written as $u = (u-v) + v$ with $(u-v) \in \BV(\Omega)$ and $v \in C(\overline{\Omega})$.
\end{proposition}
\begin{proof}
This is a direct consequence of \cref{prop_regularity_network} and the fact that $\Rc(u-v)<\infty$ implies $u-v\in BV(\Omega)$, see for instance \cite{bvfunctions,habring_master}.
\end{proof}

In order to discuss stability of the solution, we first introduce notions of convergence and coercivity of the data fidelity term for varying data, which are taken from \cite{holler20ip_review_mh}. With $(y_m)_m$ and $ y$ in $ Y$, we say $\mathcal{D}_{y_m}$ converges to $\mathcal{D}_y$ and write $\mathcal{D}_{y_m}\rightarrow \mathcal{D}_y$, if 
\begin{equation}\label{defin_data_conv}
\begin{cases}
\mathcal{D}_y(z)\leq \liminf\limits_{m\rightarrow\infty}\mathcal{D}_{y_m}(z_m)\quad &\text{for every } z_m\rightharpoonup z \text{ in } Y,\\
\mathcal{D}_y(z)\geq \limsup\limits_{m\rightarrow\infty}\mathcal{D}_{y_m}(z)\quad &\text{for every } z \in Y.
\end{cases}
\end{equation}
Moreover, we say that $(\Dc_{y_m})_m$ is equi-coercive, if there exists a coercive function $\Dc_0:Y \rightarrow [0,\infty]$ such that $\Dc_0 \leq \Dc_{y_m}$ for all $m$. 
 
Note that the notions of convergence and equi-coercivity ensure the properties of the data term that are necessary for stability estimates, without requiring to explicitly consider the interplay of a convergence of $(y_m)_m$ in $Y$ and properties of the data term. For standard choices of $\Dc_y$, they can be ensured via convergence of $(y_m)_m$ as follows: In case $\Dc_y(z) = (1/q)\|y-z\|_Y^q$, $\Dc_{y_m} \rightarrow \Dc_y$ and equi-coercivity hold, whenever $y_m \rightarrow y$ in $Y$ as $m\rightarrow \infty$. In case $\Dc_y(z) = \KL(z,y)$ with $Y = L^1(\Xi)$, $\Dc_{y_m} \rightarrow \Dc_y$ and equi-coercivity hold, whenever $\KL(y,y^m) \rightarrow 0$ as $m \rightarrow \infty$ and $y_m \leq C y$ for some $C>0$ and all $m$, see \cite[Example 2.16]{holler20ip_review_mh}.
 
Within the proof of stability, we will further make use of the following lemma.
\begin{lemma}\label{lem_sum_lim}
Let $(a_m)_m$ and $(b_m)_m$ be real sequences and $a,b\in \mathbb{R}$. Assume further, $a_m+b_m\rightarrow a+b$ and $a\leq\liminf\limits_{m\rightarrow\infty}a_m$, $b\leq\liminf\limits_{m\rightarrow\infty}b_m$. Then, $a_m\rightarrow a$ and $b_m\rightarrow b$.
\begin{proof} We simply compute
\[a \leq \liminf\limits_{m\rightarrow\infty}(a_m + b_m-b_m) = a+b+\liminf\limits_{m\rightarrow\infty}(-b_m) = a+b-\limsup\limits_{m\rightarrow\infty}b_m.\]
Hence, $\limsup\limits_{m\rightarrow\infty}b_m\leq b \leq \liminf\limits_{m\rightarrow\infty}b_m$ and accordingly $b_m\rightarrow b$. As a result, also $a_m\rightarrow a$.
\end{proof}
\end{lemma}
 
\begin{theorem}[Stability]\label{thm_stability}
Let \cref{ass_well_posedness} hold and assume that $\mathcal{D}_{y_m}\rightarrow\mathcal{D}_y$ in the sense of \eqref{defin_data_conv} and that $(\mathcal{D}_{y_m})_m$ is equi-coercive. Let, for each $m$, $(u_m,v_m)$ be a minimizer of $\Ec_{y_m}$. Then we have either
\begin{itemize}
\item[i)] $\Ec_{y_m}(u_m,v_m)\rightarrow\infty$ as $m\rightarrow\infty$ and $\Ec_y(u,v) = \infty$ for any $u,v$, or
\item[ii)] $\Ec_{y_m}(u_m,v_m)\rightarrow \min\limits_{\tilde{u},\tilde{v}\in L^q(\Omega)} \Ec_y(\tilde{u},\tilde{v})< \infty$ as $m\rightarrow\infty$ and, up to shifts in $(U \cap \ker(A)) \times \{0\}$, $(u_m,v_m)_m$ admits a weak accumulation point $(u,v)\in L^q(\Omega)^2$ that minimizes $\Ec_y$.
\end{itemize}
Moreover, in the latter case, for each subsequence $(u_{m_k},v_{m_k})_k$ converging weakly to some $(u,v)\in L^q(\Omega)^2$, we have that $(u,v)$ is a minimizer of $\Ec_y$ and $\mathcal{D}_{y_{m_k}}(Au_{m_k})\rightarrow\mathcal{D}_{y}(Au)$, $\Rc(u_{m_k}-v_{m_k})\rightarrow\Rc(u-v)$ and $\mathcal{G}(v_{m_k})\rightarrow\mathcal{G}(v)$ as $m\rightarrow\infty$. If $\Ec_y$ admits a unique minimizer, then $(u_m,v_m)\rightharpoonup (u,v)$.
\begin{proof}
Assume first that $\Ec_{y_m}(u_m,v_m)\rightarrow\infty$ as $m\rightarrow\infty$. In case there exist $u,v\in L^q(\Omega)$ such that $\Ec_y(u, v)<\infty$, from convergence of the data fidelity term it follows that
\[\limsup\limits_{m\rightarrow\infty}\Ec_{y_m}(u_m,v_m)\underbrace{\leq}_\text{optimality}\limsup\limits_{m\rightarrow\infty}\Ec_{y_m}(u,v)\leq \Ec_y(u,v)<\infty,\]
which contradicts our assumption. Therefore,  $\Ec_y$ cannot admit a finite value, showing i).

Now assume to the contrary $\Ec_{y_m}(u_m,v_m)\not\rightarrow\infty$, that is $\liminf\limits_{m\rightarrow\infty}\Ec_{y_m}(u_m,v_m)<\infty$ and let $(u_{m_k},v_{m_k})_k$ be a subsequence such that 
\[\liminf\limits_{m\rightarrow\infty} \Ec_{y_m}(u_m,v_m)=\lim\limits_{k\rightarrow\infty} \Ec_{y_{m_k}}(u_{m_k},v_{m_k}).\]
We define $w_m = u_m - (I-P_Z)P_Uu_m$ as in the proof of \cref{thm_existence} such that for each m, $(w_m,v_m)$ is again a minimizer of $\Ec_{y_m}$. Using the equi-coercivity of $(\mathcal{D}_{y_m})_m$, we find that $(w_{m_k},v_{m_k})_k$ is bounded and hence admits a (non-relabeled) subsequence weakly converging to some $(u,v)$. Now, let $\tilde{u},\tilde{v}\in L^q(\Omega)$ be arbitrary. Using the convergence of the data term, weak-to-weak continuity of the operator $A$ and the weak lower semicontinuity of $\Rc$ and $\mathcal{G}$, we obtain
\begin{equation}\label{eq_thm_stability1}
\begin{aligned}
\Ec_y(u,v)
& \leq \liminf\limits_{k\rightarrow\infty} \Ec_{y_{m_k}}(w_{m_k},v_{m_k}) =\liminf\limits_{m\rightarrow\infty} \Ec_{y_{m}}(w_{m},v_{m}) \\
 & \leq \limsup\limits_{m\rightarrow\infty} \Ec_{y_{m}}(w_{m},v_{m})
\underbrace{\leq}_\text{optimality}\limsup\limits_{m\rightarrow\infty} \Ec_{y_{m}}(\tilde{u},\tilde{v})\leq \Ec_{y}(\tilde{u},\tilde{v}).
\end{aligned}
\end{equation}
This shows that $(u,v)$ is a minimizer of $\Ec_y$. Moreover, if we plug in $(\tilde{u},\tilde{v})=(u,v)$, we find that 
\[\min\limits_{\tilde{u},\tilde{v}\in L^q(\Omega)} \Ec_y(\tilde{u},\tilde{v})=\Ec_y(u,v)=\lim\limits_{m\rightarrow\infty}\Ec_{y_{m}}(u_{m},v_{m}),\]
which finishes the proof of ii).

To conclude the proof, assume that $(u_{m_k},v_{m_k})_k$ is a subsequence of $(u_m,v_m)_m$, weakly converging to $(u,v)$. Then
\[ \Ec_y(u,v) \leq \liminf\limits_{k\rightarrow\infty} \Ec_{y_{m_k}}(u_{m_k},v_{m_k}) =\lim\limits_{m\rightarrow\infty} \Ec_{y_{m}}(u_{m},v_{m}) = \min\limits_{\tilde{u},\tilde{v}\in L^q(\Omega)} \Ec_y(\tilde{u},\tilde{v}).\]
Hence, $(u,v)$ is a minimizer of $\Ec_y$ and \cref{lem_sum_lim} implies the convergence of the three parts $(\mathcal{D}_{y_{m_k}}(Au_{m_k}))_k$, $(\Rc(u_{m_k}))_k$ and $(\mathcal{G}(u_{m_k}))_k$ individually. Finally, assume $(u,v)$ is the unique minimizer of $\Ec_y$.  Then we find that $U\cap\ker (A)=\{ 0 \}$, since if there was $0\neq w\in U\cap \ker (A)$, then $(u+w,v)$ would be a second, distinct solution. Therefore, 
\[ u_m = u_m - (I-P_Z)P_Uu_m\]
and consequently $(u_m,v_m)_m$ is already bounded. Thus, every subsequence of $(u_m,v_m)_m$ contains another subsequence weakly converging to $(u,v)$, which implies $(u_m,v_m)\rightharpoonup (u,v)$ by a standard contradiction argument.
\end{proof}
\end{theorem}

Convergence of solutions for vanishing noise can be proven with similar techniques as in \cref{thm_stability} and we refer to \Cref{supplement_proofs} of supplementary material for an additional result in that direction.

\section{Numerical Results}\label{experiments}
 
In this section, we provide numerical results for imaging applications such as inpainting, denoising, deconvolution and JPEG decompression. Further results on super-resolution can be found in the supplement \Cref{SM_section_super-resolution}. For all results, we use a particular instance of the proposed approach with the following architecture:
 
\begin{enumerate}[i)]
\item $L = 3$ and $N_l = 8$ for each $l = 1,2,3$. \label{concrete_model_dimensions}
\item $E_l = \{ (n,n)\st n \in \{1,\ldots,8\}\}$ for $l=2,3$ and $E_1 = \{(n,1) \st n \in \{1,\ldots,8\}\}$. \label{concrete_model_connections}
\item $\Rc = J^{**}$, where, for $\epsilon>0$, $J: L^1(\Omega)\rightarrow [0,\infty]$ is given as
 \label{concrete_model_r}
\begin{equation}\label{eq_application_cartoon_prior}
J(u) =  \begin{cases}
\frac{1}{|\Omega |}\int\limits_{\Omega}\sqrt{|\nabla u|^2 + \epsilon}\; dx \quad &\text{if } u\in W^{1,1}(\Omega),\\
\infty \quad &\text{else.}
\end{cases}
\end{equation}
\item For $l=1,2$, as detailed in \cref{ex_data_reg}, 
\[\Jc_l(\mu) = \begin{cases}
\frac{\gamma}{2}\| \; \mu \; \|_2^2\quad &\text{if } \mu \in L^2(\Omega^l)^{N_l},\\
\infty\quad &\text{ else.} \end{cases}\]
\label{concrete_model_j}
\end{enumerate}
The choices \eqref{concrete_model_dimensions} and \eqref{concrete_model_connections}  were made as they result in a rather simple network that still delivers significantly improved results compared, e.g., to a single-layer setting \cite{Chambolle2020}. Note in particular that, with our choice of $E_l$, different latent variables are only connected after the last layer. Besides aiming at a simple model with a reduced number of parameters, this choice is motivated by the observation that for deeper networks, as opposed to fully connected latent variables, it allows for a clear differentiation of features, see \cref{fig_network_decomposition} for an example. 

The choices of $\Rc$  and $\Jc_l$ are a compromise to allow for a rather simple numerical solution algorithm with convergence guarantees. From the modeling perspective, for $\Rc $ choices like $\TV$ or higher-order derivative-based regularization functionals (see \cite{holler20ip_review_mh}) as well as the choice $\Jc_l = \Ic_{\{0\}}$ seem preferable, but are more difficult to realize numerically, as we are dealing with a non-smooth, non-convex optimization problem.
 
Using this particular architecture, the energy in \eqref{eq_cont_problem} is discretized for images $u \in \R^{N_x \times N_y}$, where the forward operator $A$ and the data discrepancy $\Dc_y$ are chosen application specific as provided in \Cref{subsec:inpainting} to \Cref{subsec:jpeg}. In addition, as further means of dimensionality reduction in the network, we introduce a stride in the convolutions between different layers and denote the corresponding strided upconvolution as $*_\sigma$, see \Cref{supplement_discrete_model} for details.

Replacing $v$ in \eqref{eq_cont_problem} with its representation by $(\mu,\theta)$ and using the same notation for the discretized functionals as for the continuous counterparts, the corresponding minimization problem used in the implementation read as
\begin{multline}\label{eq_discrete_problem}
\tag{DP(y)}
\min\limits_{u,\mu,\theta} \;  \lambda\mathcal{D}_{y}(Au)+ s_\Rc(\nu) J( u-\sum_{n=1}^{N_1} \mu^1_{n}*_\sigma\theta^1_{n})
 + s_\mathcal{G}(\nu)\sum\limits_{l=1}^L \sum_{n=1}^{N_l} \norm{\mu^l_n}_1  \\ + \gamma\sum\limits_{l=2}^L \sum\limits_{n=1}^{N_l} \frac{1}{2}\left\| \mu^{l-1}_{n} -\mu^l_{n}*_\sigma\theta^l_n\right\|_2^2,
\end{multline}
subject to
\begin{equation*}
  \begin{cases}
  \|\theta^l_n\|^2_2 \leq 1 \quad \text{for all $l,n$}, \\
  \sum\limits_{i,j=1}^{r} (\theta^1_n)_{i,j} = 0, \quad \text{for all $n$}.
  \end{cases}
\end{equation*}
The discretizations of the $\nabla$, the different norms and integrals are rather standard and we refer to the supplement \Cref{supplement_discrete_model} for details.

A numerical solution of the discretized problem is obtained using the iPALM algorithm \cite{Pock_2016}, for which, in case of boundedness, convergence to a stationary point can be ensured. We implemented the method in Python for GPUs based on PyOpenCL \cite{klockner2012pycuda}. For additional details on the discretization, the algorithm and the implementation, we refer to the supplementary material of this paper and the publicly available source code \cite{gen_reg_git}. We also note that, in the discrete setting, we use a strided convolution, see \eqref{eq:upconvolution}, as a means of increasing the resolution of the latent variables from the deepest to the last layer, thereby reducing the dimensionality of the network.

\begin{figure}
\centering
\includegraphics[height = 2.5cm]{images_gen_reg_inpainting_barbara_crop_original.png}
\includegraphics[height = 2.5cm]{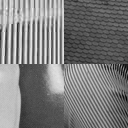}
\includegraphics[height = 2.5cm]{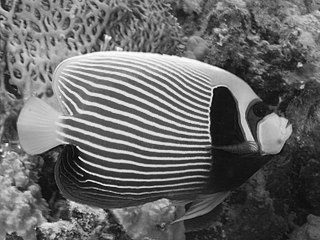}
\includegraphics[height = 2.5cm]{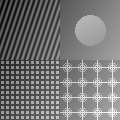}
\includegraphics[height = 2.5cm]{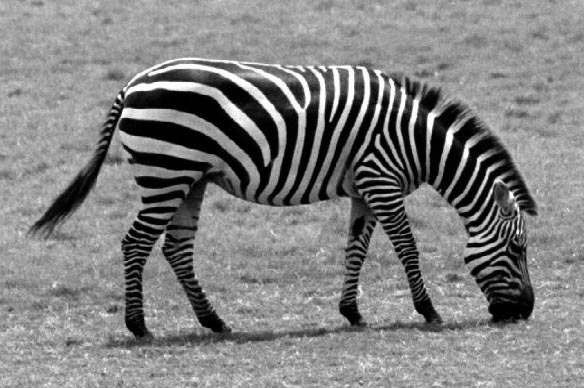}
\caption{Test images. From left to right, we refer to them as \emph{Barbara}, \emph{mix}, \emph{fish} (\href{https://commons.wikimedia.org/wiki/File:Pomocanthus_imperator_facing_right.jpg}{"Pomocanthus imperator facing right"}, by \href{https://commons.wikimedia.org/wiki/User:Albert_kok}{Albert kok}, licensed under \href{https://creativecommons.org/licenses/by-sa/4.0/}{CC BY-SA 4.0}), \emph{patchtest} and \emph{zebra} \cite{deep_image_prior_git}}
\label{fig_original_images}
\end{figure}
 
The test images we use for the numerical experiments can be found in \cref{fig_original_images}, and all experiments for the proposed generative prior as shown in this paper can be reproduced with the publicly available source code \cite{gen_reg_git}.  Note that the experiments with the \emph{patchtest} and \emph{zebra} images are omitted here for the sake of brevity and can be found in the  supplementary material \Cref{supplement_experiments}.
 
\noindent\textbf{Parameter choice.} Our method has three classes of parameters: i) The network architecture, which is fixed for all experiments and was chosen in view of obtaining a rather simple network, that still delivers good results, see \cref{table_parameters_general}. ii) The parameters $\epsilon$ and $\gamma>0$, which are necessary to achieve a smoothing of non-smooth functionals and constraints. We found their choice to be rather uncritical as long as they are sufficiently small and large, respectively, and again fixed those parameters for all experiments in the paper, see again \cref{table_parameters_general}. iii) The parameters $\lambda$ and $\nu$, which are the two effective parameters of our method. They define the trade-off between data fidelity and regularization, and between the priors $\Rc$ and $\Gc$, respectively. The choice of a regularization parameter $\lambda$ is necessary for any regularization method and can be made, e.g., according to classical parameter choice strategies, such as the discrepancy principle or the L-curve approach. The parameter $\nu$ can be interpreted as model parameter, that can be fixed for a class of images under consideration, independent of the noise level. In our experiments, for each application, we allow only a single parameter to vary (in order to obtain visually optimal results), while the other, if applicable, was fixed, see \cref{table_parameters}.

\noindent\textbf{Comparison to related methods.} We compare our method to the deep image prior \cite{deep_image_prior} (DIP), to \cite{Chambolle2020} (CL), which can be seen as convex relaxation of a single layer version of the proposed method, and to total generalized variation regularization \cite{bredies2010tgv} (TGV).
Details on the implementation used for these methods, as well as the parameter choice, can be found in \Cref{supplement_experiments}, and in particular \cref{table_numiter_dip}. In short, for each competing method, parameters were chosen in order to obtain visually optimal results, and for DIP in particular, also the network architecture was adapted to each experiment as suggested by the authors. Results for a larger experiment for inpainting with fixed parameters on subsets of the ImageNet ILSVRC2017 DET test data set \cite{ILSVRC15} can be found in \cref{tbl_inpainting_imagenet}, see \Cref{subsec:inpainting} below for details.
Note that we did not compare deconvolution to DIP since this experiment was also not presented in the original paper. Moreover, we compare to CL only on images which are also used in the original paper since the authors published their code for these images with optimal parameters.

\begin{table}[h]
\caption[PSNR/SSIM values]{PSNR/SSIM values of all results shown in the paper. Bold indicates the best result. A "-" indicates that the experiment was not carried out since it was not part of the original reference used for comparison or that no ground truth is available.}
\centering
\begin{tabular}{ p{2.7cm} p{1.8cm} p{1.8cm} p{1.8cm} p{1.8cm} p{1.8cm}}
\toprule
 & \emph{Barbara} & \emph{mix} & \emph{patchtest} & \emph{fish} & \emph{zebra}\\
 \midrule
 \textbf{Inpainting}\\
 TGV      & $20.43/0.71$       & $20.07/0.63$       & $19.21/0.64$       & $20.73/0.76$ & $25.66/0.83$\\
 CL         & $22.05/0.79$       & $25.93/0.85$       & 26.13/0.90       & - & - \\
 DIP        & 25.88/\textbf{0.91}       & \textbf{28.72}/\textbf{0.9}      & 25.78/0.9       & \textbf{26.19}/\textbf{0.88} & $27.73/0.86$\\
 proposed   & \textbf{27.17}/0.9  & 28.6/0.89  & \textbf{26.97}/\textbf{0.91}   & 25.19/0.87 & \textbf{28.59}/\textbf{0.91}\\
 
 \textbf{Denoising}\\
 TGV      & $23.64/0.75$       & $22.88/0.62$       & $21.7/0.69$      & $25.06/0.82$ & $26.4/0.79$\\
 CL         & $23.24/0.72$       & $25.41/0.75$       & $24.66/0.83$       & - & - \\
 DIP        & \textbf{26.68}/\textbf{0.82}  & $26.21/0.76$       & \textbf{26.54}/\textbf{0.84}  & \textbf{26.8}/\textbf{0.85} & \textbf{28.16}/0.80 \\
 proposed   & $26.21/0.81$       & \textbf{26.83}/\textbf{0.81}  & 25.64/\textbf{0.84}       & 25.52/0.82 & 27.19/\textbf{0.81}\\
 
 \textbf{Deconvolution}\\
 TGV      & 22.47/0.73       & 23.56/0.70 & 23.97/\textbf{0.84}         & 23.73/0.80       & 26.44/\textbf{0.83}\\
 CL         & 22.65/0.73       & 24.46/0.72 & - & - & - \\
 proposed   & \textbf{23.96}/\textbf{0.79}	  & \textbf{25.28}/\textbf{0.77}  & \textbf{24.61}/0.83  & \textbf{23.77}/\textbf{0.81}  & \textbf{26.58}/0.82\\
 
 \textbf{Super-resolution}\\
 DIP & - & - & - & 25.88/0.78 & \textbf{24.33}/0.71 \\
 proposed & - & - & - & \textbf{26.11}/\textbf{0.82} & 23.79/\textbf{0.77}\\
 
 \textbf{JPEG-decompression}\\
 DIP & 24.36/\textbf{0.81} & 25.19/\textbf{0.81} & \textbf{24.77}/\textbf{0.87} & \textbf{26.43}/\textbf{0.89} & \textbf{29.00}/\textbf{0.86} \\
 proposed & \textbf{24.84}/0.80 & \textbf{26.23}/\textbf{0.81}  & 24.3/0.85 & 25.66/0.85 & 27.67/0.85\\
 \bottomrule
\end{tabular}
\label{tbl_psnr}
\end{table}

\subsection{Inpainting} \label{subsec:inpainting}
In the case of inpainting, we aim to reconstruct an image from only a part of its pixels. Formally, given $\mathcal{M}\in \{0,1\}^{N_x\times N_y}$, such that $\mathcal{M}_{i,j}=1$ whenever the image at pixel $(i,j)$ is known and $\mathcal{M}_{i,j}=0$ else, we define
\begin{equation*}
Au = \mathcal{M} \odot u, \quad \Dc_{y} (z) = \Ic_{\{ y \}},
\end{equation*}
where $\odot$ denotes the Hadamard product and $\mathcal{M}$ is initialized such that around 70\% of the entries are zero. Results can be found in \cref{fig_inpainting} (and \cref{fig_inpainting_supplement}, including the \emph{patchtest} and \emph{zebra} images) with corresponding PSNR and SSIM values shown in \cref{tbl_psnr}.
 
Both visually and in terms of PSNR/SSIM, the methods CL, DIP and the proposed method clearly outperform TGV regularization (in particular in texture parts), which is expectable as the latter constitutes a model for piecewise smooth images. Compared to CL, the methods DIP and the proposed also deliver superior results, which is manifested in particular in a further improved reconstruction of texture parts and a smoother result overall. The latter can be attributed to the fact that CL uses a single layer, while DIP and the proposed method use multiple layers. Compared to DIP, the proposed method performs at least equally well, both visually and in terms of PSNR/SSIM, where notable differences are a slightly improved reconstruction of stripe-structures with our method and slightly less artifacts with the DIP prior, for instance, in the mouth region of the \emph{Barbara} image.

In order to show, that our method generalizes well, we also compare on two larger data sets to TGV and DIP. We used images from the ImageNet ILSVRC2017 DET test data set \cite{ILSVRC15}. The data set consists of 5500 images from which we took two different samples. For one sample we picked 100 images at random and for the second sample we handpicked 26 images containing distinctive texture parts. The 26 texture images are not contained in the random sample. We transformed all images to grayscale images and cropped them to a size of 256x256. Again we removed around 70\% of known pixels from the images for the inpainting task. We let all three methods run for 5000 iterations. For the proposed method the parameter $\nu$ was set to $0.9$ for all experiments. The resulting mean PSNR and SSIM scores and corresponding standard deviations with TGV, DIP and the proposed method can be found in \cref{tbl_inpainting_imagenet}. Compared to TGV both DIP and the proposed method perform significantly better in particular on the texture images, for which TGV naturally is a sub-optimal prior. On the texture images, DIP performs slightly better than the proposed method in terms of PSNR scores and comparably well in terms of SSIM scores. On the other hand, on the random sample, the proposed method performs slightly better both in terms of PSNR and SSIM scores. It can also be noted that the standard deviation of PSNR/SSIM values for the proposed method are consistently slightly below the ones obtained with the DIP, indicating an increased stability of our method.
Altogether, this results confirm an at least comparable performance of our method compared to DIP. 
 
\newcommand\fw{2.5cm} 
\newcommand\fwh{1.25cm} 

\newcommand\ffw{3cm} 
\newcommand\ffwh{1.5cm}

\begin{figure}
\begin{subfigure}{\textwidth}
\centering
\begin{subfigure}[t]{\fw}
\includegraphics[width = \fw]{images_gen_reg_inpainting_barbara_crop_corrupted.png}
\end{subfigure}%
\begin{subfigure}[t]{\fw}
\includegraphics[width = \fw]{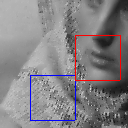}
\includegraphics[width = \fwh]{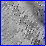}%
\includegraphics[width = \fwh]{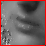}
\centering{$20.43/0.71$}
\end{subfigure}%
\begin{subfigure}[t]{\fw}
\includegraphics[width = \fw]{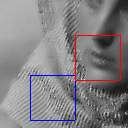}
\includegraphics[width = \fwh]{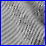}%
\includegraphics[width = \fwh]{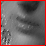}
\centering{$22.05/0.79$}
\end{subfigure}%
\begin{subfigure}[t]{\fw}
\includegraphics[width = \fw]{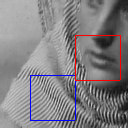}
\includegraphics[width = \fwh]{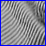}%
\includegraphics[width = \fwh]{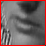}
\centering{25.88/\textbf{0.91}}
\end{subfigure}%
\begin{subfigure}[t]{\fw}
\includegraphics[width = \fw]{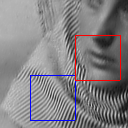}
\includegraphics[width = \fwh]{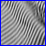}%
\includegraphics[width = \fwh]{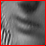}
\centering{\textbf{27.17}/0.9}
\end{subfigure}%
\begin{subfigure}[t]{\fw}
\includegraphics[width = \fw]{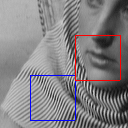}
\includegraphics[width = \fwh]{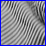}%
\includegraphics[width = \fwh]{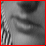}
\end{subfigure}
\end{subfigure}
 
\begin{subfigure}{\textwidth}
\centering
\begin{subfigure}[t]{\fw}
\includegraphics[width = \fw]{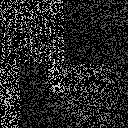}
\end{subfigure}%
\begin{subfigure}[t]{\fw}
\includegraphics[width = \fw]{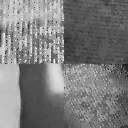}
\centering{$20.07/0.63$}
\end{subfigure}%
\begin{subfigure}[t]{\fw}
\includegraphics[width = \fw]{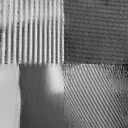}
\centering{$25.93/0.85$}
\end{subfigure}%
\begin{subfigure}[t]{\fw}
\includegraphics[width = \fw]{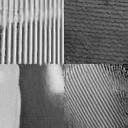}
\centering{\textbf{28.72}/\textbf{0.9}}
\end{subfigure}%
\begin{subfigure}[t]{\fw}
\includegraphics[width = \fw]{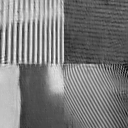}
\centering{28.6/0.89}
\end{subfigure}%
\begin{subfigure}[t]{\fw}
\includegraphics[width = \fw]{images_gen_reg_inpainting_cart_text_mix_original.png}
\centering
\end{subfigure}
\end{subfigure}
 
\begin{subfigure}{\textwidth}
\centering
\begin{subfigure}[t]{\ffw}
\includegraphics[width = \ffw]{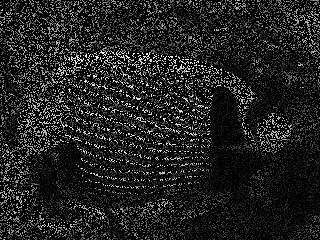}
\end{subfigure}%
\begin{subfigure}[t]{\ffw}
\includegraphics[width = \ffw]{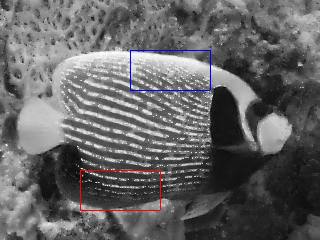}
\includegraphics[width = \ffwh]{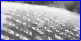}%
\includegraphics[width = \ffwh]{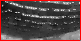}
\centering{$20.73/0.76$}
\end{subfigure}%
\begin{subfigure}[t]{\ffw}
\includegraphics[width = \ffw]{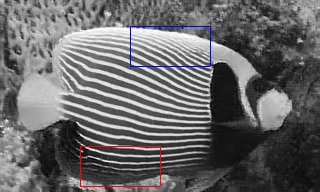}
\includegraphics[width = \ffwh]{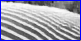}%
\includegraphics[width = \ffwh]{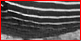}
\centering{\textbf{26.19}/\textbf{0.88}}
\end{subfigure}%
\begin{subfigure}[t]{\ffw}
\includegraphics[width = \ffw]{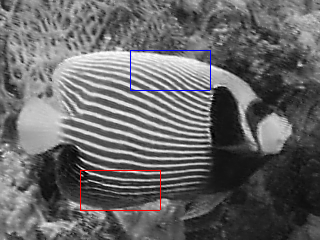}
\includegraphics[width = \ffwh]{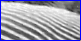}%
\includegraphics[width = \ffwh]{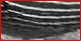}
\centering{25.19/0.87}
\end{subfigure}%
\begin{subfigure}[t]{\ffw}
\includegraphics[width = \ffw]{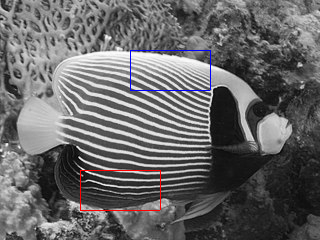}
\includegraphics[width = \ffwh]{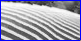}%
\includegraphics[width = \ffwh]{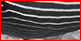}
\end{subfigure}
\end{subfigure}
 
\caption{Inpainting from 30\% known pixels with closeups. First two rows, from left to right: Data, TGV, CL, DIP, proposed, ground truth. Last row, from left to right: Data, TGV, DIP, proposed, ground truth. PSNR/SSIM values below images, bold indicates best result.}
\label{fig_inpainting}
\end{figure}

\begin{table}[h]
\caption[PSNR and SSIM values on data set.]{PSNR and SSIM values for inpainting from 30\% known pixels on subsets of the ImageNet data set. We provide the mean and standard deviation over the data set as \emph{mean} $\pm$ \emph{standard deviation}. Bold indicates the best result.}
\centering
\begin{tabular}{ p{1.5cm} p{3cm} p{2.5cm} p{2.5cm} p{2.5cm} }
\toprule
 & & \emph{TGV} & \emph{DIP} & \emph{proposed} \\
 \midrule
 \textbf{PSNR}
 &	Random images	& $24.58\pm 3.69$		& $25.19\pm 4.08$	& $\textbf{25.27}\pm 3.93$\\
 & Texture images & $22.04\pm 4.08$	& $\textbf{23.94}\pm 4.78$	& $23.71\pm 4.28$\\
 \textbf{SSIM}
 & Random images	& $0.81\pm 0.081$		& $0.83\pm 0.092$		& $\textbf{0.84}\pm 0.078$\\
 & Texture images	& $0.76\pm 0.075$	& $\textbf{0.83}\pm 0.076$		& $\textbf{0.83}\pm 0.063$ \\
 \bottomrule
\end{tabular}
\label{tbl_inpainting_imagenet}
\end{table}

\subsection{Denoising}
In the case of denoising, the forward operator $A$ is set to be the identity and the data fidelity term is given as $\Dc_{y}(z) = \frac{1}{2} \| z-y\|_2^2$, where the data is corrupted with Gaussian noise with mean zero and standard deviation 0.1 times the image range. Results can be found in \cref{fig_denoising} (and \cref{fig_denoising_supplement}, including the \emph{patchtest} and \emph{zebra} images) with corresponding PSNR and SSIM values shown in \cref{tbl_psnr}.
 
Similar as in the inpainting experiment, we observe that the DIP and the proposed method clearly outperform CL and TGV. In terms of PSNR values, the DIP seems slightly superior to the proposed method in this case, while in terms of SSIM values and visually the results seem comparable, again with slightly better results of our method for stripe structures and overall slightly smoother reconstructions obtained with DIP. It is also interesting to note that, in particular in the result for the \emph{Barbara} image, with DIP, CL and the proposed method, some structures of the texture part are also subtly visible in other parts of the image. This can be seen in particular in the smooth top-left region as well as \emph{Barbara's} face. While such artifacts can, in part, certainly be attributed to the rather strong noise level of the data in this case, they also seem related to the overall design of the three approaches.

\renewcommand\fw{2.5cm} 
\renewcommand\fwh{1.25cm} 

\renewcommand\ffw{3cm} 
\renewcommand\ffwh{1.5cm}
 
\begin{figure}
\centering
\begin{subfigure}{\textwidth}
\centering
\begin{subfigure}[t]{\fw}
\includegraphics[width = \fw]{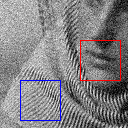}
\includegraphics[width = \fwh]{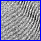}%
\includegraphics[width = \fwh]{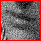}
\centering{$20.09/0.54$}
\end{subfigure}%
\begin{subfigure}[t]{\fw}
\includegraphics[width = \fw]{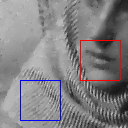}
\includegraphics[width = \fwh]{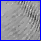}%
\includegraphics[width = \fwh]{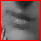}
\centering{$23.64/0.75$}
\end{subfigure}%
\begin{subfigure}[t]{\fw}
\includegraphics[width = \fw]{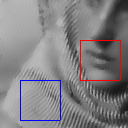}
\includegraphics[width = \fwh]{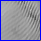}%
\includegraphics[width = \fwh]{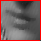}
\centering{$23.24/0.72$}
\end{subfigure}%
\begin{subfigure}[t]{\fw}
\includegraphics[width = \fw]{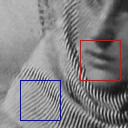}
\includegraphics[width = \fwh]{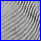}%
\includegraphics[width = \fwh]{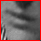}
\centering{\textbf{26.68}/\textbf{0.82}}
\end{subfigure}%
\begin{subfigure}[t]{\fw}
\includegraphics[width = \fw]{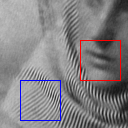}
\includegraphics[width = \fwh]{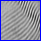}%
\includegraphics[width = \fwh]{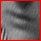}
\centering{$26.21/0.81$}
\end{subfigure}%
\begin{subfigure}[t]{\fw}
\includegraphics[width = \fw]{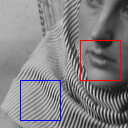}
\includegraphics[width = \fwh]{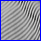}%
\includegraphics[width = \fwh]{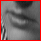}
\end{subfigure}
\end{subfigure}
 
\begin{subfigure}{\textwidth}
\centering
\begin{subfigure}[t]{\fw}
\includegraphics[width = \fw]{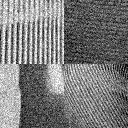}
\centering{$20.13/0.57$}
\end{subfigure}%
\begin{subfigure}[t]{\fw}
\includegraphics[width = \fw]{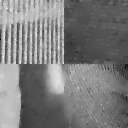}
\centering{$22.88/0.62$}
\end{subfigure}%
\begin{subfigure}[t]{\fw}
\includegraphics[width = \fw]{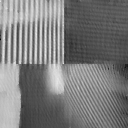}
\centering{$25.41/0.75$}
\end{subfigure}%
\begin{subfigure}[t]{\fw}
\includegraphics[width = \fw]{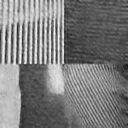}
\centering{$26.21/0.76$}
\end{subfigure}%
\begin{subfigure}[t]{\fw}
\includegraphics[width = \fw]{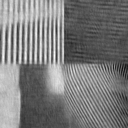}
\centering{\textbf{26.83}/\textbf{0.81}}
\end{subfigure}%
\includegraphics[width = \fw]{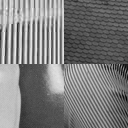}
\end{subfigure}
 
\begin{subfigure}{\textwidth}
\centering
\begin{subfigure}[t]{\ffw}
\includegraphics[width = \ffw]{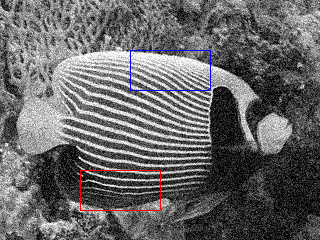}
\includegraphics[width = \ffwh]{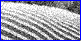}%
\includegraphics[width = \ffwh]{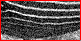}
\centering{$20.18/0.63$}
\end{subfigure}%
\begin{subfigure}[t]{\ffw}
\includegraphics[width = \ffw]{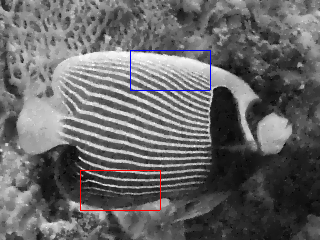}
\includegraphics[width = \ffwh]{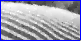}%
\includegraphics[width = \ffwh]{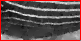}
\centering{$25.06/0.82$}
\end{subfigure}%
\begin{subfigure}[t]{\ffw}
\includegraphics[width = \ffw]{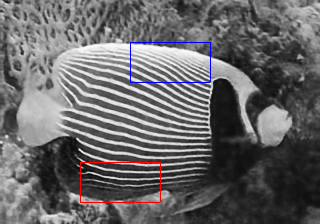}
\includegraphics[width = \ffwh]{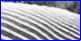}%
\includegraphics[width = \ffwh]{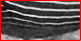}
\centering{\textbf{26.8}/\textbf{0.85}}
\end{subfigure}%
\begin{subfigure}[t]{\ffw}
\includegraphics[width = \ffw]{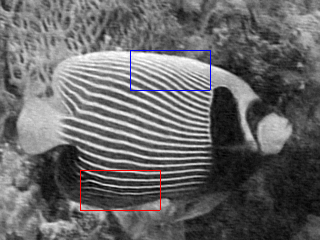}
\includegraphics[width = \ffwh]{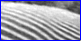}%
\includegraphics[width = \ffwh]{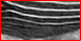}
\centering{25.52/0.82}
\end{subfigure}%
\begin{subfigure}[t]{\ffw}
\includegraphics[width = \ffw]{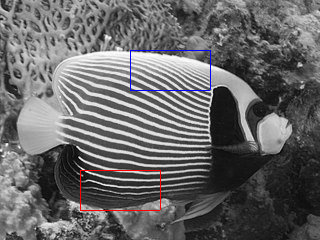}
\includegraphics[width = \ffwh]{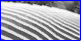}%
\includegraphics[width = \ffwh]{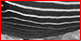}
\end{subfigure}
\end{subfigure}
 
\caption{Denoising with Gaussian noise with mean zero and standard deviation 0.1 times image range. First two rows, from left to right: Data, TGV, CL, DIP, proposed, ground truth. Last two row, from left to right: Data, TGV, DIP, proposed, ground truth. PSNR/SSIM values below images, bold indicates best result.}
\label{fig_denoising}
\end{figure}
 
\begin{figure}
\begin{subfigure}[b]{0.55\textwidth}
\centering
\includegraphics[height = 8.5cm]{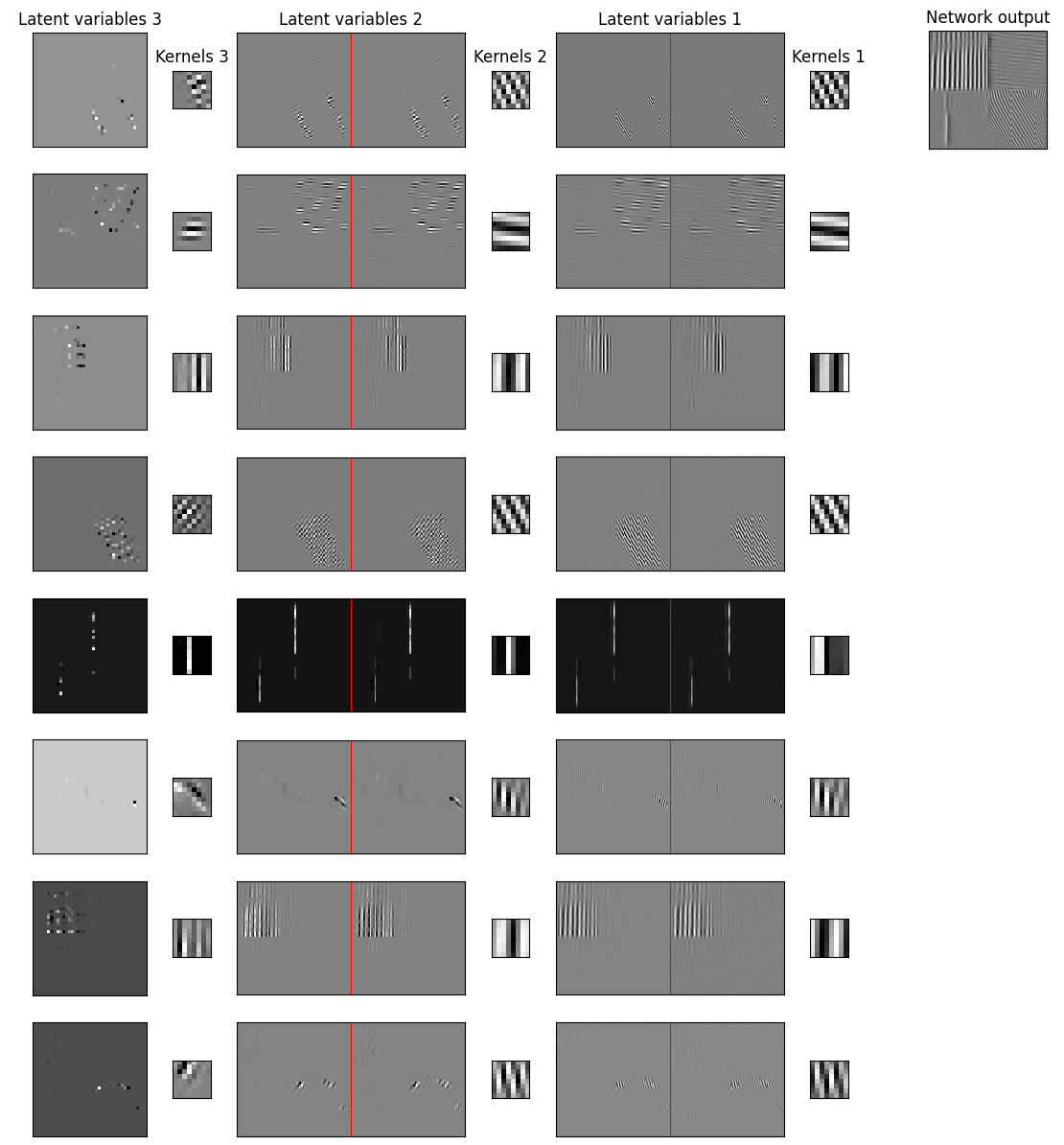}
\caption{\emph{Mix} network. The intermediate layers show the output of the previous convolution $\mu_n^{l+1}*_\sigma\theta_n^{l+1}$ (left) and the latent variable of the next convolution $\mu_n^l$ (right) next to each other, separated by a red line (recall that the discrepancy of these two is penalized in our objective functional by the functional $\Jc_l$). \label{fig_network_decomposition}}
\end{subfigure}%
\hspace*{0.3cm}
\begin{subfigure}[b]{0.4\textwidth}
\begin{subfigure}{\textwidth}
\centering
\includegraphics[height = 3cm]{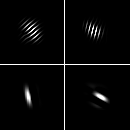}
\caption*{Four \emph{Barbara} samples.}
\end{subfigure}
\begin{subfigure}{\textwidth}
\centering
\includegraphics[height = 3cm]{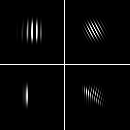}
\caption*{Four \emph{mix} samples.}
\end{subfigure}
\caption{Network samples (images are cropped):  The images $v$ where obtained via $v=\mu^1 *_1 \theta^1$ with $\mu^{l-1} = \mu^l*_l\theta^l$ for $l>1$, with $\theta$ obtained from denoising. The $\mu^L$ were choosen as delta peaks in different latent variables $\mu^L_n$.\label{fig_network_output}}
\end{subfigure}
\caption{Convolutional networks obtained from denoising. Left: Full network after denoising. Right: Samples from obtained networks with $\mu^L$ delta peaks at different latent variables.}
\label{fig_network}
\end{figure}

\cref{fig_network_decomposition} shows the generative network, consisting of filter kernels, latent variables and network output, that was obtained while denoising the \emph{mix} image with the proposed approach. We can observe that different latent variables and filter kernels clearly correspond to different features of the image, and that the main features of the image are captured rather well by the model. This is remarkable in particular in view of the fact that no training was used and that the network can only rely on noisy data to learn these features. 
 
We can also observe, in particular when considering the latent variables $\mu_2^{1}$ and $\mu_2^2*_\sigma\theta_2^2$ in the second row of the first layer of the network, that the output of the previous layer does not always coincide exactly with the latent variable that is passed forward to the next layer. This is a result of using $\Jc_l = \frac{\gamma}{2} \| \cdot \|_2 ^2 $ rather than $\Jc_l = \Ic_{\{0\}}$ in our numerical realization. It is worth noting that the difference of the two, however, mostly corresponds to noise-like artifacts. This indicates that these noise-like artifacts are difficult to generate from consecutive convolutions, due to the smoothing properties of the convolution (see \cref{fig_smoothing_convoution}). It also suggests that, by allowing for $\Jc_l = \Ic_{\{0\}}$ with an improved algorithm, our results might further be improved.
 
In \cref{fig_network_output}, we draw samples from two different learned networks, where the kernels were learned from denoising the \emph{Barbara} and \emph{mix} images.
 
\subsection{Deconvolution}
Given a convolution kernel $k\in\mathbb{R}^{(2s+1)\times (2s+1)}$, we define the forward operator as $A:\mathbb{R}^{N_x\times N_y}\rightarrow\mathbb{R}^{N_x\times N_y}$
\[ (Au)_{i,j} = \sum\limits_{i'=-s}^s \sum\limits_{j'=-s}^s k_{s+1+i',s+1+j'} u_{i-i',j-j'},\]
where we set $u_{i,j}=0$, whenever $(i,j)\notin \{1,2,...,N_x\}\times\{1,2,...,N_y\}$. The data discrepancy is given as $\Dc_{y}(z) = \frac{1}{2} \| z-y \|_2^2$. We use a Gaussian convolution kernel, with size $9\times 9$ for the \emph{Barbara} and \emph{mix} images, $13\times 13$ for the \emph{fish} image, and $15\times 15$ for the \emph{zebra} image, and a standard deviation of 0.25. The data is corrupted with Gaussian noise with standard deviation 0.025 times the image range and mean zero. 
The results of the deconvolution experiment are shown in \cref{fig_deconv} (and \cref{fig_deconv_supplement}, including the \emph{zebra} image) with PSNR and SSIM values shown in \cref{tbl_psnr}.
 
We can observe that, for the \emph{Barbara} and \emph{mix} image, the results with the proposed method are clearly superior to CL and TGV in terms of PSNR/SSIM values and also visually the results are superior in particular for the texture parts. For the \emph{fish} image the improvement is only minor both in terms of PSNR and visually. One reason for this might be the rather different size of this image and the fact that we do not adapt our network architecture between different experiments. This might be further improved by automatically adapting, e.g., the filter-kernel size to the image dimension, which is something we will consider in the course of developing improved overall algorithms for the proposed method within the scope of future work.

\renewcommand\fw{3cm} 
\renewcommand\fwh{1.5cm} 

\renewcommand\ffw{3.75cm} 
\renewcommand\ffwh{1.875cm}

\begin{figure}
\centering
\begin{subfigure}{\textwidth}
\centering
\begin{subfigure}[t]{\fw}
\includegraphics[width = \fw]{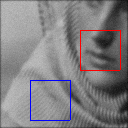}
\includegraphics[width = \fwh]{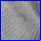}%
\includegraphics[width = \fwh]{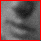}
\centering{$20.55/0.62$}
\end{subfigure}%
\begin{subfigure}[t]{\fw}
\includegraphics[width = \fw]{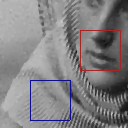}
\includegraphics[width = \fwh]{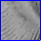}%
\includegraphics[width = \fwh]{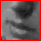}
\centering{22.47/0.73}
\end{subfigure}%
\begin{subfigure}[t]{\fw}
\includegraphics[width = \fw]{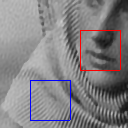}
\includegraphics[width = \fwh]{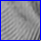}%
\includegraphics[width = \fwh]{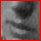}
\centering{22.65/0.73}
\end{subfigure}%
\begin{subfigure}[t]{\fw}
\includegraphics[width = \fw]{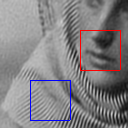}
\includegraphics[width = \fwh]{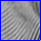}%
\includegraphics[width = \fwh]{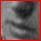}
\centering{\textbf{23.96}/\textbf{0.79}}
\end{subfigure}%
\begin{subfigure}[t]{\fw}
\includegraphics[width = \fw]{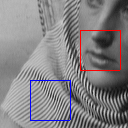}
\includegraphics[width = \fwh]{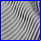}%
\includegraphics[width = \fwh]{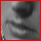}
\end{subfigure}
\end{subfigure}
 
\begin{subfigure}{\textwidth}
\centering
\begin{subfigure}[t]{\fw}
\includegraphics[width = \fw]{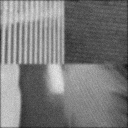}
\centering{$21.23/0.57$}
\end{subfigure}%
\begin{subfigure}[t]{\fw}
\includegraphics[width = \fw]{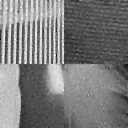}
\centering{23.56/0.70}
\end{subfigure}%
\begin{subfigure}[t]{\fw}
\includegraphics[width = \fw]{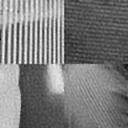}
\centering{24.46/0.72}
\end{subfigure}%
\begin{subfigure}[t]{\fw}
\includegraphics[width = \fw]{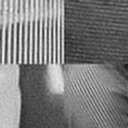}
\centering{\textbf{25.28}/\textbf{0.77}}
\end{subfigure}%
\begin{subfigure}[t]{\fw}
\includegraphics[width = \fw]{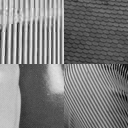}
\end{subfigure}%
\end{subfigure}
 
\begin{subfigure}{\textwidth}
\centering
\begin{subfigure}[t]{\ffw}
\includegraphics[width = \ffw]{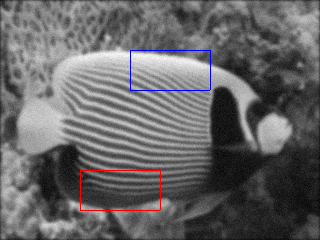}
\includegraphics[width = \ffwh]{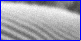}%
\includegraphics[width = \ffwh]{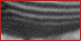}
\centering{$20.34/0.64$}
\end{subfigure}%
\begin{subfigure}[t]{\ffw}
\includegraphics[width = \ffw]{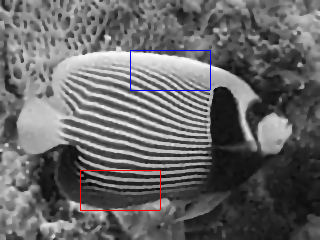}
\includegraphics[width = \ffwh]{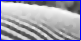}%
\includegraphics[width = \ffwh]{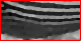}
\centering{23.73/0.80}
\end{subfigure}%
\begin{subfigure}[t]{\ffw}
\includegraphics[width = \ffw]{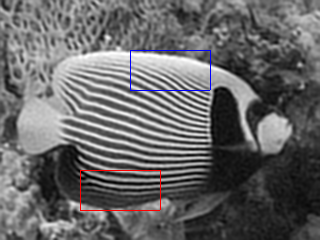}
\includegraphics[width = \ffwh]{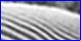}%
\includegraphics[width = \ffwh]{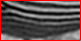}
\centering{\textbf{23.77}/\textbf{0.81}}
\end{subfigure}%
\begin{subfigure}[t]{\ffw}
\includegraphics[width = \ffw]{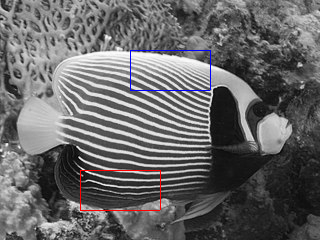}
\includegraphics[width = \ffwh]{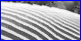}%
\includegraphics[width = \ffwh]{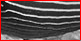}
\end{subfigure}
\end{subfigure}
 
\caption{Deconvolution. First two rows from left to right: Data, TGV, CL, proposed, ground truth. Last row from left to right: Data, TGV, proposed, ground truth. PSNR/SSIM values below images, bold indicates best result.}
\label{fig_deconv}
\end{figure}

\subsection{JPEG decompression} \label{subsec:jpeg}
The last experiment we consider is JPEG decompression. Here, the
forward operator $A$ is a block-cosine transform, while $\Dc_{y} = \Ic_{D(y)}$, with $D(y)$ containing all cosine coefficients which, when quantized with the given compression rate, would result in the same coefficients as stored in the given JPEG file $y$. For details we refer to \Cref{sm:sec_jpeg} and \cite{holler12tvjpeg_mh}.
The results for JPEG decompression can be found in \cref{fig_jpeg} (and \cref{fig_jpeg_supplement}, including the \emph{patchtest} image), and again the corresponding PSNR and SSIM values are provided in \cref{tbl_psnr}.
 
For JPEG decompression, it can be observed that the proposed method is rather clearly superior to DIP visually. In particular, JPEG compression artifacts are still visible in the DIP result, but not the results with the proposed method. With respect to PSNR/SSIM values DIP seems to perform slightly better. One should be aware, however, that both these discrepancies in performance might be a result of the different data fidelities used in DIP and our model, see \Cref{sm:sec_jpeg} for details.

In \cref{fig_jpeg_decomposition} one can observe the implicit decomposition of the images into two parts, one being generated from our network and denoted $v$ and the other one being penalized by $\Rc$ denoted $u-v$. The shown decompositions are obtained from the experiments of \cref{fig_jpeg}. One can observe that $v$ contains mostly texture and finer patterns of the image whereas $u-v$ captures piecewise smooth and/or constant regions. For instance for the \emph{Barbara} image we can clearly see that most of the face is contained in $u-v$ and the texture of the dress is mostly contained in $v$. As for the \emph{mix} image, a similar decomposition can be observed in that the textures are contained in $v$ and piecewise smooth/constant areas in $u-v$.

\renewcommand\fw{3cm} 
\renewcommand\fwh{1.5cm} 

\renewcommand\ffw{3.75cm} 
\renewcommand\ffwh{1.875cm}

\begin{figure}
\centering
\begin{subfigure}{\textwidth}
\centering
\begin{subfigure}[t]{\fw}
\includegraphics[width = \fw]{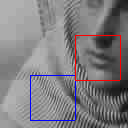}
\includegraphics[width = \fwh]{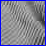}%
\includegraphics[width = \fwh]{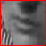}
\centering{23.49/0.75}
\end{subfigure}%
\begin{subfigure}[t]{\fw}
\includegraphics[width = \fw]{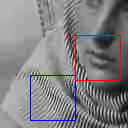}
\includegraphics[width = \fwh]{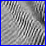}%
\includegraphics[width = \fwh]{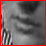}
\centering{24.36/\textbf{0.81}}
\end{subfigure}%
\begin{subfigure}[t]{\fw}
\includegraphics[width = \fw]{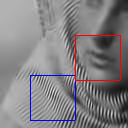}
\includegraphics[width = \fwh]{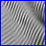}%
\includegraphics[width = \fwh]{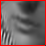}
\centering{\textbf{24.84}/0.80}
\end{subfigure}%
\begin{subfigure}[t]{\fw}
\includegraphics[width = \fw]{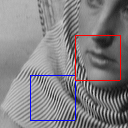}
\includegraphics[width = \fwh]{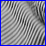}%
\includegraphics[width = \fwh]{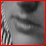}
\end{subfigure}
\end{subfigure}
\begin{subfigure}{\textwidth}
\centering
\begin{subfigure}[t]{\fw}
\includegraphics[width = \fw]{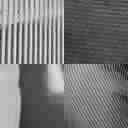}
\centering{25.53/0.77}
\end{subfigure}
\begin{subfigure}[t]{\fw}
\includegraphics[width = \fw]{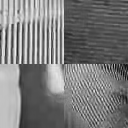}
\centering{25.19/\textbf{0.81}}
\end{subfigure}
\begin{subfigure}[t]{\fw}
\includegraphics[width = \fw]{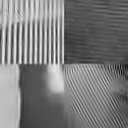}
\centering{\textbf{26.23}/\textbf{0.81}}
\end{subfigure}
\includegraphics[width = \fw]{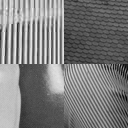}
\end{subfigure}
\caption{JPEG decompression. From left to right: Data, DIP, proposed, ground truth.}
\label{fig_jpeg}
\end{figure}

\begin{figure}
\centering
\begin{subfigure}[t]{\textwidth}
\centering
\includegraphics[width = \fw]{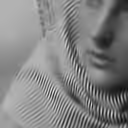}%
\includegraphics[width = \fw]{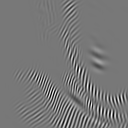}%
\includegraphics[width = \fw]{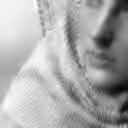}
\end{subfigure}

\begin{subfigure}{\textwidth}
\centering
\includegraphics[width = \fw]{images_gen_reg_jpeg_cart_text_mix_recon.png}%
\includegraphics[width = \fw]{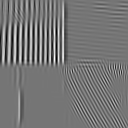}%
\includegraphics[width = \fw]{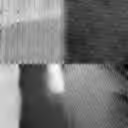}
\end{subfigure}
\caption{Image decomposition for JPEG decompression. From left to right: Reconstructed image $u$, generative part $v$, remaining part $u-v$.}
\label{fig_jpeg_decomposition}
\end{figure}

\section{Discussion}
We have introduced and investigated a novel generative variational regularization method for inverse problems in imaging. The method utilizes an energy-based prior whose architecture is inspired by the one of generative neural networks. We have proven theoretical results concerning existence/stability of solutions and convergence for vanishing noise. Moreover, we have shown that the proposed prior always generates continuous functions as outputs, which analytically confirms its efficacy for removing noise-like artifacts and sheds light on different numerical techniques and heuristics used in a large class of existing works related to the deep image prior \cite{deep_image_prior}.

For a discretized version of the proposed method, we have derived a numerical solution algorithm that allows a direct application to a rather general class of inverse problems in imaging. 
Experimentally, we have shown that our method outperforms classical variational methods and competes with recent state of the art deep learning methods such as the DIP, while having the advantage of a significantly reduced number of parameters and of a solid mathematical background.
 
Our next research goals are to further improve the proposed class of generative variational priors both in terms of analytic understanding, e.g., with respect to the structure of solutions, and in terms of numerical algorithms, e.g., to ensure general, resolution-independent applicability without manual parameter selection. In particular, we aim at developing convex relaxations of the proposed prior in the spirit of \cite{Chambolle2020_git}.

\printbibliography

\makeatletter\@input{xx2.tex}\makeatother

\end{document}


\newrefcontext[labelprefix=SM]
\maketitle

\section{Additional analytic results and proofs}\label{supplement_proofs}
This section provides additional proofs that were left out in the main part of this work. In particular, we first prove \cref{lem_conv_weak*_cont}, which is a result on weak* continuity of the convolution, and then consider a result on convergence for vanishing noise. 

For convenience of the reader, we repeat the statement of the lemma before proving it.

\begin{lemma}[Sequential weak*-continuity of the convolution]
Take $q \in (1,2]$ arbitrary and let $(g_m)_m$, $ g$ in $ L^2(\Sigma)$ and $(\mu_m)_m $, $\mu $ in $\mathcal{M}(\Omega_\Sigma)$ be such that $g_m\rightharpoonup g$  and $\mu_m \xrightharpoonup{*} \mu$ as $m \rightarrow\infty $. Then, it follows that
\[
\mu_m*g_m\rightharpoonup\mu*g \text{ in }L^q(\Omega) \text{ as }m \rightarrow \infty .\]
\end{lemma}
\begin{proof}
We fix $\phi \in C_c(\Omega)$ and proceed in three steps: First we show that a subsequence of $(\langle\mu_m*g_m,\phi\rangle )_m$ converges to $\langle\mu*g,\phi\rangle$, then we extend this to the entire sequence and lastly, via density, we conclude weak* convergence as claimed.

For $g \in L^2(\Sigma)$, define $F_g: \R^d \rightarrow \R$ as 
\[F_g(y) = \int\limits_{\Sigma} \tilde{\phi}(x+y) g(x) \; dx. \]
We  show that $F_g \in C_c(\Omega_\Sigma)$. Defining $K=\supp(\phi) \subset \Omega$, first note that $\supp(F_g) \subset K  - \overline{\Sigma}$, which is compact, hence, to ensure $\supp(F_g) \subset \Omega_\Sigma$, it suffices to show that $K - \overline{\Sigma} \subset \Omega_\Sigma$.
For the latter, note that for any $y=z-x\in K-\overline{\Sigma}$ with $z\in K$ and $x\in\overline{\Sigma}$, we can select $\delta >0 $ such that $\{z'\st |z-z'|<\delta\} \subset \Omega$.

Moreover, we can select $\tilde{x}\in \Sigma$, such that $|x-\tilde{x}|<\delta$, which implies  $ y=z-x=z +(\tilde{x}-x) - \tilde{x} \in \Omega - \Sigma=\Omega_\Sigma$ and hence $\supp(F_g) \subset \Omega_\Sigma$. Continuity of $F_g$ follows easily from the Lebesgue dominated convergence theorem, hence $F_g \in C_c(\Omega_\Sigma)$.
Using this, weak convergence of $(g_m)_m$ to $g$ in $L^2(\Sigma)$ implies that $F_{g_m} (y) \rightarrow F_g(y)$ for any $y \in \Omega_\Sigma$ as $m\rightarrow \infty$, where both $F_{g_m} $ for each $m$ and $F_g $ are contained in $C(\overline{\Omega_\Sigma})$. Using the Arzelà–Ascoli theorem, we now want to establish uniform convergence of a subsequence of $(F_{g_m})_m$, for which we show uniform boundedness and equi-continuity of $(F_{g_m})_m$. The former holds true, since, for  $y\in \overline{\Omega_\Sigma}$ arbitrary, 
\begin{equation*}
|F_{g_m}(y)| \leq \int\limits_{\Sigma} |\tilde{\phi}(x+y)| |g_m(x)| \; dx \underbrace{\leq}_\text{Hölder} \| \phi \|_2 \| g_m \|_2,
\end{equation*}
with the right hand side being uniformly bounded due to convergence of $(g_m)_m$ to $g$ in $L^2(\Sigma)$.
 
Regarding equi-continuity, let $\delta>0$ and $y_1,y_2\in \overline{\Omega_\Sigma}$, such that $|y_1-y_2|<\delta$. Then
\begin{align*}
|F_{g_m}(y_1)- F_{g_m}(y_2)| 
& \leq \int\limits_{\Sigma}|\tilde{\phi}(x+y_1)-\tilde{\phi}(x+y_2)|\;|g_m(x)|\;dx \\ 
& \leq \sup\limits_{\substack{z_1, z_2\in \Omega: \\ |z_1-z_2|<\delta}}|\phi(z_1)-\phi(z_2)| \; \| g_m\|_1 
\leq \sup\limits_{\substack{z_1, z_2\in \Omega: \\ |z_1-z_2|<\delta}}|\phi(z_1)-\phi(z_2)| \; C \| g_m\|_2 \rightarrow 0
\end{align*}
as $\delta\rightarrow 0$ uniformly for all $m$, since $\phi$ is uniformly continuous as a compactly supported, continuous function. With this, the Arzelà–Ascoli theorem yields a uniformly convergent subsequence $(F_{g_{m_k}})_k$, with the limit being the same as the pointwise limit $F_g$. Then the weak* convergence of $(\mu_m)_m$ implies
 
\begin{align*}
\left|\langle \mu_{m_k}*g_{m_k}-\mu*g, \phi\rangle\right| 
& \leq  \left|\langle \mu_{m_k}*g_{m_k}-\mu_{m_k}*g, \phi\rangle\right| + \left|\langle \mu_{m_k}*g-\mu*g, \phi\rangle\right| \\
& =  \bigg | \int\limits_{\Omega_\Sigma} F_{g_{m_k}}(y) - F_g(y) \; d\mu_{m_k}(y) \bigg| +  \bigg| \int\limits_{\Omega_\Sigma} F_g(y)\; d\mu_{m_k}(y)-\int\limits_{\Omega_\Sigma} F_g(y)\; d\mu(y)\bigg| \\
& \leq \underbrace{\| F_{g_{m_k}}-F_g \|_\infty\; \| \mu_{m_k} \|_\mathcal{M}}_{(a)} + \underbrace{\bigg| \int\limits_{\Omega_\Sigma} F_g(y)\; d\mu_{m_k}(y)-\int\limits_{\Omega_\Sigma} F_g(y)\; d\mu(y)\bigg|}_{(b)} \rightarrow 0,
\end{align*}
where $(a)$ goes to zero by the uniform convergence of $(F_{g_{m_k}})_k$ and boundedness of $(\mu_{m_k})_k$ and $(b)$ goes to zero by the weak* convergence of $(\mu_{m_k})_k$.
 
Assume now that this convergence does not hold true for the entire sequence, i.e., there exists $\epsilon>0$ and a subsequence $m_k$, such that for all $k$, $ \left|\langle\mu_{m_k}*g_{m_k}-\mu*g, \phi\rangle\right|>\epsilon$.
Then by the same arguments as above, we could extract a further subsequence with indices $m_{k_l}$, such that
$ \left|\langle\mu_{m_{k_l}}*g_{m_{k_l}}-\mu*g, \phi\rangle \right|\rightarrow 0 $
as $l\rightarrow\infty$, which is a contradiction. Hence, $\langle \mu_m*g_m, \phi \rangle \rightarrow \langle \mu*g, \phi \rangle$ as $m\rightarrow\infty$ for all $\phi\in C_c(\Omega)$.
This means that $(\mu_m \conv g_m)_m$ converges weakly on the dense subset $C_c(\Omega)$ of $L^{q'}(\Omega)$, which, since $\| \mu_m \conv g_m\|_q$ is bounded by boundedness of both $(\mu_m)_m$ and $(g_m)_m$, implies weak convergence by a standard argument, hence the proof is complete.
\end{proof}

Now we consider the convergence of solutions of \eqref{eq_cont_problem} in the main paper as the data converges to the ground truth data and the regularization parameter $\lambda\rightarrow\infty$ (the parameter $\nu$ is fixed).
 
For this, we use the notion of regularization minimizing solution of $Au=y^\dagger$, with $y^\dagger \in Y$ regarded as ground truth data, for any $(u,v)\in L^q(\Omega)^2$ solving
\[ (u,v) \in \argmin_{\tilde{u},\tilde{v} \in L^q(\Omega)} \Rc(\tilde{u}-\tilde{v}) + \Gc (\tilde{v}) \quad\text{subject to: } A\tilde{u}  = y^\dagger . 
\]
Note that, with the same arguments as in the proof of \cref{thm_existence}, existence of a regularization minimizing solution of $Au=y^\dagger$ can be ensured whenever there are $u,v \in L^q(\Omega)$ such that $Au = y^\dagger$ and $\Rc(u-v) + \Gc (v)< \infty$.

\begin{theorem}[Convergence for vanishing noise]\label{SM_thm_convergence}
Let $y^\dagger \in Y$ be some ground-truth data such that there exist $u,v \in L^q(\Omega)$ with $Au=y^\dagger$ and $\Rc(u-v) + \Gc(v)< \infty$. For a sequence of noise levels $(\delta_m)_m$ in $(0,\infty)$ such that $\lim_{m \rightarrow \infty} \delta_m = 0$, let $y_m \in Y$ be noisy data such that $\Dc_{y_m}(y^\dagger) < \delta_m$. 
 
Suppose that the data fidelity functionals $(\mathcal{D}_{y_m})_m$ are equi-coercive, converge to $\mathcal{D}_{y^\dagger}$ in the sense of \eqref{defin_data_conv} as $m \rightarrow \infty$ and that $\mathcal{D}_{y^\dagger}(z)=0$ if and only if $z=y^\dagger$. Choose the parameters $(\lambda_m)_m$ in $(0,\infty)$ such that 
\[\lambda_m\rightarrow\infty, \quad \lambda_m\delta_m\rightarrow 0 \quad \text{as } m\rightarrow \infty.\]
Then, with $(u_m,v_m)$ a solution to \eqref{eq_cont_problem} with data $y_m$ and parameter $\lambda_m$, up to shifts of $u_m$ in $U\cap\ker(A)$, $(u_m,v_m)_m$ admits a weak accumulation point in $L^q(\Omega)^2$. 
Further, each such accumulation point is a regularization minimizing solution of $Au=y^\dagger$. Moreover, we have 
that $\lim_{m \rightarrow \infty} \Dc_{y_m}(Au_m) = 0$ and $\lim_{m\rightarrow \infty} s_{\Rc}(\nu) \Rc(u_{m} - v_{m}) + s_{\Gc}(\nu) \Gc(v_m) =s_{\Rc}(\nu) \Rc(u^\dagger - v^\dagger) + s_{\Gc}(\nu) \Gc(v^\dagger)$, with $(u^\dagger,v^\dagger)$ a regularization minimizing solution of $Au=y^\dagger$.
\end{theorem}
\begin{proof}
Optimality of $(u_m,v_m)$ compared to a regularization minimizing solution $(u^\dagger,v^\dagger)$ yields
\begin{equation} \label{eq:convergence_noise_estimate}
\lambda_m \Dc_{y_m}(Au_m)+ s_{\Rc}(\nu) \Rc(u_m - v_m) + s_{\Gc}(\nu) \Gc(v_m) \leq \lambda_m \delta_m + s_{\Rc}(\nu) \Rc(u^\dagger - v^\dagger) + s_{\Gc}(\nu) \Gc(v^\dagger).
\end{equation}
Since $\lambda_m \rightarrow \infty$ as $m\rightarrow \infty$, this implies that $\Dc_{y_m}(Au_m) \rightarrow 0$. Further, since $\lambda_m \delta_m \rightarrow 0$, it follows that $((\Rc(u_m - v_m) , \Gc(v_m)))_m$ is bounded. Hence, again, as in \cref{thm_stability}, up to shifts in $U \cap \ker(A)$, we can assume that $((u_m,v_m))_m$ admits a weak accumulation point in  $ L^q(\Omega)^2$. 
 
Now let $(u,v)$ be any such accumulation point and $(u_{m_k},v_{m_k})_k$ weakly converging to $(u,v)$. By convergence of $(\Dc_{y_m})_m$ and weak-to-weak continuity of $A$, we obtain \[ \Dc_{y^\dagger}(Au) \leq \liminf_{k \rightarrow \infty}\Dc_{y_{m_k}}(Au_{m_k}) = 0 ,\] such that $A u = y^\dagger$.  Further, by \eqref{eq:convergence_noise_estimate} and weak lower semicontinuity of $\Rc$ and $\Gc$, we obtain that $s_{\Rc}(\nu) \Rc(u - v) + s_{\Gc}(\nu) \Gc(v) \leq s_{\Rc}(\nu)\Rc(u^\dagger - v^\dagger) + s_{\Gc}(\nu) \Gc(v^\dagger) $, such that $(u,v)$ is a regularization minimizing solution.
 
Finally, by \eqref{eq:convergence_noise_estimate}, since each subsequence of $((u_m,v_m))_m$ admits another (convergent) subsequence $((u_{m_k},v_{m_k}))_k$ such that $\lim_{k\rightarrow \infty} s_{\Rc}(\nu) \Rc(u_{m_k} - v_{m_k}) + s_{\Gc}(\nu) \Gc(v_{m_k}) = \Rc(u^\dagger - v^\dagger) + s_{\Gc}(\nu) \Gc(v^\dagger)$, it follows that
$ \lim_{m\rightarrow \infty} s_{\Rc}(\nu) \Rc(u_{m} - v_{m}) + s_{\Gc}(\nu) \Gc(v_m) =s_{\Rc}(\nu) \Rc(u^\dagger - v^\dagger) + s_{\Gc}(\nu) \Gc(v^\dagger)$, concluding the proof.
\end{proof}

\section{The Discrete Model}\label{supplement_discrete_model}
This section deals with the discretization of the proposed variational energy.
Recall that, in the main paper, the numerical results were obtained with a discretized version of the energy \eqref{eq_cont_problem}, where the following architecture and functionals were used:
\begin{enumerate}[i)]
\item $L = 3$ and $N_l = 8$ for each $l = 1,2,3$. 
\item $E_l = \{ (n,n)\st n \in \{1,\ldots,8\}\}$ for $l=2,3$ and $E_1 = \{(n,1) \st n \in \{1,\ldots,8\}\}$. 
\item $\Rc = J^{**}$, where, for $\epsilon>0$, $J: L^1(\Omega)\rightarrow [0,\infty]$ is given as
\begin{equation}\label{eq_application_cartoon_prior_supp}
J(u) =  \begin{cases}
\frac{1}{|\Omega |}\int\limits_{\Omega}\sqrt{|\nabla u|^2 + \epsilon}\; dx \quad &\text{if } u\in W^{1,1}(\Omega),\\
\infty \quad &\text{else.}
\end{cases}
\end{equation}
\item For $l=1,2$, as detailed in \cref{ex_data_reg}, 
\[\Jc_l(\mu) = \begin{cases}
\frac{\gamma}{2}\| \; f \; \|_2^2\quad &\text{if } \mu = f \in L^2(\Omega^l)^{N_l},\\
\infty\quad &\text{ else.} \end{cases}\]
\end{enumerate}
A summary of all involved parameter can be found in \cref{table_parameters_general}.

In order to elaborate on the discretization in more detail, we start by introducing some notation. For the sake of simplicity, we deal with grayscale images and represent discrete images by matrices $u \in \R^{N_x \times N_y}$ with $N_x,N_y\in\mathbb{N}$. For $1\leq p<\infty$ and discrete variables of arbitrary size $n,m \in \N$, we define the norms
\begin{equation}
\begin{aligned}
\|\;.\;\|_p: \mathbb{R}^{n\times m} &\rightarrow [0,\infty)\\
\|u\|_p &= \left( \frac{1}{n m} \sum\limits_{i=1}^{n} \sum\limits_{j=1}^{m} |
u_{i,j}|^p\right)^\frac{1}{p}.
\end{aligned}
\end{equation}
The division by the number of entries is added in order to make the model parameters independent of the image size and corresponds to a proper discretization of $\| u \|_{L^p(\Omega)} = (\frac{1}{|\Omega|}\int_\Omega |u|^p\;dx)^{\frac{1}{p}}$. For discrete latent variables $\mu\in\mathbb{R}^{M_x \times M_y}$ and filter kernels $\theta\in\mathbb{R}^{r \times r}$, such that $r\leq \min\{M_x,M_y\}$, we define the discrete convolution $\mu*\theta\in\mathbb{R}^{(M_x-r+1) \times (M_y-r+1)}$ as 
\begin{equation*}
\begin{gathered}
\left(\mu * \theta\right)_{n,m} \coloneqq \sum_{i,j=1}^r \mu_{n+r-i,m+r-j}\theta_{i,j}
\end{gathered}
\end{equation*}
for $n=1, 2, ..., M_x-r+1$ and $m = 1, 2, ..., M_y-r+1$. Further, for a stride $\sigma\in\mathbb{N}$ and $\tilde{M_x}, \tilde{M_y}\in\mathbb{N}$, such that $\sigma(M_x-1)+1\leq \tilde{M}_x \leq \sigma M_x$ and $\sigma(M_y-1)+1\leq \tilde{M}_y \leq \sigma M_y$, we define the strided upconvolution $\mu*_\sigma\theta$ as 
\begin{equation}\label{eq:upconvolution}
\begin{gathered}
\mu *_\sigma \theta \coloneqq \tilde{\mu}^\sigma * \theta,
\end{gathered}
\end{equation}
where $\tilde{\mu}^\sigma\in\mathbb{R}^{\tilde{M}_x \times \tilde{M}_y}$ is a zero interpolation defined as
\begin{equation*}
\begin{gathered}
\tilde{\mu}^\sigma_{i,j} \coloneqq 
\begin{cases}
\mu_{k,l} & \text{if } \exists k, l:\; i=\sigma (k-1) +1, j=\sigma  (l-1) +1, \\
0 & \text{else.}
\end{cases}
\end{gathered}
\end{equation*}
The strided upconvolution realizes an increase of resolution of the latent variables within the convolutions between different layers and thereby reduces dimensionality of the network. Further, we define the discrete gradient as
\begin{equation*}
\begin{gathered}
\nabla : \mathbb{R}^{N_x\times N_y} \rightarrow \mathbb{R}^{N_x\times N_y \times 2}\\
u \mapsto \nabla u,
\end{gathered}
\end{equation*}
with
\begin{equation*}
(\nabla u)_{i,j}^1\coloneqq
\begin{cases}
u_{i+1,j}-u_{i,j} & \text{for } i<N_x,\\
0 & \text{for } i=N_x,
\end{cases} \quad 
(\nabla u)_{i,j}^2\coloneqq
\begin{cases}
u_{i,j+1}-u_{i,j} & \text{for } j<N_y,\\
0 & \text{for } j=N_y.
\end{cases}
\end{equation*}
We introduce ${\TV}_\epsilon$ as a discretization of \eqref{eq_application_cartoon_prior_supp} via
\begin{equation}
  \begin{aligned}
  {\TV}_\epsilon: \mathbb{R}^{N_x \times N_y} &\rightarrow [0,\infty ) \\
  u &\mapsto \frac{1}{N_xN_y} \sum_{i=1}^{N_x}\sum_{j=1}^{N_y} \sqrt{((\nabla u)_{i,j}^1)^2+((\nabla u)_{i,j}^2)^2+\epsilon}.
  \end{aligned}
\end{equation}

Now, let $A:\R^{N_x \times N_y} \rightarrow Y$ be a discrete forward operator, where now $Y=\R^{K}$ is also finite dimensional. We denote the sought ground truth as $u^\dagger\in\R^{N_x \times N_y}$ and the corresponding ground truth data $y^\dagger = Au^\dagger \in Y$. Further, we write $y = y^\dagger + \eta$ for the given, possibly noisy, data, where $\eta$ denotes the noise. With a discrepancy functional $\Dc_{y}: Y \rightarrow [0,\infty]$,  a discrete version of the minimization problem \eqref{eq_cont_problem} in order to obtain an approximation of $u^\dagger$ is given as
\begin{equation}\label{eq_discrete_problem}
  \tag{$G$-reg}
\begin{aligned}
\min\limits_{u,\mu,\theta} \Ec^D_{y}(u,\mu,\theta) \quad \text{ with }\quad \Ec^D_{y}(u,\mu,\theta) 
 := & \lambda\mathcal{D}_{y}(Au)+ s_\Rc(\nu) {\TV}_\epsilon( u-\sum_{n=1}^{N_1} \mu^1_{n}*_\sigma\theta^1_{n})\\
& + s_\mathcal{G}(\nu)\sum\limits_{l=1}^L \sum_{n=1}^{N_l} \norm{\mu^l_n}_1   + \gamma\sum\limits_{l=2}^L \sum\limits_{n=1}^{N_l} \frac{1}{2}\left\| \mu^{l-1}_{n} -\mu^l_{n}*_\sigma\theta^l_n\right\|_2^2,
\end{aligned}
\end{equation}
subject to
\begin{equation*}
  \begin{cases}
  \sum\limits_{i,j=1}^{r} |(\theta^l_n)_{i,j}|^2 \leq 1 \quad \text{for all $l,n$}, \\
  \sum\limits_{i,j=1}^{r} (\theta^1_n)_{i,j} = 0, \quad \text{for all $n$}.
  \end{cases}
\end{equation*}
Recall that the functions $s_\Rc$ and $s_\mathcal{G}$ are defined as
\[s_\Rc(\nu)= \frac{\nu}{\min(\nu, 1-\nu)},\quad s_\mathcal{G}(\nu)= \frac{1-\nu}{\min(\nu, 1-\nu)}.\]
The data fidelity term $\mathcal{D}_{y}$ and the forward operator $A$ depend on the specific application and are defined in the corresponding subsections of \Cref{supplement_experiments}, where, for algorithmic purposes, $\mathcal{D}_{y}$ is split as $\mathcal{D}_{y} = \mathcal{D}^1_{y} + \mathcal{D}^2_{y}$ with one of them always being zero.

\section{The Algorithm}\label{supplement_algorithm}
In order to solve \eqref{eq_discrete_problem}, we employ the iPALM algorithm \cite{Pock_2016}. To this aim, we distinguish different blocks of variables, which are updated separately in the algorithm, i.e., we define $x_1 \coloneqq u$, $x_2 \coloneqq \mu$ and $x_{3} \coloneqq \theta$. Further, we split the objective functional $\Ec^D_{y}$ into different parts as
\[\Ec^D_{y}(u,\mu,\theta) = H(u,\mu,\theta)+f_1(u)+f_2(\mu)+f_3(\theta),\]
where $H $ will be continuously differentiable and the functions $f_i$, $i=1,2,3$, will be proper, lower semicontinuous and convex and such that $\Ec^D_{y}$ is a Kurdyka \L{}ojasiewicz (KL) function (for more details on KL functions, see, e.g. \cite{kl_functions}).

More concretely, we set
\[ H(u,\mu,\theta) = \lambda\mathcal{D}^1_{y}(Au) + s_\Rc(\nu) {\TV}_\epsilon( u-\sum_{n=1}^{N_1} \mu^1_{n}*_\sigma\theta^1_{n}) + \gamma\sum\limits_{l=2}^L \sum\limits_{n=1}^{N_l} \frac{1}{2}\left\| \mu^{l-1}_{n} -\mu^l_{n}*_\sigma\theta^l_n\right\|_2^2,\]
and
\[
f_1(u)  = \lambda\mathcal{D}^2_{y}(Au),\quad
f_2(\mu)  =  s_\mathcal{G}(\nu)\sum\limits_{l=1}^L \sum_{n=1}^{N_l} \norm{\mu^l_n}_1, \quad
f_3(\theta) = \mathcal{I}_\mathcal{A}(\theta),
\]
with
\[ \mathcal{A} = \left\{ \theta \; \middle|\;  \sum\limits_{i,j=1}^{r} |(\theta^l_n)_{i,j}|^2 \leq 1 \quad \text{for all $l,n$ and} \sum\limits_{i,j=1}^{r} (\theta^1_n)_{i,j} = 0, \text{ for all $n$}\right\},\]
where, again, the choices of $\Dc_{y}^1$, $\Dc_{y}^2$ and $A$ depend on the concrete application and will be specified in the respective subsections of \Cref{supplement_experiments}.

Using this notation, the algorithm reads as follows.
\begin{algorithm}[H]
\caption{ \textsc{iPALM}} \label{algo_general}
\begin{algorithmic}[1]
\State \textbf{initialize}: $x^0 = (x^0_1, x^0_2, x^0_3)$
 \For{$m=1, 2, ...$:}
  \For{$i=1, 2, 3$:}
    \State Take $\alpha^m_i, \beta^m_i \in [0,1]$ and $\tau_i^m>0$ and compute
\begin{equation*}
\begin{gathered}
y_i^m = x_i^m + \alpha_i^m (x_i^m-x_i^{m-1}), \\
z_i^m = x_i^m + \beta_i^m (x_i^m-x_i^{m-1}), \\
\end{gathered}
\end{equation*}
\begin{equation}\label{eq_algostep}
x_i^{m+1} \in {\prox}_{\tau_i^m}^{f_i} \left( y_i^m-\frac{1}{\tau_i^m}\nabla_{x_i} H(x_1^{m+1}, ..., x_{i-i}^{m+1}, z_i^m, x_{i+1}^m, ..., x_n^m) \right) .
\end{equation}
\EndFor
\EndFor
\end{algorithmic}
\end{algorithm}

According to \cite{Pock_2016}, convergence of the algorithm in case of bounded iterates can be ensured by choosing the algorithmic parameters $(\alpha_i^m)_{i,m}$, $(\beta_i^m)_{i,m}$, $(\tau_i^m)_{i,m}$, and $(\delta_i)_i$ according to the following conditions.
\begin{enumerate}
\item\label{item_eps_algo} For $\epsilon\in (0,1)$, there are $\overline{\alpha_i}$, $\overline{\beta_i} \in (0,1-\epsilon)$, such that for all $i=1, 2, 3$ and $m\in\mathbb{N}$, $0\leq \alpha_i^m \leq \overline{\alpha_i}$ and $0\leq \beta_i^m \leq \overline{\beta_i}$.
\item\label{item_step} Using $\epsilon$ from \eqref{item_eps_algo} above, define, for $i=1, 2, 3$,
\begin{equation*}
\begin{aligned}
\delta_i & = \frac{\overline{\alpha_i}+2\overline{\beta_i}}{2(1-\epsilon-\overline{\alpha_i})}L_i, \\
\tau_i^m & = \frac{(1+\epsilon)\delta_i + (1+\beta_i^m)L_i^m}{2-\alpha_i^m},
\end{aligned}
\end{equation*}
where $0<L_i^m\leq L_i$ for all $m$ and $L_i^m$ satisfies the following properties:
\begin{enumerate}
\item Descent property:
\begin{multline*}
H(x_1^{m+1}, ..., x_{i-i}^{m+1}, x_i^{m+1}, x_{i+1}^m, ..., x_n^m)  
\leq   H(x_1^{m+1}, ..., x_{i-i}^{m+1}, x_i^{m}, x_{i+1}^m, ..., x_n^m) \\
 + \langle x_i^{m+1}-x_i^m , \nabla_{x_i} H(x_1^{m+1}, ..., x_{i-i}^{m+1}, x_i^{m}, x_{i+1}^m, ..., x_n^m) \rangle 
 + \frac{L_i^m}{2}\norm{x_i^{m+1}-x_i^m}_2^2
\end{multline*}
\item Lipschitz property:
\begin{multline*}
\norm{\nabla_{x_i} H(x_1^{m+1},.., x_{i-i}^{m+1}, z_i^{m}, x_{i+1}^m,.., x_n^m)-\nabla_{x_i} H(x_1^{m+1},.., x_{i-i}^{m+1}, x_i^{m}, x_{i+1}^m,.., x_n^m) }_2 \\
\leq L_i^m \norm{ z_i^m-x_i^m }_2
\end{multline*}
\end{enumerate}
\end{enumerate}

\begin{remark}\
\begin{itemize}
\item Note that in the descent and Lipschitz property above, the norm $\|\cdot \|_2$ is computed solely as the square root of the sum of all squared entries, without division by the number of entries.
\item In order to decrease computational effort, inspired by the experiments in \cite{Pock_2016}, in practice we replace $L_i$ by $L^m_i$ in the computation of $\delta_i$ in \eqref{item_step}.
\item To determine $L^m_i$ satisfying the descent and Lipschitz property, we use a backtracking scheme, which iteratively increases $L^m_i$ until both properties are satisfied.
\item Concrete choices of the algorithmic parameters are summarized in \cref{table_parameters_general}. 
For that, note that we used a rather large number of iterations, in order to ensure optimality, but the objective functional usually does not change much after around 2000 iterations, see \cref{fig:obj_func}, suggesting that also a lower number of iterations
would be sufficient.
\end{itemize}
\end{remark}

\begin{figure}
\centering
\includegraphics[scale=0.5]{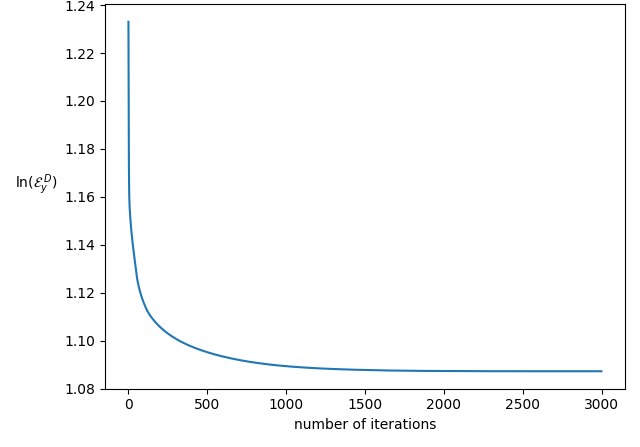}
\caption{Objective functional for denoising with the \emph{fish} image.\label{fig:obj_func}}
\end{figure}

For all choices of $\Dc_{y}^i$, $i=1,2$ and $A$, in our experiments, the gradient of $H$ and the proximal mappings of $f_i$ can be computed explicitly and rather easily, therefore we omit those computations and rather refer to the publicly available source code \cite{gen_reg_git}. For the latter, we implemented our method in Python with a parallelization for GPUs based on  PyOpenCL \cite{klockner2012pycuda}.

\noindent\textbf{Initialization.} Due to the non-convexity of our objective functional, the initialization of the algorithm has a non-negligible impact on the solution. We found that this issue becomes more apparent with deeper networks where, with a bad initialization, often large parts of the network become trivial, i.e., many latent variables are zero. As a remedy, we use an initialization strategy, that successively increases the network depth during the first iterations of the algorithm. That is, we start with a one layer network, compute some iterations, then increase the network depth by one layer and repeat this until the desired network depth is obtained. When increasing the network depth, say from $L$ to $L+1$, we inherit the already used kernels of the network and initialize the new kernels $\theta^{L+1}$ randomly. The new latent variables, $\mu^{L+1}$ are initialized as the adjoint operator of $\mu^{L+1}\mapsto\mu^{L+1}*_\sigma\theta^{L+1}$ applied to $\mu^L$. We further set the latent variables $\mu^l_n$ for $l<L+1$ to $\mu^{l+1}_n*_\sigma\theta^{l+1}_n$ to ensure that, at the beginning of the new iterations, the functionals $\Jc_l$ evaluate to zero. The effect of this initialization can be observed in \cref{fig_network_supp}.

\begin{figure}
\centering
\includegraphics[height = 7.5cm]{images_gen_reg_denoising_cart_text_mix_network_lines.png}%
\includegraphics[height = 7.5cm]{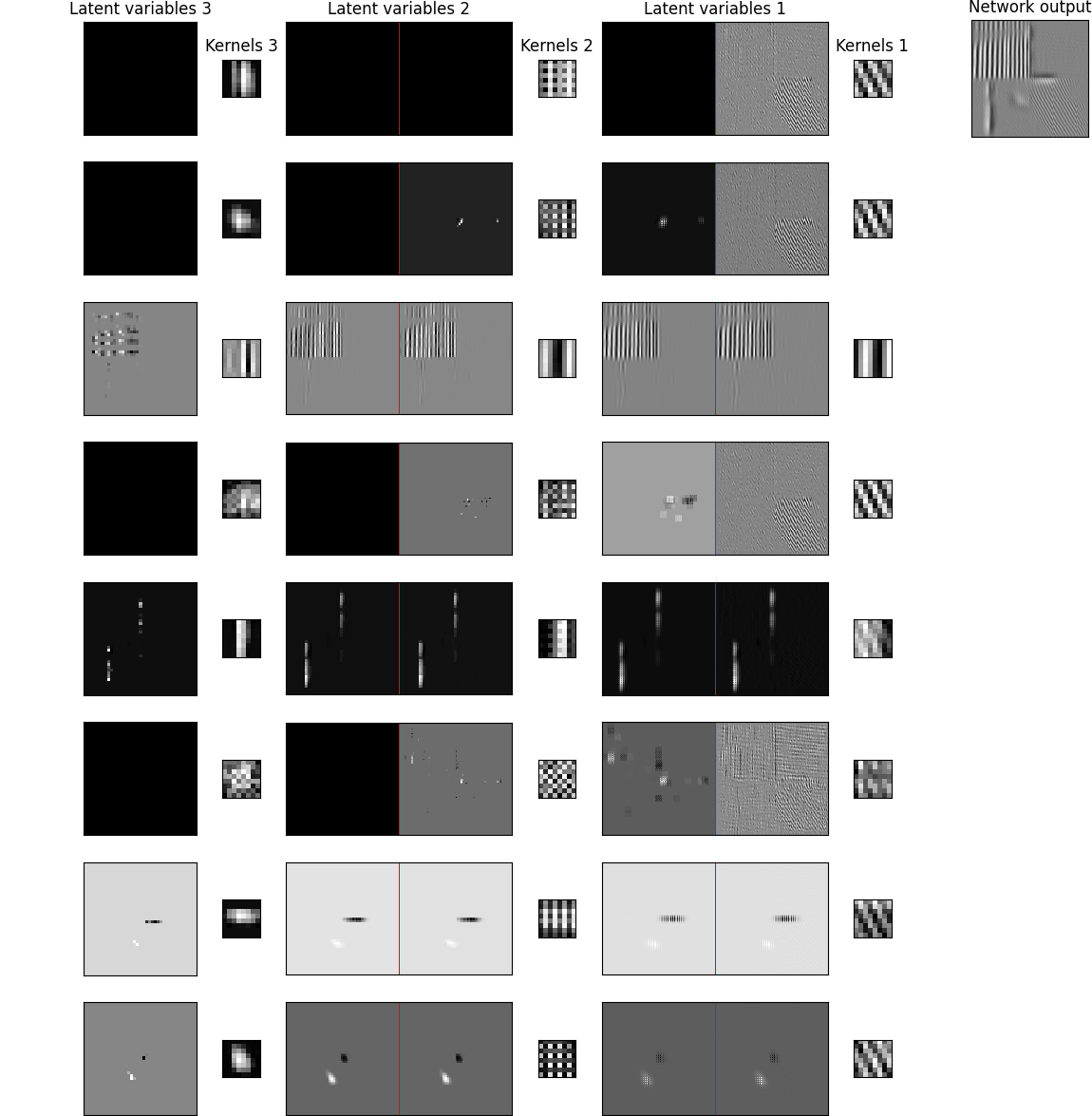}
\caption{
Convolutional networks obtained from denoising of the \emph{mix} image with the proposed method, using two different initialization strategies. Left: proposed initialization strategy that successively increases the network depth. Right: Directly starting with the full network depth.}
\label{fig_network_supp}
\end{figure}

\section{Numerical Results}\label{supplement_experiments}

In this section, we provide some additional experimental results with the \emph{zebra} and \emph{patchtest} images, as well as further details on the experimental setup and the comparison to related methods, which we omitted in the main paper. In particular, the following subsections provide a summary of the respective choices of $\Dc_{y}^i$, $i=1,2$, $A$ and the noise $\eta$ for all experiments. The latter was partly covered in the main paper, but we will repeat it here for the sake of readability and completeness.

\noindent\textbf{Parameter choice and comparison to related methods.} 
Leaving the architecture of our generative prior as in \cref{table_parameters_general} fixed for all experiments, our method requires the choice of one or two parameters, depending on the data discrepancy used. In order to ensure a fair comparison, only a single parameter of our method was optimized in each type of experiment (according to visual quality of the results), while the other was left fixed, if applicable. A summary of our parameter choice is provided in \cref{table_parameters}.

For the comparisons to DIP and CL we use the codes the authors have made publicly available, where for CL we use the convex version of the model and the parameters as suggested by the authors (parameters were chosen for optimal visual results). Since the DIP uses different architectures for different applications, for the experiments with the DIP, we always use an architecture, that was used in \cite{deep_image_prior} for the same type of experiment (inpainting, denoising, etc.). The exact choices are mentioned later in the corresponding subsections for the different experiments. When necessary (that is, whenever overfitting is an issue), we further optimize the number of iterations of the DIP, see \cref{table_numiter_dip}, trying to obtain visually optimal results. For the experiments with TGV regularization, we use a custom implementation and results were optimized over the single regularization parameter to also obtain visually optimal results.
 
In the DIP code, the images are cropped at the beginning, such that their dimensions are multiples of 32 or 64 which changes the sizes of the \emph{fish} and \emph{zebra} images. For the \emph{patchtest} image, we omit the cropping, since it would remove a substantial part of the image. We also point out that, in the DIP code, the data discrepancy functional is always the squared 2-norm, i.e., $\Dc_{y}(z) = \|z-y\|_2^2$, since the method relies on differentiability. In contrast to that, in our method we use also non-smooth data discrepancies, namely indicator functions, in the cases of inpainting, super-resolution and JPEG decompression.

\begin{table}[h]
\caption[Parameters]{List of all parameters involved in our model and its algorithmic realization}
\centering
\begin{tabular}{ p{4cm} p{6cm} }
 \toprule
 \multicolumn{2}{ c }{Parameter choice in \eqref{eq_discrete_problem}}\\
 \midrule
 $\epsilon$   & 0.05\\
 Network depth $L$ & 3\\
 Filter depth $(N_1,N_2,N_3)$ & (8,8,8)\\
 Connections & $E_1 = \{(n,1) \st n \in \{1,\ldots,8\}\}$ \\ 
  & $E_l = \{ (n,n)\st n \in \{1,\ldots,8\}\}$, $l=2,3$ \\
 Kernel size $r$    & 8\\
 Stride $\sigma$ & $\sigma=1$ in $1^\text{st}$ layer \\
  & $\sigma = 2$ in remaining layers\\
 $\gamma$ & $2000$\\
 $\lambda$ & application-dependent, see \cref{table_parameters}\\
 $\nu$ & application-dependent, see \cref{table_parameters}\\
 \midrule
 \multicolumn{2}{ c }{Parameter choice in \cref{algo_general}}\\
 \midrule
 $\epsilon$ & 0.03 \\
 $\beta_1^m,\overline{\beta_1},\beta_2^m,\overline{\beta_2},\beta_3^m,\overline{\beta_3}$   & 0.7\\
 $\alpha_1^m,\overline{\alpha_1},\alpha_2^m,\overline{\alpha_2},\alpha_3^m,\overline{\alpha_3}$   & 0.7\\
 Number of iterations & 8000\\
 \bottomrule
\end{tabular}
\label{table_parameters_general}
\end{table}

\begin{table}[h]
\caption[Regularization parameters]{Parameter choices for the regularization parameters $\lambda$ and $\nu$ for all experiments.}
\centering
\begin{tabular}{ p{4cm} p{1.5cm} p{1.5cm} p{1.5cm} p{1.5cm} p{1.5cm}}
\toprule
 & \emph{Barbara} & \emph{mix} & \emph{patchtest} & \emph{fish} & \emph{zebra}\\
 \midrule
 \textbf{Inpainting}\\
 $\nu$        & 0.975       & 0.975       & 0.975       & 0.925 & 0.975 \\
 
 \textbf{Denoising}\\
  $\nu = 0.925$\\
  $\lambda$       & 22.5      & 20.0 &      20.0    & 30.0      & 30.0\\
 
 \textbf{Deconvolution}\\
  $\nu = 0.925$\\
  $\lambda$       & 600       & 700 &     800   & 500       & 1000\\

 \textbf{Super-resolution}\\
  $\nu$ & 0.975 & - & - & 0.95 & 0.9\\
 
 \textbf{JPEG decompression}\\
  $\nu$ & 0.875 & 0.875 & 0.95 & 0.85 & 0.75\\
 \bottomrule
\end{tabular}
\label{table_parameters}
\end{table}

\begin{table}[h]
\caption{Number of iterations used for DIP results.}
\centering
\begin{tabular}{ p{4cm} p{1.5cm} p{1.5cm} p{1.5cm} p{1.5cm} p{1.5cm}}
\toprule
 & \emph{Barbara} & \emph{mix} & \emph{patchtest} & \emph{fish} & \emph{zebra}\\
 \midrule
 \textbf{Denoising} & 700       & 700       & 600       & 1500 & 2500 \\
 \textbf{JPEG decompression} & 6000 & 10200 & 2000 & 8500 & 14900\\
 \bottomrule
\end{tabular}
\label{table_numiter_dip}
\end{table}

\subsection{Inpainting}
In the case of inpainting, we assume that some pixels of an original image $u^\dagger$ are known exactly, while the remaining ones are unknown. Given an inpainting mask $\mathcal{M}\in \{0,1\}^{N_x\times N_y}$, such that $\mathcal{M}_{i,j}=1$ whenever $u^\dagger_{i,j}$ is known and $\mathcal{M}_{i,j}=0$ else, we define
\begin{equation*}
Au = \mathcal{M} \odot u, \quad \Dc^1_{y} \equiv 0, \quad \Dc^2_{y} (z) = \Ic_{\{ y \}},
\end{equation*}
where $\odot$ denotes the Hadamard product. In this case, the noise $\eta$ is zero such that $y = y^\dagger = Au^\dagger$. For our experiments, we used a randomly generated inpainting mask $\mathcal{M}$, where around $70\%$ of the entries were set to zero. 
For all inpainting experiments with the DIP, we use the architecture as proposed by the authors for the inpainting experiment shown in \cite[Figure 7]{deep_image_prior}, and set the number of iterations to 8000. For CL, the parameter as proposed by the authors was used, and TGV regularization, in this case, does not require a regularization parameter. A summary of all results for inpainting can be found in \cref{fig_inpainting_supplement}.

\newcommand\fw{2.5cm} 
\newcommand\fwh{1.25cm} 

\newcommand\ffw{3cm} 
\newcommand\ffwh{1.5cm} 

\begin{figure}
\centering
\begin{subfigure}{\textwidth}
\centering
\begin{subfigure}[t]{\fw}
\includegraphics[width = \fw]{images_gen_reg_inpainting_barbara_crop_corrupted.png}
\end{subfigure}%
\begin{subfigure}[t]{\fw}
\includegraphics[width = \fw]{images_tgv_inpainting_barbara_crop_recon_closeups.png}
\includegraphics[width = \fwh]{images_tgv_inpainting_barbara_crop_recon_detail_1.png}%
\includegraphics[width = \fwh]{images_tgv_inpainting_barbara_crop_recon_detail_2.png}
\end{subfigure}%
\begin{subfigure}[t]{\fw}
\includegraphics[width = \fw]{images_convex_learning_inpainting_barbara_crop_recon_closeups.png}
\includegraphics[width = \fwh]{images_convex_learning_inpainting_barbara_crop_recon_detail_1.png}%
\includegraphics[width = \fwh]{images_convex_learning_inpainting_barbara_crop_recon_detail_2.png}
\end{subfigure}%
\begin{subfigure}[t]{\fw}
\includegraphics[width = \fw]{images_ulyanov_inpainting_barbara_crop_recon_closeups.png}
\includegraphics[width = \fwh]{images_ulyanov_inpainting_barbara_crop_recon_detail_1.png}%
\includegraphics[width = \fwh]{images_ulyanov_inpainting_barbara_crop_recon_detail_2.png}
\end{subfigure}%
\begin{subfigure}[t]{\fw}
\includegraphics[width = \fw]{images_gen_reg_inpainting_barbara_crop_recon_closeups.png}
\includegraphics[width = \fwh]{images_gen_reg_inpainting_barbara_crop_recon_detail_1.png}%
\includegraphics[width = \fwh]{images_gen_reg_inpainting_barbara_crop_recon_detail_2.png}
\end{subfigure}%
\begin{subfigure}[t]{\fw}
\includegraphics[width = \fw]{images_gen_reg_inpainting_barbara_crop_original_closeups.png}
\includegraphics[width = \fwh]{images_gen_reg_inpainting_barbara_crop_original_detail_1.png}%
\includegraphics[width = \fwh]{images_gen_reg_inpainting_barbara_crop_original_detail_2.png}
\end{subfigure}
\end{subfigure}

\begin{subfigure}{\textwidth}
\centering
\includegraphics[width = \fw]{images_gen_reg_inpainting_cart_text_mix_corrupted.png}%
\includegraphics[width = \fw]{images_tgv_inpainting_cart_text_mix_recon.png}%
\includegraphics[width = \fw]{images_convex_learning_inpainting_cart_text_mix_recon.png}%
\includegraphics[width = \fw]{images_ulyanov_inpainting_cart_text_mix_recon.png}%
\includegraphics[width = \fw]{images_gen_reg_inpainting_cart_text_mix_recon.png}%
\includegraphics[width = \fw]{images_gen_reg_inpainting_cart_text_mix_original.png}
\end{subfigure}

\begin{subfigure}{\textwidth}
\centering
\includegraphics[width = \fw]{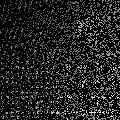}%
\includegraphics[width = \fw]{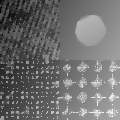}%
\includegraphics[width = \fw]{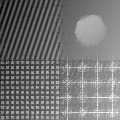}%
\includegraphics[width = \fw]{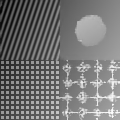}%
\includegraphics[width = \fw]{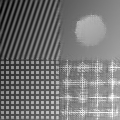}%
\includegraphics[width = \fw]{images_gen_reg_inpainting_patchtest_original.png}
\end{subfigure}

\begin{subfigure}{\textwidth}
\centering
\begin{subfigure}[t]{\ffw}
\includegraphics[width = \ffw]{images_gen_reg_inpainting_fish_corrupted.png}
\end{subfigure}%
\begin{subfigure}[t]{\ffw}
\includegraphics[width = \ffw]{images_tgv_inpainting_fish_recon_closeups.png}
\includegraphics[width = \ffwh]{images_tgv_inpainting_fish_recon_detail_1.png}%
\includegraphics[width = \ffwh]{images_tgv_inpainting_fish_recon_detail_2.png}
\end{subfigure}%
\begin{subfigure}[t]{\ffw}
\includegraphics[width = \ffw]{images_ulyanov_inpainting_fish_recon_closeups.png}
\includegraphics[width = \ffwh]{images_ulyanov_inpainting_fish_recon_detail_1.png}%
\includegraphics[width = \ffwh]{images_ulyanov_inpainting_fish_recon_detail_2.png}
\end{subfigure}%
\begin{subfigure}[t]{\ffw}
\includegraphics[width = \ffw]{images_gen_reg_inpainting_fish_recon_closeups.png}
\includegraphics[width = \ffwh]{images_gen_reg_inpainting_fish_recon_detail_1.png}%
\includegraphics[width = \ffwh]{images_gen_reg_inpainting_fish_recon_detail_2.png}
\end{subfigure}%
\begin{subfigure}[t]{\ffw}
\includegraphics[width = \ffw]{images_gen_reg_inpainting_fish_original_closeups.png}
\includegraphics[width = \ffwh]{images_gen_reg_inpainting_fish_original_detail_1.png}%
\includegraphics[width = \ffwh]{images_gen_reg_inpainting_fish_original_detail_2.png}
\end{subfigure}
\end{subfigure}

\begin{subfigure}{\textwidth}
\centering
\includegraphics[width = \ffw]{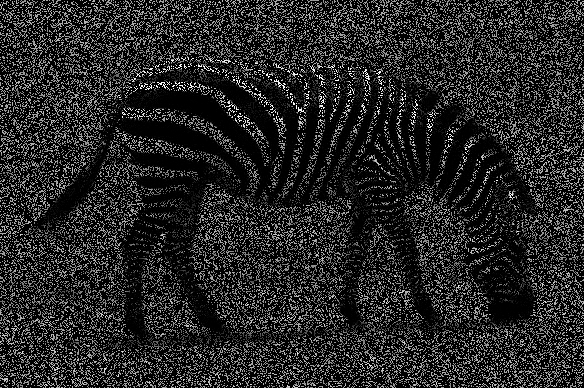}%
\includegraphics[width = \ffw]{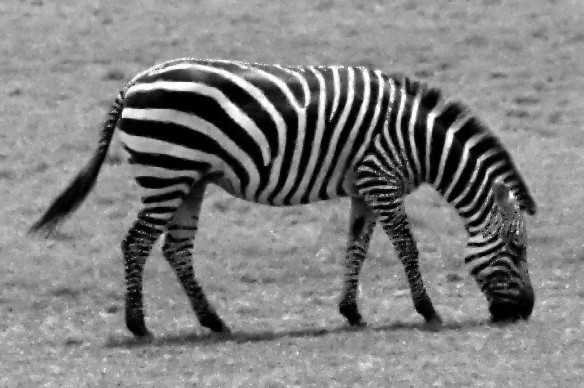}%
\includegraphics[width = \ffw]{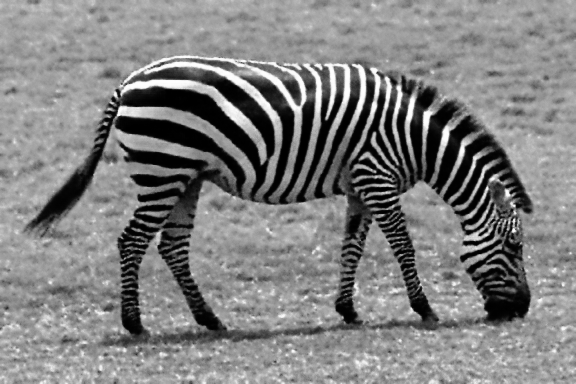}%
\includegraphics[width = \ffw]{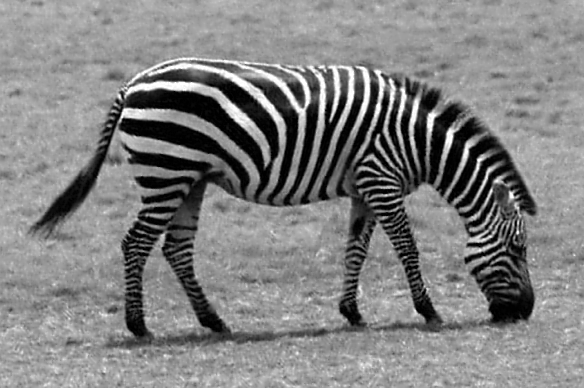}%
\includegraphics[width = \ffw]{images_gen_reg_inpainting_zebra_original.png}
\end{subfigure}
\caption{Inpainting from 30\% known pixels with closeups. First three rows, from left to right: Data, TGV, CL, DIP, proposed, ground truth. Last two rows, from left to right: Data, TGV, DIP, proposed, ground truth.}
\label{fig_inpainting_supplement}
\end{figure}

\subsection{Denoising}
In the case of denoising, the forward operator $A$ is set to be the identity. The data fidelity term is given as 
\[ \Dc^1_{y}(z) = \frac{1}{2} \| z-y\|_2^2, \quad \Dc^2_{y} \equiv 0. \]
For denoising, the data is corrupted with Gaussian noise $\eta$ with mean zero and standard deviation 0.1 times the range of $y^\dagger$. 

For the DIP we transformed the corrupted gray scale image to obtain an RGB image with three identical channels, which we then used as data for the algorithm. We used the architecture as proposed by the authors for the denoising experiment corresponding to \cite[Figure 4]{deep_image_prior}, and optimized the number of iterations for each image, see \cref{table_numiter_dip}. For CL again the setting as proposed by the authors was used, and for TGV the regularization parameter was chosen to achieve optimal visual image quality. A summary of all results for denoising can be found in \cref{fig_denoising_supplement}.

\renewcommand\fw{2.5cm} 
\renewcommand\fwh{1.25cm} 

\renewcommand\ffw{3cm} 
\renewcommand\ffwh{1.5cm}

\begin{figure}
\centering
\begin{subfigure}{\textwidth}
\centering
\begin{subfigure}[t]{\fw}
\includegraphics[width = \fw]{images_gen_reg_denoising_barbara_crop_corrupted_closeups.png}
\includegraphics[width = \fwh]{images_gen_reg_denoising_barbara_crop_corrupted_detail_1.png}%
\includegraphics[width = \fwh]{images_gen_reg_denoising_barbara_crop_corrupted_detail_2.png}
\end{subfigure}%
\begin{subfigure}[t]{\fw}
\includegraphics[width = \fw]{images_tgv_denoising_barbara_crop_recon_closeups.png}
\includegraphics[width = \fwh]{images_tgv_denoising_barbara_crop_recon_detail_1.png}%
\includegraphics[width = \fwh]{images_tgv_denoising_barbara_crop_recon_detail_2.png}
\end{subfigure}%
\begin{subfigure}[t]{\fw}
\includegraphics[width = \fw]{images_convex_learning_denoising_barbara_crop_recon_closeups.png}
\includegraphics[width = \fwh]{images_convex_learning_denoising_barbara_crop_recon_detail_1.png}%
\includegraphics[width = \fwh]{images_convex_learning_denoising_barbara_crop_recon_detail_2.png}
\end{subfigure}%
\begin{subfigure}[t]{\fw}
\includegraphics[width = \fw]{images_ulyanov_denoising_barbara_crop_recon_closeups.png}
\includegraphics[width = \fwh]{images_ulyanov_denoising_barbara_crop_recon_detail_1.png}%
\includegraphics[width = \fwh]{images_ulyanov_denoising_barbara_crop_recon_detail_2.png}
\end{subfigure}%
\begin{subfigure}[t]{\fw}
\includegraphics[width = \fw]{images_gen_reg_denoising_barbara_crop_recon_closeups.png}
\includegraphics[width = \fwh]{images_gen_reg_denoising_barbara_crop_recon_detail_1.png}%
\includegraphics[width = \fwh]{images_gen_reg_denoising_barbara_crop_recon_detail_2.png}
\end{subfigure}%
\begin{subfigure}[t]{\fw}
\includegraphics[width = \fw]{images_gen_reg_denoising_barbara_crop_original_closeups.png}
\includegraphics[width = \fwh]{images_gen_reg_denoising_barbara_crop_original_detail_1.png}%
\includegraphics[width = \fwh]{images_gen_reg_denoising_barbara_crop_original_detail_2.png}
\end{subfigure}
\end{subfigure}

\begin{subfigure}{\textwidth}
\centering
\includegraphics[width = \fw]{images_gen_reg_denoising_cart_text_mix_corrupted.png}%
\includegraphics[width = \fw]{images_tgv_denoising_cart_text_mix_recon.png}%
\includegraphics[width = \fw]{images_convex_learning_denoising_cart_text_mix_recon.png}%
\includegraphics[width = \fw]{images_ulyanov_denoising_cart_text_mix_recon.png}%
\includegraphics[width = \fw]{images_gen_reg_denoising_cart_text_mix_recon.png}%
\includegraphics[width = \fw]{images_gen_reg_denoising_cart_text_mix_original.png}
\end{subfigure}
\begin{subfigure}{\textwidth}
\centering
\includegraphics[width = \fw]{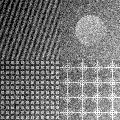}%
\includegraphics[width = \fw]{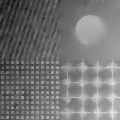}%
\includegraphics[width = \fw]{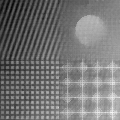}%
\includegraphics[width = \fw]{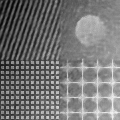}%
\includegraphics[width = \fw]{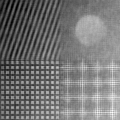}%
\includegraphics[width = \fw]{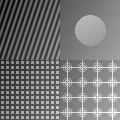}
\end{subfigure}

\begin{subfigure}{\textwidth}
\centering
\begin{subfigure}[t]{\ffw}
\includegraphics[width = \ffw]{images_gen_reg_denoising_fish_corrupted_closeups.png}
\includegraphics[width = \ffwh]{images_gen_reg_denoising_fish_corrupted_detail_1.png}%
\includegraphics[width = \ffwh]{images_gen_reg_denoising_fish_corrupted_detail_2.png}
\end{subfigure}%
\begin{subfigure}[t]{\ffw}
\includegraphics[width = \ffw]{images_tgv_denoising_fish_recon_closeups.png}
\includegraphics[width = \ffwh]{images_tgv_denoising_fish_recon_detail_1.png}%
\includegraphics[width = \ffwh]{images_tgv_denoising_fish_recon_detail_2.png}
\end{subfigure}%
\begin{subfigure}[t]{\ffw}
\includegraphics[width = \ffw]{images_ulyanov_denoising_fish_recon_closeups.png}
\includegraphics[width = \ffwh]{images_ulyanov_denoising_fish_recon_detail_1.png}%
\includegraphics[width = \ffwh]{images_ulyanov_denoising_fish_recon_detail_2.png}
\end{subfigure}%
\begin{subfigure}[t]{\ffw}
\includegraphics[width = \ffw]{images_gen_reg_denoising_fish_recon_closeups.png}
\includegraphics[width = \ffwh]{images_gen_reg_denoising_fish_recon_detail_1.png}%
\includegraphics[width = \ffwh]{images_gen_reg_denoising_fish_recon_detail_2.png}
\end{subfigure}%
\begin{subfigure}[t]{\ffw}
\includegraphics[width = \ffw]{images_gen_reg_denoising_fish_original_closeups.png}
\includegraphics[width = \ffwh]{images_gen_reg_denoising_fish_original_detail_1.png}%
\includegraphics[width = \ffwh]{images_gen_reg_denoising_fish_original_detail_2.png}
\end{subfigure}
\end{subfigure}

\begin{subfigure}{\textwidth}
\centering
\includegraphics[width = \ffw]{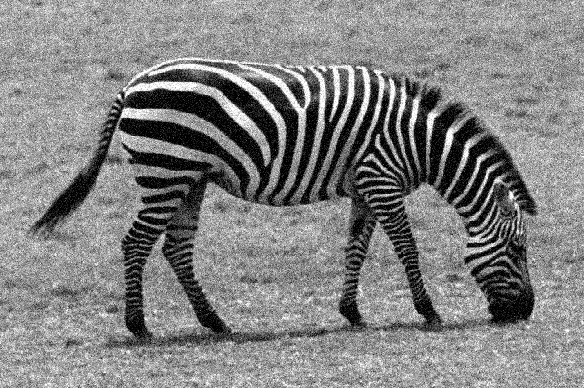}%
\includegraphics[width = \ffw]{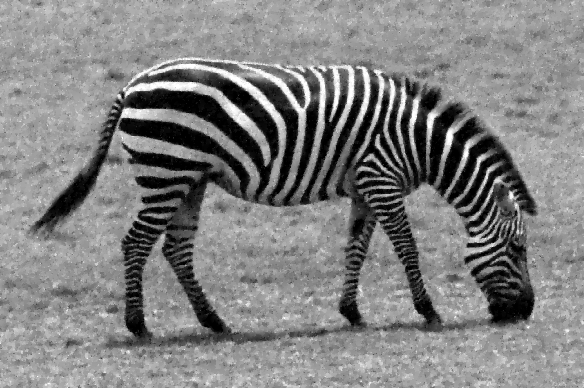}%
\includegraphics[width = \ffw]{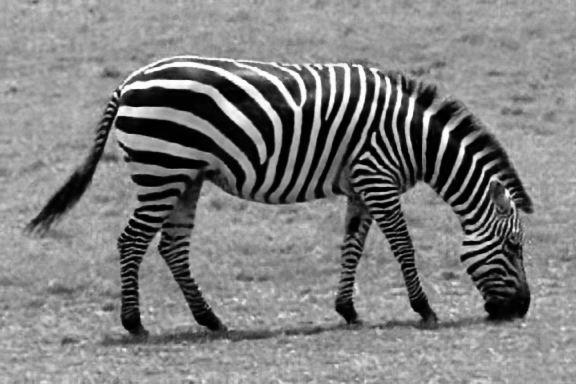}%
\includegraphics[width = \ffw]{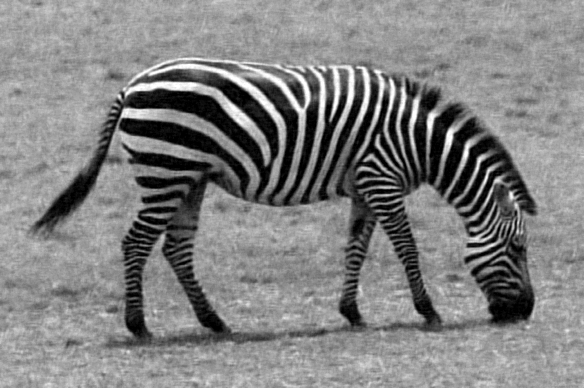}%
\includegraphics[width = \ffw]{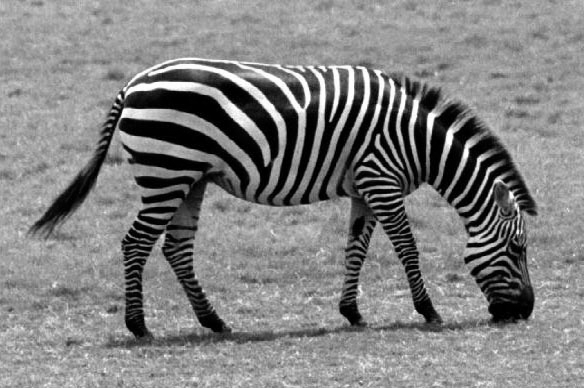}
\end{subfigure}
\caption{Denoising with Gaussian noise with mean zero and standard deviation 0.1 times image range. First three rows, from left to right: Data, TGV, CL, DIP, proposed, ground truth. Last two rows, from left to right: Data, TGV, DIP, proposed, ground truth.}
\label{fig_denoising_supplement}
\end{figure}

\subsection{Deconvolution}
Now we consider a deconvolution experiment. To this aim, given a convolution kernel $k\in\mathbb{R}^{(2s+1)\times (2s+1)}$, we define the forward operator as $A:\mathbb{R}^{N_x\times N_y}\rightarrow\mathbb{R}^{N_x\times N_y}$
\[ (Au)_{i,j} = \sum\limits_{i'=-s}^s \sum\limits_{j'=-s}^s k_{s+1+i',s+1+j'} u_{i-i',j-j'},\]
where we set $u_{i,j}=0$, whenever $(i,j)\notin \{1,2,...,N_x\}\times\{1,2,...,N_y\}$. The data fidelity terms are the same as in the case of denoising, i.e.,
\[ \Dc^1_{y}(z) = \frac{1}{2} \| z-y \|_2^2, \quad \Dc^2_{y} \equiv 0. \]
In the experiments, we use a Gaussian convolution kernel, with size $9\times 9$ for the \emph{Barbara} and \emph{mix} images, $13\times 13$ for the \emph{fish} image and $15\times 15$ for the \emph{zebra} image and standard deviation 0.25. The noise $\eta$ is Gaussian noise with standard deviation 0.025 times the range of $y^\dagger$ and mean zero. 
Here, we do not compare to DIP, since in \cite{deep_image_prior} deconvolution is not included in the experiments. For CL and TGV, the parameters were again chosen as suggested by the authors and for optimal visual results, respectively.
A summary of all results for deconvolution can be found in \cref{fig_deconv_supplement}.

\renewcommand\fw{3cm} 
\renewcommand\fwh{1.5cm} 

\renewcommand\ffw{3.75cm} 
\renewcommand\ffwh{1.875cm}

\begin{figure}
\centering
\begin{subfigure}{\textwidth}
\centering
\begin{subfigure}[t]{\fw}
\includegraphics[width = \fw]{images_gen_reg_deconvolution_barbara_crop_corrupted_closeups.png}
\includegraphics[width = \fwh]{images_gen_reg_deconvolution_barbara_crop_corrupted_detail_1.png}%
\includegraphics[width = \fwh]{images_gen_reg_deconvolution_barbara_crop_corrupted_detail_2.png}
\end{subfigure}%
\begin{subfigure}[t]{\fw}
\includegraphics[width = \fw]{images_tgv_deconvolution_barbara_crop_recon_closeups.png}
\includegraphics[width = \fwh]{images_tgv_deconvolution_barbara_crop_recon_detail_1.png}%
\includegraphics[width = \fwh]{images_tgv_deconvolution_barbara_crop_recon_detail_2.png}
\end{subfigure}%
\begin{subfigure}[t]{\fw}
\includegraphics[width = \fw]{images_convex_learning_deconvolution_barbara_crop_recon_closeups.png}
\includegraphics[width = \fwh]{images_convex_learning_deconvolution_barbara_crop_recon_detail_1.png}%
\includegraphics[width = \fwh]{images_convex_learning_deconvolution_barbara_crop_recon_detail_2.png}
\end{subfigure}%
\begin{subfigure}[t]{\fw}
\includegraphics[width = \fw]{images_gen_reg_deconvolution_barbara_crop_recon_closeups.png}
\includegraphics[width = \fwh]{images_gen_reg_deconvolution_barbara_crop_recon_detail_1.png}%
\includegraphics[width = \fwh]{images_gen_reg_deconvolution_barbara_crop_recon_detail_2.png}
\end{subfigure}%
\begin{subfigure}[t]{\fw}
\includegraphics[width = \fw]{images_gen_reg_deconvolution_barbara_crop_original_closeups.png}
\includegraphics[width = \fwh]{images_gen_reg_deconvolution_barbara_crop_original_detail_1.png}%
\includegraphics[width = \fwh]{images_gen_reg_deconvolution_barbara_crop_original_detail_2.png}
\end{subfigure}
\end{subfigure}

\begin{subfigure}{\textwidth}
\centering
\includegraphics[width = \fw]{images_gen_reg_deconvolution_cart_text_mix_corrupted.png}%
\includegraphics[width = \fw]{images_tgv_deconvolution_cart_text_mix_recon.png}%
\includegraphics[width = \fw]{images_convex_learning_deconvolution_cart_text_mix_recon.png}%
\includegraphics[width = \fw]{images_gen_reg_deconvolution_cart_text_mix_recon.png}%
\includegraphics[width = \fw]{images_gen_reg_deconvolution_cart_text_mix_original.png}
\end{subfigure}

\begin{subfigure}{\textwidth}
\centering
\begin{subfigure}[t]{\ffw}
\includegraphics[width = \ffw]{images_gen_reg_deconvolution_fish_corrupted_closeups.png}
\includegraphics[width = \ffwh]{images_gen_reg_deconvolution_fish_corrupted_detail_1.png}%
\includegraphics[width = \ffwh]{images_gen_reg_deconvolution_fish_corrupted_detail_2.png}
\end{subfigure}%
\begin{subfigure}[t]{\ffw}
\includegraphics[width = \ffw]{images_tgv_deconvolution_fish_recon_closeups.png}
\includegraphics[width = \ffwh]{images_tgv_deconvolution_fish_recon_detail_1.png}%
\includegraphics[width = \ffwh]{images_tgv_deconvolution_fish_recon_detail_2.png}
\end{subfigure}%
\begin{subfigure}[t]{\ffw}
\includegraphics[width = \ffw]{images_gen_reg_deconvolution_fish_recon_closeups.png}
\includegraphics[width = \ffwh]{images_gen_reg_deconvolution_fish_recon_detail_1.png}%
\includegraphics[width = \ffwh]{images_gen_reg_deconvolution_fish_recon_detail_2.png}
\end{subfigure}%
\begin{subfigure}[t]{\ffw}
\includegraphics[width = \ffw]{images_gen_reg_deconvolution_fish_original_closeups.png}
\includegraphics[width = \ffwh]{images_gen_reg_deconvolution_fish_original_detail_1.png}%
\includegraphics[width = \ffwh]{images_gen_reg_deconvolution_fish_original_detail_2.png}
\end{subfigure}
\end{subfigure}

\begin{subfigure}{\textwidth}
\centering
\includegraphics[width = \ffw]{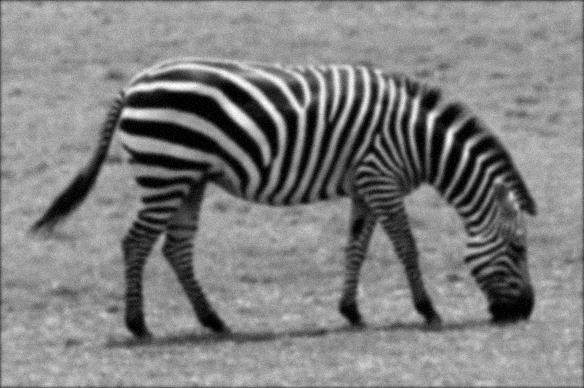}%
\includegraphics[width = \ffw]{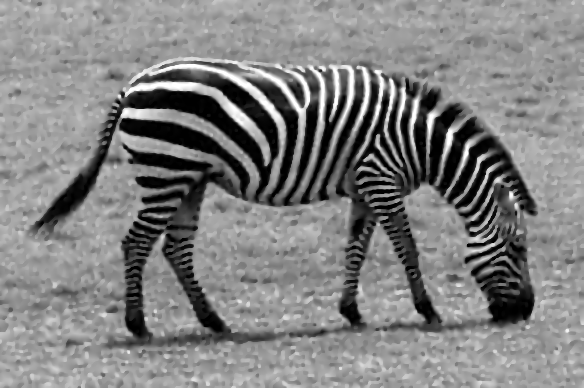}%
\includegraphics[width = \ffw]{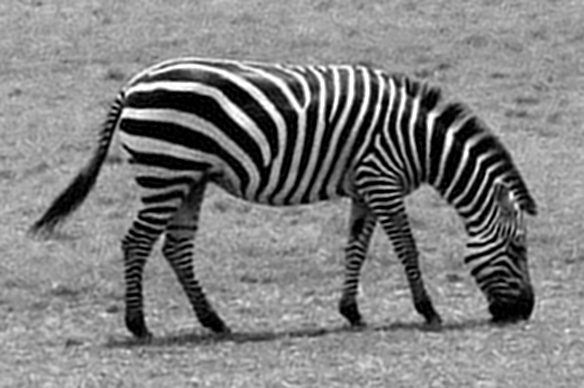}%
\includegraphics[width = \ffw]{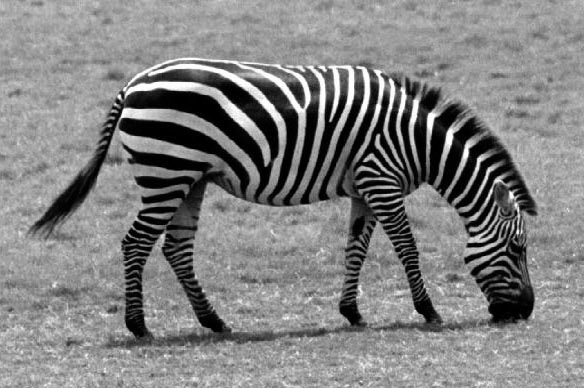}
\end{subfigure}
\caption{Deconvolution. First two rows from left to right: Data, TGV, CL, proposed, ground truth. Last two rows from left to right: Data, TGV, proposed, ground truth.}
\label{fig_deconv_supplement}
\end{figure}

\subsection{Super-resolution}\label{SM_section_super-resolution}
In the experiments of this section, our goal is to increase the resolution of a given image. Accordingly, the forward operator is a downsampling operator. For our experiments, we use averaging as a downsampling method. That is, given a scaling factor $s\in \mathbb{N}$, a divisor of $N_x$ and $N_y$, we define $A:\mathbb{R}^{N_x\times N_y}\rightarrow\mathbb{R}^{(N_x/s) \times (N_y/s)}$ as
\[ (Au)_{i,j} = \frac{1}{s^2}\sum\limits_{i',j'=1}^s u_{s(i-1)+i',s(j-1)+j'},\]
for $i=1,\ldots , N_x/s$, $j=1,\ldots , N_y/s$ and
\begin{equation*}
\Dc^1_{y} \equiv 0, \quad \Dc^2_{y} = \Ic_{\{y\}},
\end{equation*}
where $y$ is the given low resolution image and no noise is added to $y$. As a downsampling rate for the forward operator, we use $s=4$.  For the \emph{fish} and \emph{zebra} images, we used a subsampled version of the full-resolution image as data, where subsampling was done with the forward operator. With this, we aim to evaluate the ideal scenario that the forward model is correct. For the \emph{Barbara} image, we used the full-resolution image as data such that no bias from artificial subsampling (and no ground truth) is present.

For the experiments with the DIP, we used the network architecture corresponding to \cite[Figure 1]{deep_image_prior} and, in the code, we set the forward operator as above (default is downsampling with a Lanczos kernel) and the number of iterations to 2500. 
A summary of all results for super-resolution can be found in \cref{fig_supres_supplement}.

\begin{figure}
\centering
\begin{subfigure}{\textwidth}
\centering
\begin{subfigure}[t]{\textwidth}
\includegraphics[width = 0.334\textwidth, valign = t]{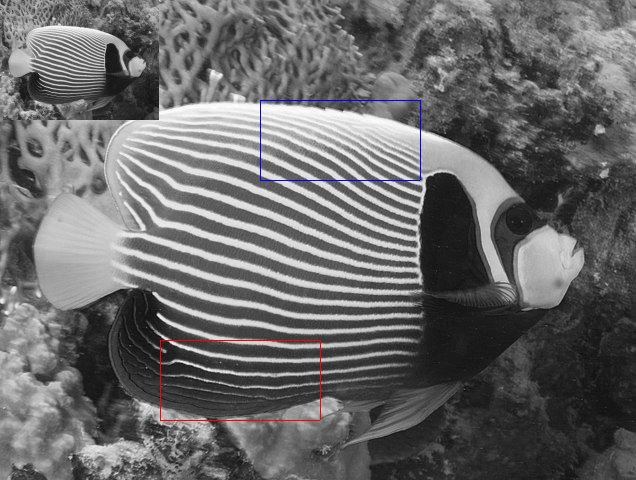}%
\includegraphics[width = 0.333\textwidth, valign = t]{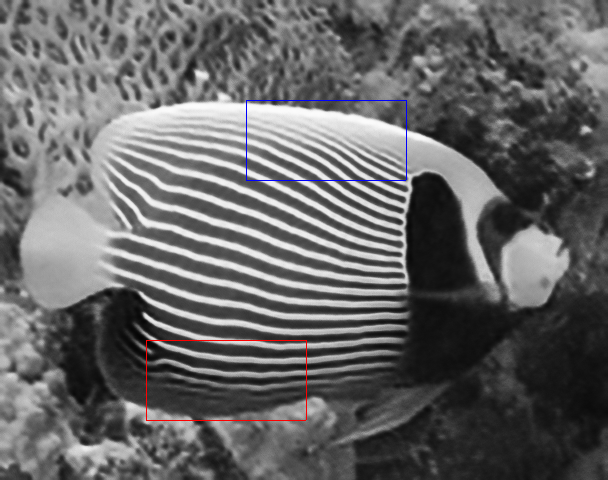}%
\includegraphics[width = 0.333\textwidth, valign = t]{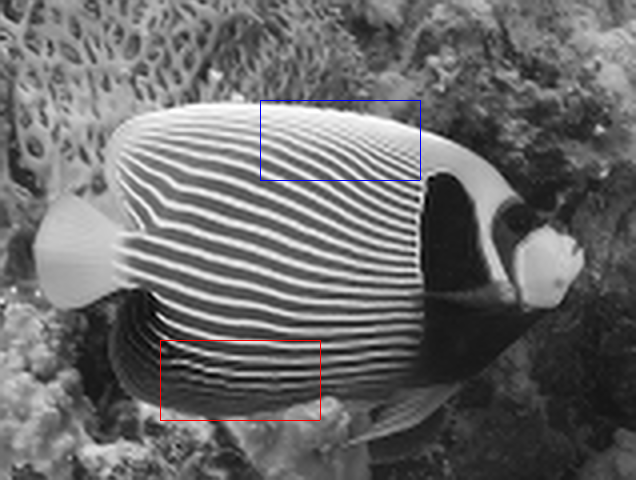}
\end{subfigure}
\begin{subfigure}[t]{0.334\textwidth}
\includegraphics[width = 0.5\textwidth]{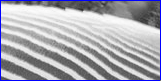}%
\includegraphics[width = 0.5\textwidth]{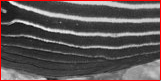}
\end{subfigure}%
\begin{subfigure}[t]{0.333\textwidth}
\includegraphics[width = 0.5\textwidth]{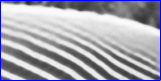}%
\includegraphics[width = 0.5\textwidth]{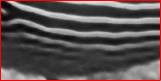}
\end{subfigure}%
\begin{subfigure}[t]{0.333\textwidth}
\includegraphics[width = 0.5\textwidth]{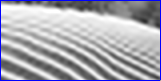}%
\includegraphics[width = 0.5\textwidth]{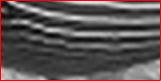}
\end{subfigure}
\end{subfigure}

\centering
\begin{subfigure}{\textwidth}
\includegraphics[width = 0.334\textwidth]{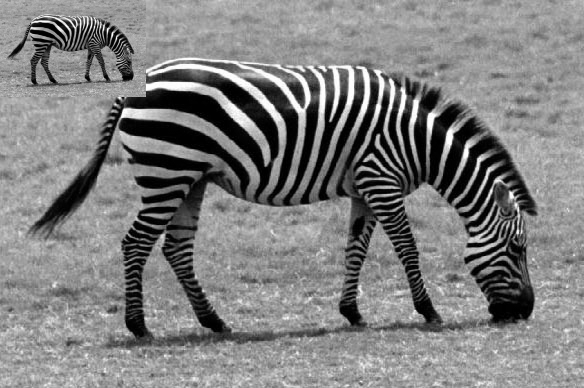}%
\includegraphics[width = 0.333\textwidth]{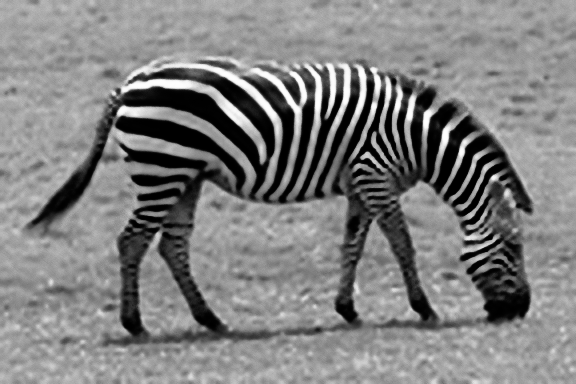}%
\includegraphics[width = 0.333\textwidth]{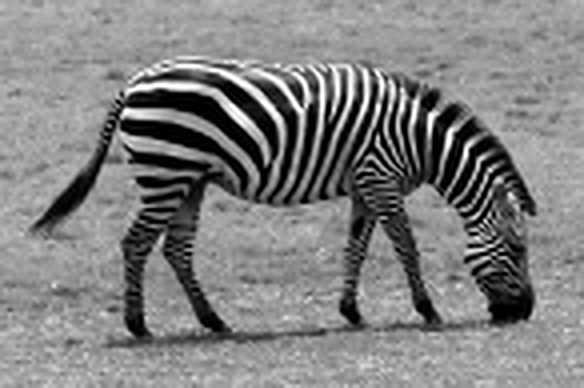}
\end{subfigure}

\centering
\begin{subfigure}{\textwidth}
\flushright
\includegraphics[width = 0.334\textwidth]{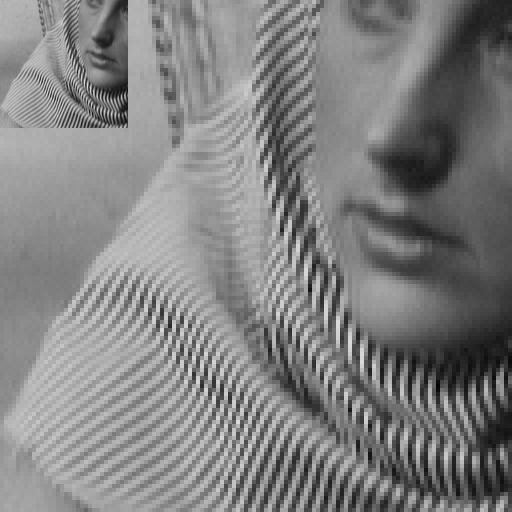}%
\includegraphics[width = 0.333\textwidth]{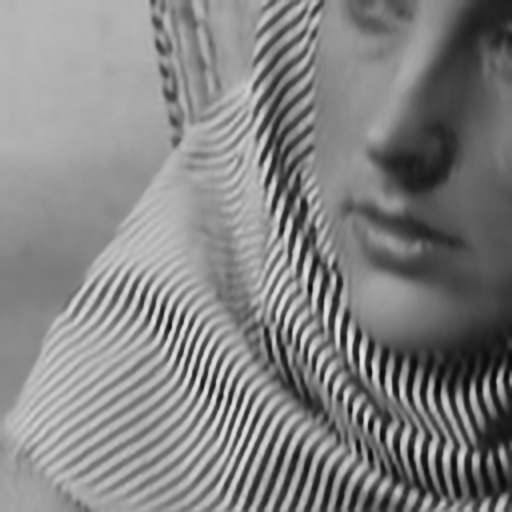}%
\includegraphics[width = 0.333\textwidth]{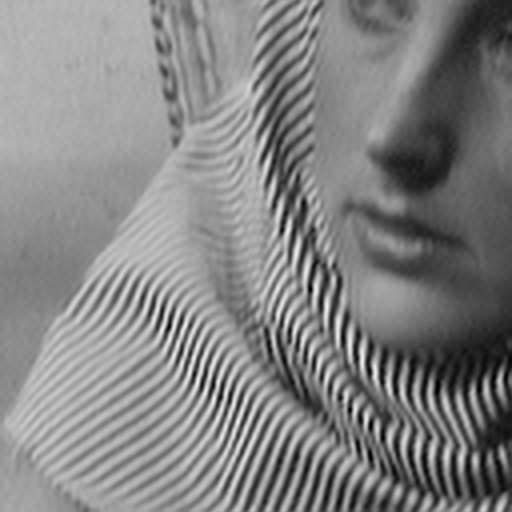}
\end{subfigure}
\caption{Super-resolution by a factor 4. From left to right: Original (low resolution in upper left corner), DIP, proposed. For the \emph{Barbara} image, no high resolution original is available and we show an upsampling by pixel repetition instead.}
\label{fig_supres_supplement}
\end{figure}

\subsection{JPEG decompression} \label{sm:sec_jpeg}
The last experiment we consider is JPEG decompression. Modeling the compression process, the forward operator $A$ is in this case a block-cosine transform, while the data fidelity functional is the convex indicator function of a set $D(y)$ containing all cosine coefficients which, when quantized with the given compression rate, would result in the same coefficients as stored in the given JPEG file $y$. That is, we set $\Dc^1_{y} \equiv 0 $ and $\Dc^2_{y} = \Ic_{D(y)}$. Again, we set the noise $\eta$ to zero in this case. 

For all JPEG decompression experiments with the DIP, we used the architecture as proposed by the authors corresponding to the JPEG experiment shown in \cite[Figure 3]{deep_image_prior}, and optimized the number of iterations for each image, see \cref{table_numiter_dip}. Note that here, not only the data fidelity used in DIP differs from ours, but also the forward operator does. In \cite{deep_image_prior}, the authors simply use the identity as a forward operator, whereas the proposed method makes use of the specific form of data corruption.
A summary of all results for JPEG decompression can be found in \cref{fig_jpeg_supplement}.

\renewcommand\fw{3cm} 
\renewcommand\fwh{1.5cm} 

\renewcommand\ffw{3.75cm} 
\renewcommand\ffwh{1.875cm}

\begin{figure}
\centering
\begin{subfigure}{\textwidth}
\centering
\begin{subfigure}[t]{\fw}
\includegraphics[width = \fw]{images_gen_reg_jpeg_barbara_crop_corrupted_closeups.png}
\includegraphics[width = \fwh]{images_gen_reg_jpeg_barbara_crop_corrupted_detail_1.png}%
\includegraphics[width = \fwh]{images_gen_reg_jpeg_barbara_crop_corrupted_detail_2.png}
\end{subfigure}%
\begin{subfigure}[t]{\fw}
\includegraphics[width = \fw]{images_ulyanov_jpeg_barbara_crop_recon_closeups.png}
\includegraphics[width = \fwh]{images_ulyanov_jpeg_barbara_crop_recon_detail_1.png}%
\includegraphics[width = \fwh]{images_ulyanov_jpeg_barbara_crop_recon_detail_2.png}
\end{subfigure}%
\begin{subfigure}[t]{\fw}
\includegraphics[width = \fw]{images_gen_reg_jpeg_barbara_crop_recon_closeups.png}
\includegraphics[width = \fwh]{images_gen_reg_jpeg_barbara_crop_recon_detail_1.png}%
\includegraphics[width = \fwh]{images_gen_reg_jpeg_barbara_crop_recon_detail_2.png}
\end{subfigure}%
\begin{subfigure}[t]{\fw}
\includegraphics[width = \fw]{images_gen_reg_jpeg_barbara_crop_original_closeups.png}
\includegraphics[width = \fwh]{images_gen_reg_jpeg_barbara_crop_original_detail_1.png}%
\includegraphics[width = \fwh]{images_gen_reg_jpeg_barbara_crop_original_detail_2.png}
\end{subfigure}
\end{subfigure}
\begin{subfigure}{\textwidth}
\centering
\includegraphics[width = \fw]{images_gen_reg_jpeg_cart_text_mix_corrupted.png}%
\includegraphics[width = \fw]{images_ulyanov_jpeg_cart_text_mix_recon.png}%
\includegraphics[width = \fw]{images_gen_reg_jpeg_cart_text_mix_recon.png}%
\includegraphics[width = \fw]{images_gen_reg_jpeg_cart_text_mix_original.png}
\end{subfigure}
\begin{subfigure}{\textwidth}
\centering
\includegraphics[width = \fw]{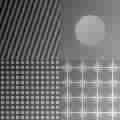}%
\includegraphics[width = \fw]{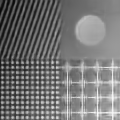}%
\includegraphics[width = \fw]{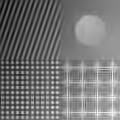}%
\includegraphics[width = \fw]{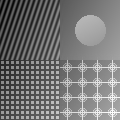}
\end{subfigure}
\caption{JPEG decompression. From left to right: Data, DIP, proposed, ground truth.}
\label{fig_jpeg_supplement}
\end{figure}

\printbibliography

\makeatletter\@input{xx.tex}\makeatother